\documentclass[12pt]{amsart}       
\usepackage{txfonts}
\usepackage{amssymb}
\usepackage{eucal}
\usepackage{graphicx}
\usepackage{amsmath}
\usepackage{amscd}
\usepackage[all]{xy}           
\usepackage{amsfonts,latexsym}
\usepackage{xspace}
\usepackage{epsfig}

\usepackage{epstopdf}
\usepackage{subfig}
\usepackage{enumitem}
\usepackage{amsmath,amssymb,amsthm}

\usepackage{tikz-cd}

\usepackage{float}
\usepackage{color}
\usepackage{fancybox}
\usepackage{colordvi}
\usepackage{multicol}
\usepackage{colordvi}
\usepackage{tikz}
\usepackage{wasysym}
\usepackage[active]{srcltx} 
\ifpdf
  \usepackage[colorlinks,final,backref=page,hyperindex]{hyperref}
\else
  \usepackage[colorlinks,final,backref=page,hyperindex,hypertex]{hyperref}
\fi

\usepackage{extarrows}

\newcommand{\nc}{\newcommand}
\newcommand{\delete}[1]{}

\nc{\mlabel}[1]{\label{#1}}  
\nc{\mcite}[1]{\cite{#1}}  
\nc{\mref}[1]{\ref{#1}}  
\nc{\meqref}[1]{\eqref{#1}}  
\nc{\mbibitem}[1]{\bibitem{#1}} 

\delete{
\nc{\mlabel}[1]{\label{#1}  
{\hfill \hspace{1cm}{\small\tt{{\ }\hfill(#1)}}}}
\nc{\mcite}[1]{\cite{#1}{\small{\tt{{\ }(#1)}}}}  
\nc{\mref}[1]{\ref{#1}{{\tt{{\ }(#1)}}}}  
\nc{\meqref}[1]{\eqref{#1}{{\tt{{\ }(#1)}}}}  
\nc{\mbibitem}[1]{\bibitem[\bf #1]{#1}} 
}

\makeatletter

\newcommand{\Rmnum}[1]{\expandafter\@slowromancap\romannumeral #1@}
\makeatother

\newtheorem{theorem}{Theorem}[section]
\newtheorem{prop}[theorem]{Proposition}
\newtheorem{lemma}[theorem]{Lemma}

\theoremstyle{definition}
\newtheorem{defn}[theorem]{Definition}
\newtheorem{prop-def}{Proposition-Definition}[section]

\newtheorem{remark}[theorem]{Remark}

\newtheorem{conjecture}[theorem]{Conjecture}

\newtheorem{tempex}[theorem]{Example}
\newtheorem{tempexs}[theorem]{Examples}
\newenvironment{exam}{\begin{tempex}\rm}{\end{tempex}}

\nc{\Irr}{\mathrm{Irr}}
\nc{\ncrbw}{\calr}  
\nc{\NS}{U_{NS}}
\nc{\FN}{F_{\mathrm Nij}}
\nc{\dfgen}{V} \nc{\dfrel}{R}
\nc{\dfgenb}{\vec{v}} \nc{\dfrelb}{\vec{r}}
\nc{\dfgene}{v} \nc{\dfrele}{r}
\nc{\dfop}{\odot}
\nc{\dfoa}{\dfop^{(1)}} \nc{\dfob}{\dfop^{(2)}}
\nc{\dfoc}{\dfop^{(3)}} \nc{\dfod}{\dfop^{(4)}}
\nc{\mapm}[1]{\lfloor\!|{#1}|\!\rfloor}
\nc{\cmapm}[1]{\frakC(#1)}
\nc{\red}{\mathrm{Red}}
\nc{\cm}{C}
\nc{\supp}{\mathrm{Supp}}
\nc{\lex}{\mathrm{lex}}

\nc{\disp}[1]{\displaystyle{#1}}
\nc{\bin}[2]{ (_{\stackrel{\scs{#1}}{\scs{#2}}})}  
\nc{\bs}{\bar{S}} \nc{\ep}{\epsilon}
\nc{\dbigcup}{\stackrel{\bullet}{\bigcup}}
\nc{\la}{\longrightarrow} \nc{\cprod}{\ast} \nc{\rar}{\rightarrow}
\nc{\dar}{\downarrow} \nc{\labeq}[1]{\stackrel{#1}{=}}
\nc{\dap}[1]{\downarrow \rlap{$\scriptstyle{#1}$}}
\nc{\uap}[1]{\uparrow \rlap{$\scriptstyle{#1}$}}
\nc{\defeq}{\stackrel{\rm def}{=}} \nc{\dis}[1]{\displaystyle{#1}}
\nc{\dotcup}{\ \displaystyle{\bigcup^\bullet}\ }
\nc{\sdotcup}{\tiny{ \displaystyle{\bigcup^\bullet}\ }}
\nc{\fe}{\'{e}}
\nc{\hcm}{\ \hat{,}\ } \nc{\hcirc}{\hat{\circ}}
\nc{\hts}{\hat{\shpr}} \nc{\lts}{\stackrel{\leftarrow}{\shpr}}
\nc{\denshpr}{\den{\shpr}}
\nc{\rts}{\stackrel{\rightarrow}{\shpr}} \nc{\lleft}{[}
\nc{\lright}{]} \nc{\uni}[1]{\tilde{#1}} \nc{\free}[1]{\bar{#1}}
\nc{\freea}[1]{\tilde{#1}} \nc{\freev}[1]{\hat{#1}}
\nc{\dt}[1]{\hat{#1}}
\nc{\wor}[1]{\check{#1}}
\nc{\intg}[1]{F_C(#1)}
\nc{\den}[1]{\check{#1}} \nc{\lrpa}{\wr} \nc{\mprod}{\pm}
\nc{\dprod}{\ast_P} \nc{\curlyl}{\left \{ \begin{array}{c} {} \\
{} \end{array}
    \right .  \!\!\!\!\!\!\!}
\nc{\curlyr}{ \!\!\!\!\!\!\!
    \left . \begin{array}{c} {} \\ {} \end{array}
    \right \} }
\nc{\longmid}{\left | \begin{array}{c} {} \\ {} \end{array}
    \right . \!\!\!\!\!\!\!}
\nc{\lin}{\call} \nc{\ot}{\otimes}
\nc{\ora}[1]{\stackrel{#1}{\rar}}
\nc{\ola}[1]{\stackrel{#1}{\la}}
\nc{\scs}[1]{\scriptstyle{#1}} \nc{\mrm}[1]{{\rm #1}}
\nc{\margin}[1]{\marginpar{\rm #1}}   
\nc{\dirlim}{\displaystyle{\lim_{\longrightarrow}}\,}
\nc{\invlim}{\displaystyle{\lim_{\longleftarrow}}\,}
\nc{\mvp}{\vspace{0.5cm}}
\nc{\mult}{m}       
\nc{\svp}{\vspace{2cm}} \nc{\vp}{\vspace{8cm}}
\nc{\proofbegin}{\noindent{\bf Proof: }}
\nc{\proofend}{$\blacksquare$ \vspace{0.5cm}}
\nc{\sha}{{\mbox{\cyr X}}}  
\nc{\ncsha}{{\mbox{\cyr X}^{\mathrm NC}}}
\newfont{\scyr}{wncyr10 scaled 550}
\nc{\ssha}{\mbox{\bf \scyr X}}
\nc{\ncshao}{{\mbox{\cyr X}^{\mathrm NC,\,0}}}
\nc{\shpr}{\diamond}    
\nc{\shprc}{\shpr_c}
\nc{\shpro}{\diamond^0}    
\nc{\shpru}{\check{\diamond}} \nc{\spr}{\cdot}
\nc{\catpr}{\diamond_l} \nc{\rcatpr}{\diamond_r}
\nc{\lapr}{\diamond_a} \nc{\lepr}{\diamond_e} \nc{\sprod}{\bullet}
\nc{\un}{u}                 
\nc{\vep}{\varepsilon} \nc{\labs}{\mid\!} \nc{\rabs}{\!\mid}
\nc{\hsha}{\widehat{\sha}} \nc{\psha}{\sha^{+}} \nc{\tsha}{\tilde{\sha}}
\nc{\lsha}{\stackrel{\leftarrow}{\sha}}
\nc{\rsha}{\stackrel{\rightarrow}{\sha}} \nc{\lc}{\lfloor}
\nc{\rc}{\rfloor} \nc{\sqmon}[1]{\langle #1\rangle}
\nc{\altx}{\Lambda} \nc{\vecT}{\vec{T}} \nc{\piword}{{\mathfrak P}}
\nc{\lbar}[1]{\overline{#1}}
\nc{\dep}{\mathrm{dep}}

\nc{\mmbox}[1]{\mbox{\ #1\ }}
\nc{\ayb}{\mrm{AYB}} \nc{\mayb}{\mrm{mAYB}} \nc{\cyb}{\mrm{cyb}}
\nc{\ann}{\mrm{ann}} \nc{\Aut}{\mrm{Aut}} \nc{\cabqr}{\mrm{CABQR
}} \nc{\can}{\mrm{can}} \nc{\colim}{\mrm{colim}}
\nc{\Cont}{\mrm{Cont}} \nc{\rchar}{\mrm{char}}
\nc{\cok}{\mrm{coker}} \nc{\dtf}{{R-{\rm tf}}} \nc{\dtor}{{R-{\rm
tor}}}

\nc{\Div}{{\mrm Div}} \nc{\End}{\mrm{End}} \nc{\Ext}{\mrm{Ext}}
\nc{\FG}{\mrm{FG}} \nc{\Fil}{\mrm{Fil}} \nc{\Frob}{\mrm{Frob}}
\nc{\Gal}{\mrm{Gal}} \nc{\GL}{\mrm{GL}} \nc{\Hom}{\mrm{Hom}}
\nc{\hsr}{\mrm{H}} \nc{\hpol}{\mrm{HP}} \nc{\id}{\mrm{id}} \nc{\Id}{\mathrm{Id}}  \nc{\ID}{\mathrm{ID}}
\nc{\im}{\mrm{im}} \nc{\incl}{\mrm{incl}} \nc{\Loday}{\mrm{ABQR}\
} \nc{\length}{\mrm{length}} \nc{\LR}{\mrm{LR}} \nc{\mchar}{\rm
char} \nc{\pmchar}{\partial\mchar} \nc{\map}{\mrm{Map}}
\nc{\MS}{\mrm{MS}} \nc{\OS}{\mrm{OS}} \nc{\NC}{\mrm{NC}}
\nc{\rba}{\rm{Rota-Baxter algebra}\xspace}
\nc{\rbas}{\rm{Rota-Baxter algebras}\xspace}
\nc{\rbw}{\ncrbw}
\nc{\rbws}{\rm{RBWs}\xspace}
\nc{\rbadj}{\rm{RB}\xspace}
\nc{\mpart}{\mrm{part}} \nc{\ql}{{\QQ_\ell}} \nc{\qp}{{\QQ_p}}
\nc{\rank}{\mrm{rank}} \nc{\rcot}{\mrm{cot}} \nc{\rdef}{\mrm{def}}
\nc{\rdiv}{{\rm div}} \nc{\rtf}{{\rm tf}} \nc{\rtor}{{\rm tor}}
\nc{\res}{\mrm{res}} \nc{\SL}{\mrm{SL}} \nc{\Spec}{\mrm{Spec}}
\nc{\tor}{\mrm{tor}} \nc{\Tr}{\mrm{Tr}}
\nc{\mtr}{\mrm{tr}}

\nc{\ab}{\mathbf{Ab}} \nc{\Alg}{\mathbf{Alg}}
\nc{\Bax}{\mathbf{CRB}} \nc{\Algo}{\mathbf{Alg}^0}
\nc{\cRB}{\mathbf{CRB}} \nc{\cRBo}{\mathbf{CRB}^0}
\nc{\RBo}{\mathbf{RB}^0} \nc{\BRB}{\mathbf{RB}}
\nc{\Dend}{\mathbf{DD}} \nc{\bfk}{{\bf k}} \nc{\bfone}{{\bf 1}}
\nc{\base}[1]{{a_{#1}}} \nc{\Cat}{\mathbf{Cat}}
 \nc{\DN}{\mathbf{DN}}
\nc{\NA}{\mathbf{NA}}
\nc{\SDN}{\mathbf{SDN}}
\nc{\Diff}{\mathbf{Diff}} \nc{\gap}{\marginpar{\bf
Incomplete}\noindent{\bf Incomplete!!}
    \svp}
\nc{\FMod}{\mathbf{FMod}} \nc{\Int}{\mathbf{Int}}
\nc{\Mon}{\mathbf{Mon}}
\nc{\RB}{\mathbf{RB}} \nc{\remarks}{\noindent{\bf Remarks: }}
\nc{\Rep}{\mathbf{Rep}} \nc{\Rings}{\mathbf{Rings}}
\nc{\Sets}{\mathbf{Sets}} \nc{\bfx}{\mathbf{x}}
\nc{\BA}{{\Bbb A}} \nc{\CC}{{\Bbb C}} \nc{\DD}{{\Bbb D}}
\nc{\EE}{{\Bbb E}} \nc{\FF}{{\Bbb F}} \nc{\GG}{{\Bbb G}}
\nc{\HH}{{\Bbb H}} \nc{\LL}{{\Bbb L}} \nc{\NN}{{\Bbb N}}
\nc{\QQ}{{\Bbb Q}} \nc{\RR}{{\Bbb R}} \nc{\TT}{{\Bbb T}}
\nc{\VV}{{\Bbb V}} \nc{\ZZ}{{\Bbb Z}}

\nc{\cala}{{\mathcal A}} \nc{\calb}{{\mathcal B}}
\nc{\calc}{{\mathcal C}}
\nc{\cald}{{\mathcal D}} \nc{\cale}{{\mathcal E}}
\nc{\calf}{{\mathcal F}} \nc{\calg}{{\mathcal G}}
\nc{\calh}{{\mathcal H}} \nc{\cali}{{\mathcal I}}
\nc{\calj}{{\mathcal J}} \nc{\call}{{\mathcal L}}
\nc{\calm}{{\mathcal M}} \nc{\caln}{{\mathcal N}}
\nc{\calo}{{\mathcal O}} \nc{\calp}{{\mathcal P}}
\nc{\calr}{{\mathcal R}} \nc{\cals}{{\mathcal S}} \nc{\calt}{{\mathcal T}}
\nc{\calw}{{\mathcal W}} \nc{\calx}{{\mathcal X}} \nc{\caly}{{\mathcal Y}} \nc{\calz}{{\mathcal Z}}
\nc{\CA}{\mathcal{A}}

\nc{\frakA}{{\mathfrak A}}
\nc{\fraka}{{\mathfrak a}}
\nc{\frakB}{{\mathfrak B}}
\nc{\frakb}{{\mathfrak b}}
\nc{\frakC}{{\mathfrak C}}
\nc{\frakd}{{\mathfrak d}}
\nc{\frakF}{{\mathfrak F}}
\nc{\frakg}{{\mathfrak g}}
\nc{\frakm}{{\mathfrak m}}
\nc{\frakM}{{\mathfrak M}}
\nc{\frakMo}{{\mathfrak M}^0}
\nc{\frakP}{{\mathfrak P}}
\nc{\frakp}{{\mathfrak p}}
\nc{\frakS}{{\mathfrak S}}
\nc{\frakSo}{{\mathfrak S}^0}
\nc{\fraks}{{\mathfrak s}}
\nc{\os}{\overline{\fraks}}
\nc{\frakT}{{\mathfrak T}}
\nc{\frakTo}{{\mathfrak T}^0}
\nc{\oT}{\overline{T}}
\nc{\frakX}{{\mathfrak X}}
\nc{\frakXo}{{\mathfrak X}^0}
\nc{\frakx}{{\mathbf x}}
\nc{\frakTx}{\frakT}      
\nc{\frakTa}{\frakT^a}        
\nc{\frakTxo}{\frakTx^0}   
\nc{\caltao}{\calt^{a,0}}   
\nc{\ox}{\overline{\frakx}} \nc{\fraky}{{\mathfrak y}}
\nc{\frakz}{{\mathfrak z}} \nc{\oX}{\overline{X}} \font\cyr=wncyr10

\nc{\tred}[1]{\textcolor{red}{#1}} \nc{\tgreen}[1]{\textcolor{green}{#1}}
\nc{\tblue}[1]{\textcolor{blue}{#1}} \nc{\tpurple}[1]{\textcolor{purple}{#1}}

\nc{\li}[1]{\tpurple{\underline{Li:}#1 }}
\nc{\liadd}[1]{\tpurple{#1}}
\nc{\xing}[1]{\tblue{\underline{Xing:}#1 }}
\nc{\YZ}[1]{\tred{\underline{Yaozhou:} #1}}

\nc{\deleted}[1]{\delete{#1}}
\nc{\astarrow}{\overset{\raisebox{-2pt}{{\scriptsize $\ast$}}}{\rightarrow}}\nc{\tvarrow}[3]{#1\overset{(t,v)}{\longrightarrow}_{#3} #2}

\nc{\fg}[2]{\tilde{\mathfrak{g}}_{#1}^{(#2)}}
\nc{\sg}[2]{g_{#1,#2}^{(k)}}
\nc{\De}[2]{\Delta^{#1}(#2)}
\nc{\bott}{b}
\nc{\down}[2]{{\rm down}_{#1}(#2)}    \nc{\bdo}[2]{{\rm d}_{#1}(#2)}
\nc{\Down}[3]{{\rm down}_{#1}^{#2}(#3)}  \nc{\bDo}[3]{{\rm d}_{#1}^{#2}(#3)}
\nc{\uup}[2]{{\rm up}_{#1}(#2)}
\nc{\dkmu}{\Delta^k(\mu)} \nc{\dkl}{\Delta^k(\lambda)}
\nc{\dkga}{\Delta^k(\gamma)}
\nc{\pk}{{\rm P}^k}
\nc{\pkl}{{\rm P}_{\ell}^{k}}
\nc{\tpkl}{\tilde{{\rm P}}_{\ell}^{k}}
\nc{\spkl}{\hat{\rm P}_{\ell}^{k}}
\nc{\lmx}{\mu} \nc{\lsum}{\nu}
\nc{\ddh}{\bar{z}}

\begin{document}
\title[$K$-$k$-Schur alternating conjecture]{Weighted $K$-$k$-Schur functions and their application to the  $K$-$k$-Schur alternating conjecture}
%
\author{Yaozhou Fang}
\address{School of Mathematics and Statistics, Lanzhou University
Lanzhou, 730000, China
}
\email{fangyzh21@lzu.edu.cn}

\author{Xing Gao$^*$}\thanks{*Corresponding author}
\address{School of Mathematics and Statistics, Lanzhou University,
Lanzhou, 730000, China; Gansu Provincial Research Center for Basic Disciplines of Mathematics and Statistics, Lanzhou, 730070, China
}
\email{gaoxing@lzu.edu.cn}

\date{\today}
\begin{abstract}
We introduce the new concept of weighted $K$-$k$-Schur functions---a novel family within the broader class of Katalan functions---that unifies and extends both $K$-$k$-Schur functions and closed $k$-Schur Katalan functions. This new notion exhibits a fundamental alternating property under certain conditions on the indexed $k$-bounded partitions.
As a central application, we resolve the $K$-$k$-Schur alternating conjecture---posed by Blasiak, Morse, and Seelinger in 2022---for a wide class of $k$-bounded partitions, including all strictly decreasing $k$-bounded partitions. Our results shed new light on the combinatorial structure of $K$-theoretic symmetric functions.
\end{abstract}

\makeatletter
\@namedef{subjclassname@2020}{\textup{2020} Mathematics Subject Classification}
\makeatother
\subjclass[2020]{
05E05, 
05E10, 
14N15, 
}

\keywords{Katalan function, (weighted) $K$-$k$-Schur function, closed $k$-Schur Katalan function, positivity}

\maketitle

\tableofcontents

\setcounter{section}{0}

\allowdisplaybreaks

\section{Introduction}
In this paper, we introduce the notion of weighted $K$-$k$-Schur functions, a family of Katalan functions and a unifying framework that encompasses both $K$-$k$-Schur functions and closed $k$-Schur Katalan functions. This notion and a successful application of Katalan formula reveal a fundamental alternating property and lead to a positive resolution of the $K$-$k$-Schur alternating conjecture~\cite{BMS} for a large class of $k$-bounded partitions.
Some usual notations are listed in Subsection~\mref{ss:outline}.

\subsection{Catalan functions and Katalan functions}
{\em Catalan functions}, introduced by Chen-Haiman~\mcite{Ch} and Panyushev~\mcite{Pan}, form a broad and unifying family of symmetric functions that extend both $k$-Schur functions and parabolic Hall-Littlewood polynomials. Arising naturally in the study of Euler characteristics of vector bundles on flag varieties~\cite{Br, Ch, Pan, SW}, they are defined using Demazure operators---central tools in the geometric analysis of cohomology on such spaces. Catalan functions have since been further developed in~\cite{BMP24, BMPS19, BMPS20}, and continue to reveal rich connections to Schubert calculus, Gromov-Witten theory, and the representation theory of quantum groups~\mcite{Hagl}.

Building on this foundation of Catalan functions, {\em Katalan functions} were introduced by Blasiak, Morse, and Seelinger~\cite{BMS} as a class of inhomogeneous symmetric functions that may be viewed as inhomogeneous analogues of Catalan functions. Katalan functions unify and extend several important families of symmetric functions, including the $K$-$k$-Schur functions $g_\lambda^{(k)}$ and the closed $k$-Schur Katalan functions $\fg{\lambda}{k}$. They arise naturally in the study of Schubert calculus in $K$-theory, the geometry of the affine Grassmannian, and the combinatorics of $k$-Schur functions. In essence, the theory of Katalan functions provides a powerful and flexible framework that links combinatorics, geometry, and representation theory, while illuminating several open problems.
For example, based on the Katalan formula of the $K$-$k$-Schur function $g_\lambda^{(k)}$, Blasiak et al. showed that $g_\lambda^{(k)}$ can degenerate to $g_{\lambda}$---the dual stable Grothendieck polynomials studied in~\cite{Las01,Le}---when $k$ is sufficiently large. They also proved that $g_\lambda^{(k)}$ is $\{g_{\lambda}\}_{\lambda\in\pk}$ alternating positive (see~\cite[Corollaries~2.9 and2.10]{BMS}).

\subsection{$K$-$k$-Schur functions}
The {\em $K$-$k$-Schur functions} $g_\lambda^{(k)}$ arose in the development of $K$-theoretic Schubert calculus, a $K$-theoretic extension of classical and affine Schubert calculus.
Schubert calculus itself traces back to Hermann Schubert's foundational work in 19th-century enumerative geometry~\mcite{Sch}, and gained mathematical rigor and prominence through Hilbert's 15th problem~\mcite{Hil}. Since then, it has played a central role across geometry, topology, combinatorics, and representation theory, particularly through the study of the cohomology ring structure constants of Grassmannians $Gr(n,k)$.
In Schubert calculus, Schubert classes in $H^*(GL(n,k))$ are represented precisely by Schur functions.

To refine Schubert calculus, Lapointe, Lascoux, and Morse~\cite{LLM} introduced the $k$-Schur functions in connection with a deeper understanding of Macdonald positivity conjecture~\cite{LLMSSZ}. Later, Lam~\cite{Lam} advanced the study of affine Schubert calculus by showing that the Schubert classes in the homology of the affine Grassmannian ${\rm Gr}$ correspond to $k$-Schur functions under a Hopf algebra isomorphism originally studied by Bott~\cite{Bot}, mapping $H_*({\rm Gr})$ to the subalgebra of symmetric functions $\Lambda_{(k)} := \mathbb{Q}[h_1, \ldots, h_k]$.

Extending this perspective to $K$-theory, Lam, Schilling, and Shimozono~\cite{LSS} proposed a $K$-theoretic version of affine Schubert calculus. They demonstrated that the $K$-homology $K_*({\rm Gr})$ of the affine Grassmannian ${\rm Gr}$ also admits a Hopf algebra isomorphism to $\Lambda_{(k)}$. Within this framework, they introduced the $K$-$k$-Schur functions $g_\lambda^{(k)}$ as representatives for the classes of ideal sheaves of the boundaries of Schubert varieties. Further, Lam et al.~\cite{LLMS} showed that the class of the structure sheaf of a Schubert variety maps to the closed $K$-$k$-Schur function $\sum_{w_\mu \leq w_\lambda} g_\mu^{(k)}$ under the same isomorphism. This development significantly enriched the study of $K$-theoretic Schubert calculus and deepened the connection between geometry, representation theory, and symmetric functions.

\subsection{Closed $k$-Schur Katalan functions}
Quite recently, Ikeda, Iwao, and Maeno~\mcite{IIM} proposed a $K$-theoretic analogue of the isomorphism constructed by Lam and Shimozono~\cite{LS}, denoted by $\Phi$, between localizations of the quantum $K$-theory ring of the complete flag variety $F\ell_{k+1}$ and the localized $K$-homology of the affine Grassmannian:
\begin{equation*}
\begin{tikzcd}[column sep=scriptsize, row sep=scriptsize]
\mathcal{QH}^*(F\ell_{k+1})_{{\rm loc}} \arrow{rr}{\text{\cite{LS}}}[swap]{\backsimeq} \arrow{dd}[swap]{\text{$K$-theoretic}}{\text{version}} &&H_*({\rm Gr})_{{\rm loc}} \arrow{dd}[swap]{\text{$K$-theoretic}}{\text{version}}\\
\\
\mathcal{QK}(F\ell_{k+1})_{{\rm loc}} \arrow{rr}{\text{\cite{IIM}}~\Phi}[swap]{\backsimeq} &&K_*({\rm Gr})_{{\rm loc}}
\end{tikzcd}
\end{equation*}
Here, $\mathcal{QH}^*(F\ell_{k+1})$ is the quantum cohomology ring, and $\mathcal{QK}(F\ell_{k+1})$ is the quantum $K$-theory ring of $F\ell_{k+1}$, previously studied by Givental and Lee~\mcite{GL}.
In their foundational work~\cite{LS}, Lam and Shimozono established that $k$-Schur functions form the core of the image of the quantum Schubert polynomials, introduced by Fomin, Gelfand, and Postnikov~\cite{FGP}. This result bridged two seemingly disparate structures---$k$-Schur functions and quantum Schubert polynomials---and enabled the transfer of combinatorial and algebraic properties between them.
In the $K$-theoretic setting, Lenart and Maeno~\mcite{LeMa} introduced a natural analogue of quantum Schubert polynomials, called quantum Grothendieck polynomials $\mathfrak{B}_w^Q \in \mathcal{QK}(F\ell{k+1})$, for permutations $w$ in the symmetric group $S_{k+1}$. To better understand the images $\Phi(\mathfrak{B}_w^Q)$ under the $K$-theoretic isomorphism $\Phi$, Ikeda, Iwao, and Maeno~\mcite{IIM} introduced a family of symmetric functions $\tilde{g}_w$, closely related to closed $K$-$k$-Schur functions.

To study $\tilde{g}_w$, Blasiak, Morse, and Seelinger~\mcite{BMS} proposed a distinguished subfamily of Katalan functions, called {\em closed $k$-Schur Katalan functions} $\fg{\lambda}{k}$ (see Definition~\mref{defn:cKataFunc}), and formulated a series of six conjectures regarding their structures and properties.
\begin{conjecture}(\cite[Conjecture~2.12]{BMS})
Let $\lambda\in\pkl$ and $w\in S_{k+1}$. Then
\begin{enumerate}[label={(\alph*)}]
\item $\tilde{g}_{w} = \fg{\nu}{k}$, where $\nu := \theta(w)^{\omega_k}$, \mlabel{it:conja}

\item
\[
\Phi(\mathfrak{B}_w^Q) = \frac{\fg{\nu}{k}}{\prod_{d\in {\rm Des}(w)}g_{R_d}},
\]\mlabel{it:conjb}

\item (alternating dual Pieri rule) the coefficients in $G_{1^m}^\perp \fg{\lambda}{k} = \sum_{\mu}c_{\lambda\mu}\fg{\mu}{k}$ satisfy $(-1)^{|\lambda|-|\mu|-m}c_{\lambda\mu}\in\ZZ_{\geq 0}$, \mlabel{it:conjc}

\item ($k$-branching) the coefficients in $\fg{\lambda}{k} = \sum_{\mu}a_{\lambda\mu}\fg{\mu}{k+1}$ satisfy $(-1)^{|\lambda|-|\mu|}a_{\lambda\mu}\in\ZZ_{\geq 0}$, \mlabel{it:conjd}

\item ($K$-$k$-Schur alternating) the coefficients in
$
\fg{\lambda}{k} = \sum_{\mu\in{\rm P}}b_{\lambda\mu}g_\mu^{(k)}
$
satisfy $(-1)^{|\lambda|-|\mu|} b_{\lambda\mu}\in\ZZ_{\geq0}$, \mlabel{it:conje}

\item ($k$-rectangle property) for $d\in[k]$, $g_{R_d}\fg{\lambda}{k} = \fg{\lambda\cup R_d}{k}$. \mlabel{it:conjf}
\end{enumerate}
See~\cite[Subsection~2.4]{BMS} for detailed explanations of the notations.
\mlabel{conj:aim}
\end{conjecture}

Among these six conjectures, several logical relationships hold. Assuming~\mref{it:conja}, part~\mref{it:conjf} follows from the work of Takigiku~\cite{Ta19}, while~\mref{it:conjc} implies~\mref{it:conjd} by~\cite[Proposition2.16~(c)]{BMS}. Parts~\mref{it:conja} and~\mref{it:conjb} of Conjecture~\mref{conj:aim} were proved by Ikeda, Iwao, and Naito~\mcite{IIN}. Independently, we gave an alternative proof of part~\mref{it:conja} using a different approach~\cite{FGG}, and also established parts~\mref{it:conjc} and~\mref{it:conjd} for certain families of $k$-bounded partitions~\cite{FG}.

Part~\mref{it:conjd} is motivated by the branching rule for $k$-Schur functions~\cite[Theorem~2.6]{BMPS19}, and is used in~\cite{BMS} to
show that $\fg{\lambda}{k}$ is $\{g_{\lambda}\}_{\lambda\in\pk}$ alternating positive, using the fact that $\fg{\lambda}{k} = g_{\lambda}$ if $k$ is sufficiently large.
For variants of part~\mref{it:conjf}, Lapointe and Morse~\cite[Theorem~40]{LM} established the $k$-rectangle property for $k$-Schur functions. Takigiku~\cite{Ta19} prove this property for closed $K$-$k$-Schur functions, while also demonstrating that $K$-$k$-Schur functions do not possess the $k$-rectangle property.

\subsection{Main results and outline of the paper} \mlabel{ss:outline}
In this paper,
we prove that Conjecture~\mref{conj:aim}~\mref{it:conje} is valid for the following class of partitions,
\begin{equation}
\spkl:=\{ \lambda\in\pkl\mid \lambda_{x-1}>\lambda_{x} \,\text{ if }\, k-\lambda_x+x<\ell\}.
\mlabel{eq:spkl}
\end{equation}
We interpret that $\lambda_0 = \infty$ so that $\lambda_0>\lambda_1$, and use the convention $\hat{\rm P}_0^k:={\rm P}_0^k:=\emptyset$. It is obvious that $\spkl$ includes the $k$-bounded strictly decreasing partitions with length $\ell$, which satisfy $\lambda\in\pkl$ and $\lambda_1>\lambda_2>\cdots$.

Inspired by the method in~\cite{BMPS19,BMPS20,BMS,FG, FGG}, we propose the concept of weighted $K$-$k$-Schur function (see Definition~\mref{def:semiKk}), which generalize simultaneously $K$-$k$-Schur functions $g_\lambda^{(k)}$ and closed $k$-Schur Katalan functions $\fg{\lambda}{k}$, and prove
the following main results.

\begin{theorem}
Let $\lambda\in\pkl$ and $z\in[0,\bott_\lambda]$ with $\lambda_1>\cdots >\lambda_z$. Suppose
further $\lambda_{z}>\lambda_{z+1}$ when $z\neq \bott_\lambda$.
Then
\begin{equation}
\sg{\lambda}{z+1} = \sum_{\mu\in{\rm P}}b_{\lambda\mu}g_\mu^{(k)}, \quad
\text{where }\, (-1)^{|\lambda|-|\mu|}b_{\lambda\mu}\in\ZZ_{\geq0}.
\mlabel{eq:finalonlyneed}
\end{equation}
\mlabel{thm:aim}
\end{theorem}

As an immediate application, we can prove Conjecture~\mref{conj:aim}~\mref{it:conje} for $\lambda\in \spkl$.

\begin{theorem}
Let $\lambda\in\spkl$. Then the coefficients $b_{\lambda\mu}$ in
\begin{equation*}
\fg{\lambda}{k} = \sum_{\mu\in{\rm P}}b_{\lambda\mu}g_\mu^{(k)}
\end{equation*}
satisfy $(-1)^{|\lambda|-|\mu|} b_{\lambda\mu}\in\ZZ_{\geq0}$.
\mlabel{thm:aim1}
\end{theorem}

\noindent{\bf Outline of the paper.}
In Section~\mref{sec:equiv}, we begin by recalling key background concepts and results on Katalan functions and their combinatorial interpretations, with particular emphasis on the Mirror Lemma. We then review the Katalan formula of the $K$-$k$-Schur function and the definition of the closed $k$-Schur Katalan function, and introduce the $K$-$k$-Schur function of weight $z$, denoted by $\sg{\lambda}{z}$ (Definition~\mref{def:semiKk}). We show that $\sg{\lambda}{z}$ coincides with the $K$-$k$-Schur function when $z = 1$, and with the closed $k$-Schur Katalan function when $z$ is sufficiently large (Proposition~\mref{prop:threeg}).

Section~\mref{sec:lower} focuses on expressing $K$-$k$-Schur function $\sg{\lambda}{z+1}$ of weight $z+1$ as a linear combination of $K$-$k$-Schur functions $\sg{\mu}{z}$ of lower weight $z$, as established in Propositions~\mref{prop:bigd},~\mref{prop:smalld} and~\mref{prop:ind}.
Finally, in Section~\mref{sec:comproof}, we present the detailed proofs of Theorems~\mref{thm:aim} and~\mref{thm:aim1},
which relies on Proposition~\mref{prop:ind} from Section~\mref{sec:lower}.

\vskip 0.1in

{\bf Notations.}
We fix some notations used throughout the paper.
\begin{enumerate}

\item For two integers $a,b$, let $[a,b]:=\{ a, a+1,\ldots,b \}$. We use the convention that $[a,b]=\emptyset$ if $a>b$. In particularly, denote $[b]:=[1,b]$ for any integer $b$.

\item Denote by $\epsilon_{i}$ the sequence of length $\ell$ with a $1$ in position $i$ and $0$'s elsewhere. Set $\varepsilon_{\alpha}:=\epsilon_{i}-\epsilon_{j}$ with $\alpha=(i,j)$.

\item
Let ${\rm P}$ be the set of partitions and ${\rm P}^{k}$ (resp. ${\rm P}_{\ell}^{k}$) be the set of $k$-bounded partitions (resp.  $k$-bounded partitions of length $\ell\in\mathbb{Z}_{\geq0}$) for $k\in \ZZ_{\geq 0}$.

\item For $\gamma := (\gamma_1,\ldots,\gamma_\ell)\in\ZZ^\ell$ with $\ell\in\ZZ_{\geq 0}$, define the size of $\gamma$ by $|\gamma|:= \sum_{i=1}^\ell\gamma_i$.


\item For $d\in \ZZ_{\geq 0}$, denote by $h_d$ the elementary complete homogeneous symmetric function.
\end{enumerate}

\section{Katalan functions and weighted $K$-$k$-Schur functions}
\mlabel{sec:equiv}
In this section, we first review the basic notation and fundamental properties of Katalan functions, with particular focus on the Mirror Lemma. We then recall the Katalan formula for the $K$-$k$-Schur function and the definition of the closed $k$-Schur Katalan function. Finally, we introduce a new subfamily of Katalan functions, called weighted $K$-$k$-Schur functions, which simultaneously generalize both $K$-$k$-Schur functions and closed $k$-Schur Katalan functions.

\subsection{Katalan functions}\mlabel{ss:Katalan}
This subsection is devoted to recalling the definition of Katalan functions and some key facts of them used in later arguments.

For $\ell\in\ZZ_{\geq 0}$, consider the set $$\Delta_{\ell}^{+}:=\{(i,j) \mid 1\le i<j\le\ell \},$$
with the convention $\Delta_{0}^{+}:=\Delta_{1}^{+} := \emptyset$.
A {\em root ideal} $\Psi$ is an upper order ideal of the poset $\Delta_{\ell}^{+}$ with partial order given by $(a,b)\le(c,d)$ if $a\ge c$ and $b\le d$. Namely, if $(a,b)\in\Psi$, then $(c,d)\in\Psi$ for $(c,d)\in\Delta_{\ell}^{+}$ with $(c,d)\ge(a,b)$.
Let $m\in\mathbb{Z}$ and $r\in\mathbb{Z}_{\geq 0}$.
Let $k_{m}^{(r)}$ be an inhomogeneous version of $h_m$ given by
\begin{equation*}
    k_{m}^{(r)} :=\sum_{i=0}^{m}\binom{r+i-1}{i}h_{m-i}.
\end{equation*}
Notice that $h_{m}$ is the monomial of highest degree of $k_{m}^{(r)}$.
Here we use the convention that
$$k_{m}^{(r)}=0\,\text{ for }\,  m<0 \,\text{ and }\, k_{m}^{(0)}=h_{m}.$$
For $\gamma\in\mathbb{Z}^{\ell}$, define $$g_{\gamma}:=\text{det}(k_{\gamma_{i}+j-i}^{(i-1)})_{1\le i,j\le\ell}.$$
If $\gamma$ is a partition, $g_{\gamma}$ is known as the {\em dual stable Grothendieck polynomial} studied in~\cite{Las01,Le}.
The following is the concept of Katalan function.

\begin{defn}(\cite[Definition~2.1]{BMS})
For a root ideal $\Psi\subseteq\Delta_{\ell}^{+}$, a multiset $M$ with $\text{supp}(M)\subseteq [\ell]$ and $\gamma\in\mathbb{Z}^{\ell}$, define the {\em Katalan function}
\begin{equation*}
K(\Psi;M;\gamma):=\prod_{z\in M}(1-L_{z})\prod_{(i,j)\in\Psi}(1-R_{ij})^{-1}g_{\gamma},
\end{equation*}
where $L_{z}$ is a {\em lowering operator} acting on $g_{\gamma}$ by $L_{z}g_{\gamma}:=g_{\gamma-\epsilon_{z}}$ and $R_{ij}$ is a {\em raising operator} acting on $g_{\gamma}$ by $R_{ij}g_{\gamma}:=g_{\gamma+\epsilon_{i}-\epsilon_{j}}$.
\mlabel{defn:KataFunc}
\end{defn}

%
%
%

Given a multiset $M$ on $[\ell]$ and $a\in M$, denote by $m_{M}(a)$ the number of occurrences of $a$ in $M$. A Katalan function $K(\Psi;M;\gamma)$ can be represented by the $\ell\times\ell$ grid of boxes (labelled by matrix-style coordinates), with the boxes of $\Psi$ shaded, $m_{M}(a)$ $\bullet$'s in column $a$ and the entries of $\gamma$ on the diagonal of these boxes. Let us expose an example for illustration.

\begin{exam}
For $\ell=4$, $\Psi=\{ (1,3),(1,4),(2,4) \}$, $M=\{2,3,4,4\}$ and $\gamma=(3,2,1,3)$, we represent the related Katalan function by
\begin{equation*}
K(\Psi;M;\gamma)=
\begin{tikzpicture}[scale=.4,line width=0.5pt,baseline=(a.base)]
\draw (0,2) rectangle (1,1);\node at(0.5,1.5){\scriptsize\( 3 \)};
\draw (1,2) rectangle (2,1);\node at(1.5,1.5){\scriptsize\( \bullet \)};
\filldraw[red,draw=black] (2,2) rectangle (3,1);\node at(2.5,1.5){\scriptsize\( \bullet \)};
\filldraw[red,draw=black] (3,2) rectangle (4,1);\node at(3.5,1.5){\scriptsize\( \bullet \)};
\draw (0,1) rectangle (1,0);
\draw (1,1) rectangle (2,0);\node at(1.5,0.5){\scriptsize\( 2 \)};
\draw (2,1) rectangle (3,0);
\filldraw[red,draw=black] (3,1) rectangle (4,0);\node at(3.5,0.5){\scriptsize\( \bullet \)};
\node (a) [align=center] {\\[-5pt] };
\draw (0,0) rectangle (1,-1);
\draw (1,0) rectangle (2,-1);
\draw (2,0) rectangle (3,-1);\node at(2.5,-0.5){\scriptsize\( 1 \)};
\draw (3,0) rectangle (4,-1);
\draw (0,-1) rectangle (1,-2);
\draw (1,-1) rectangle (2,-2);
\draw (2,-1) rectangle (3,-2);
\draw (3,-1) rectangle (4,-2);\node at(3.5,-1.5){\scriptsize\( 3 \)};
\end{tikzpicture}
\, .
\end{equation*}
\hspace*{\fill}$\square$
\end{exam}

In order to conveniently describe the variation of Katalan functions, we recall the following concepts in~\cite{BMPS19,BMS}.
Let $\Psi\subseteq\Delta_{\ell}^{+}$ be a root ideal and $x\in[\ell]$.
\begin{enumerate}
\item A root $\alpha\in\Psi$ is called {\em removable} if $\Psi\setminus \alpha$ is still a root ideal; a root $\alpha\in\Delta_{\ell}^{+}\setminus\Psi$ is called {\em addable} if $\Psi\cup\alpha$ is also a root ideal.

\item If there exists $j\in[\ell]$ such that $(x,j)$ is removable in $\Psi$, we denote $\text{down}_{\Psi}(x):=j$, else we say that $\text{down}_{\Psi}(x)$ is undefined; if there exists $j\in[\ell]$ such that $(j,x)$ is removable in $\Psi$, we denote $\text{up}_{\Psi}(x):=j$, else we say that $\text{up}_{\Psi}(x)$ is undefined.

\item The {\em bounce graph} of $\Psi$ is a graph on the vertex set $[\ell]$ with edges $(x,\text{down}_{\Psi}(x))$ for each $x\in[\ell]$ such that $\text{down}_{\Psi}(x)$ is defined. The bounce
    graph of $\Psi$ is a disjoint union of paths called {\em bounce paths} of $\Psi$.

\item Denote by $\text{bot}_{\Psi}(x)$ (resp. $\text{top}_{\Psi}(x)$) the maximum (resp. minimum) vertex, as a number in $[\ell]$, in the unique bounce path of $\Psi$ containing $x$.

\item For $a,b\in[\ell]$ in the same bounce path of $\Psi$ with $a\leq b$, define
\begin{align*}
\text{path}_{\Psi}(a,b):=(a,\text{down}_{\Psi}(a),
\text{down}_{\Psi}^{2}(a),\ldots,b).
\end{align*}
\end{enumerate}

\begin{defn}
A root ideal $\Psi\subseteq\Delta_{\ell}^{+}$ is said to have
\begin{enumerate}
\item a {\em wall} in rows $r, r+1$ if rows $r, r+1$ have the same length;

\item a {\em ceiling} in columns $c, c+1$ if columns $c, c+1$ have the same length;

\item a {\em mirror} in rows $r, r+1$ if $\Psi$ has removable roots $(r,c)$ and $(r+1,c+1)$ for some $c\in[r+2,\ell-1]$;

\item the {\em bottom} in row $\bott$ if $(\bott,p)\in\Psi$ for some $p\in [\ell]$ and $(\bott+1,q)\notin\Psi$ for all $q\in [\ell]$.
\end{enumerate}
\mlabel{def:com}
\end{defn}

Now we recall some essential properties of Katalan functions learned in~\cite{BMS}, especially the Mirror Lemma.

\begin{lemma}(\cite[p.18]{BMS})
For a root ideal $\Psi\subseteq\Delta_{\ell}^{+}$, a multiset $M$ with $\text{supp}(M)\subseteq [\ell]$, $\gamma\in\mathbb{Z}^{\ell}$ and $z\in[\ell]$,
$$L_z K(\Psi;M;\gamma)=K(\Psi;M;\gamma-\epsilon_z).$$
\mlabel{lem:LK}
\end{lemma}

\begin{lemma}(\cite[Proposition~3.9]{BMS})
Let $\Psi\subseteq\Delta_{\ell}^{+}$ be a root ideal, $M$ a multiset with \textup{supp}$(M)\subseteq [\ell]$ and $\gamma\in\mathbb{Z}^{\ell}$. Then
\begin{enumerate}
\item for any removable root $\beta$ of $\Psi$,
\begin{equation*}
K(\Psi;M;\gamma)=K(\Psi\setminus\beta;M;\gamma)+K(\Psi;M;\gamma+\varepsilon_{\beta});
\end{equation*}
\mlabel{it:relk1}

\item for any addable root $\alpha$ of $\Psi$,
\begin{equation*}
K(\Psi;M;\gamma)=K(\Psi\cup\alpha;M;\gamma)
-K(\Psi\cup\alpha;M;\gamma+\varepsilon_{\alpha});
\end{equation*}
\mlabel{it:relk2}

\item for any $m\in M$,
\begin{equation*}
K(\Psi;M;\gamma)=K(\Psi;M\setminus m;\gamma)-K(\Psi;M\setminus m;\gamma-\epsilon_{m}),
\end{equation*}
where $M\setminus m$ means removing only one element $m$ from $M$;
\mlabel{it:relk3}

\item for any $m\in[\ell]$,
\begin{equation*}
K(\Psi;M;\gamma)=K(\Psi;M\sqcup m;\gamma)+K(\Psi;M;\gamma-\epsilon_{m}).
\end{equation*}
\mlabel{it:relk4}
\end{enumerate}
\mlabel{lem:relk}
\end{lemma}

\begin{lemma}({\bf Mirror Lemma}, \cite[Lemma~4.6]{BMS})
Let $\Psi\subseteq\Delta_{\ell}^{+}$ be a root ideal, $M$ a multiset with \textup{supp}$(M)\subseteq [\ell]$ and $\gamma\in\mathbb{Z}^{\ell}$. Let $1\le y\le z<\ell$ be indices in a same bounce path  of $\Psi$ satisfying
\begin{enumerate}
\item $\Psi$ has a ceiling in columns $y,y+1$;
\item $\Psi$ has a mirror in rows $x,x+1$ for each $x\in\textup{path}_{\Psi}(y,\textup{up}_{\Psi}(z))$;
\item $\Psi$ has a wall in rows $z,z+1$;
\item $\gamma_{x}=\gamma_{x+1}$ for all $x\in\textup{path}_{\Psi}(y,\textup{up}_{\Psi}(z))$;
\item $\gamma_{z}+1=\gamma_{z+1}$;
\item $m_{M}(x)+1=m_{M}(x+1)$ for all $x\in\textup{path}_{\Psi}(\textup{down}_{\Psi}(y),z)$.
\end{enumerate}
If $m_{M}(y)+1=m_{M}(y+1)$, then $K(\Psi;M;\gamma)=0$; if $m_{M}(y)=m_{M}(y+1)$, then $K(\Psi;M;\gamma)=K(\Psi;M;\gamma-\epsilon_{z+1})$.
\mlabel{lem:mirr1}
\end{lemma}

To make it clear, the Mirror Lemma can be exhibited by combinations as the following example.

\begin{exam}
By Lemma~\mref{lem:mirr1}, the following Katalan functions satisfy
\begin{equation*}
K(\Psi;M_1;\gamma):=
\begin{tikzpicture}[scale=.3,line width=0.5pt,baseline=(a.base)]
\draw (0,7) rectangle (1,6);\node at(0.5,6.5){\tiny\( 7 \)};
\draw (1,7) rectangle (2,6);
\filldraw[red,draw=black] (2,7) rectangle (3,6);\node at(2.5,6.5){\tiny\( \bullet \)};
\filldraw[red,draw=black] (3,7) rectangle (4,6);\node at(3.5,6.5){\tiny\( \bullet \)};
\filldraw[red,draw=black] (4,7) rectangle (5,6);\node at(4.5,6.5){\tiny\( \bullet \)};
\filldraw[red,draw=black] (5,7) rectangle (6,6);\node at(5.5,6.5){\tiny\( \bullet \)};
\filldraw[red,draw=black] (6,7) rectangle (7,6);\node at(6.5,6.5){\tiny\( \bullet \)};
\filldraw[red,draw=black] (7,7) rectangle (8,6);\node at(7.5,6.5){\tiny\( \bullet \)};
\filldraw[red,draw=black] (8,7) rectangle (9,6);\node at(8.5,6.5){\tiny\( \bullet \)};
\filldraw[red,draw=black] (9,7) rectangle (10,6);\node at(9.5,6.5){\tiny\( \bullet \)};
\filldraw[red,draw=black] (10,7) rectangle (11,6);\node at(10.5,6.5){\tiny\( \bullet \)};
\filldraw[red,draw=black] (11,7) rectangle (12,6);\node at(11.5,6.5){\tiny\( \bullet \)};
\filldraw[red,draw=black] (12,7) rectangle (13,6);\node at(12.5,6.5){\tiny\( \bullet \)};
\filldraw[red,draw=black] (13,7) rectangle (14,6);\node at(13.5,6.5){\tiny\( \bullet \)};
\filldraw[red,draw=black] (14,7) rectangle (15,6);\node at(14.5,6.5){\tiny\( \bullet \)};
\draw (0,6) rectangle (1,5);
\draw (1,6) rectangle (2,5);\node at(1.5,5.5){\tiny\( 7 \)};
\draw (2,6) rectangle (3,5);
\filldraw[yellow,draw=black] (3,6) rectangle (4,5);
\filldraw[yellow,draw=black] (4,6) rectangle (5,5);\node at(4.5,5.5){\tiny\( \bullet \)};
\filldraw[red,draw=black] (5,6) rectangle (6,5);\node at(5.5,5.5){\tiny\( \bullet \)};
\filldraw[red,draw=black] (6,6) rectangle (7,5);\node at(6.5,5.5){\tiny\( \bullet \)};
\filldraw[red,draw=black] (7,6) rectangle (8,5);\node at(7.5,5.5){\tiny\( \bullet \)};
\filldraw[red,draw=black] (8,6) rectangle (9,5);\node at(8.5,5.5){\tiny\( \bullet \)};
\filldraw[red,draw=black] (9,6) rectangle (10,5);\node at(9.5,5.5){\tiny\( \bullet \)};
\filldraw[red,draw=black] (10,6) rectangle (11,5);\node at(10.5,5.5){\tiny\( \bullet \)};
\filldraw[red,draw=black] (11,6) rectangle (12,5);\node at(11.5,5.5){\tiny\( \bullet \)};
\filldraw[red,draw=black] (12,6) rectangle (13,5);\node at(12.5,5.5){\tiny\( \bullet \)};
\filldraw[red,draw=black] (13,6) rectangle (14,5);\node at(13.5,5.5){\tiny\( \bullet \)};
\filldraw[red,draw=black] (14,6) rectangle (15,5);\node at(14.5,5.5){\tiny\( \bullet \)};
\draw (0,5) rectangle (1,4);
\draw (1,5) rectangle (2,4);
\draw (2,5) rectangle (3,4);\node at(2.5,4.5){\tiny\( 6 \)};
\draw (3,5) rectangle (4,4);
\draw (4,5) rectangle (5,4);
\filldraw[red,draw=black] (5,5) rectangle (6,4);\node at(5.5,4.5){\tiny\( \bullet \)};
\filldraw[red,draw=black] (6,5) rectangle (7,4);\node at(6.5,4.5){\tiny\( \bullet \)};
\filldraw[red,draw=black] (7,5) rectangle (8,4);\node at(7.5,4.5){\tiny\( \bullet \)};
\filldraw[red,draw=black] (8,5) rectangle (9,4);\node at(8.5,4.5){\tiny\( \bullet \)};
\filldraw[red,draw=black] (9,5) rectangle (10,4);\node at(9.5,4.5){\tiny\( \bullet \)};
\filldraw[red,draw=black] (10,5) rectangle (11,4);\node at(10.5,4.5){\tiny\( \bullet \)};
\filldraw[red,draw=black] (11,5) rectangle (12,4);\node at(11.5,4.5){\tiny\( \bullet \)};
\filldraw[red,draw=black] (12,5) rectangle (13,4);\node at(12.5,4.5){\tiny\( \bullet \)};
\filldraw[red,draw=black] (13,5) rectangle (14,4);\node at(13.5,4.5){\tiny\( \bullet \)};
\filldraw[red,draw=black] (14,5) rectangle (15,4);\node at(14.5,4.5){\tiny\( \bullet \)};
\draw (0,4) rectangle (1,3);
\draw (1,4) rectangle (2,3);
\draw (2,4) rectangle (3,3);
\draw (3,4) rectangle (4,3);\node at(3.5,3.5){\tiny\( \textcolor{red}{7} \)};
\draw (4,4) rectangle (5,3);
\filldraw[green,draw=black] (5,4) rectangle (6,3);\node at(5.5,3.5){\tiny\( \bullet \)};
\filldraw[red,draw=black] (6,4) rectangle (7,3);\node at(6.5,3.5){\tiny\( \bullet \)};
\filldraw[red,draw=black] (7,4) rectangle (8,3);\node at(7.5,3.5){\tiny\( \bullet \)};
\filldraw[red,draw=black] (8,4) rectangle (9,3);\node at(8.5,3.5){\tiny\( \bullet \)};
\filldraw[red,draw=black] (9,4) rectangle (10,3);\node at(9.5,3.5){\tiny\( \bullet \)};
\filldraw[red,draw=black] (10,4) rectangle (11,3);\node at(10.5,3.5){\tiny\( \bullet \)};
\filldraw[red,draw=black] (11,4) rectangle (12,3);\node at(11.5,3.5){\tiny\( \bullet \)};
\filldraw[red,draw=black] (12,4) rectangle (13,3);\node at(12.5,3.5){\tiny\( \bullet \)};
\filldraw[red,draw=black] (13,4) rectangle (14,3);\node at(13.5,3.5){\tiny\( \bullet \)};
\filldraw[red,draw=black] (14,4) rectangle (15,3);\node at(14.5,3.5){\tiny\( \bullet \)};
\draw (0,3) rectangle (1,2);
\draw (1,3) rectangle (2,2);
\draw (2,3) rectangle (3,2);
\draw (3,3) rectangle (4,2);
\draw (4,3) rectangle (5,2);\node at(4.5,2.5){\tiny\( \textcolor{red}{7} \)};
\draw (5,3) rectangle (6,2);
\filldraw[green,draw=black] (6,3) rectangle (7,2);\node at(6.5,2.5){\tiny\( \bullet \)};
\filldraw[red,draw=black] (7,3) rectangle (8,2);\node at(7.5,2.5){\tiny\( \bullet \)};
\filldraw[red,draw=black] (8,3) rectangle (9,2);\node at(8.5,2.5){\tiny\( \bullet \)};
\filldraw[red,draw=black] (9,3) rectangle (10,2);\node at(9.5,2.5){\tiny\( \bullet \)};
\filldraw[red,draw=black] (10,3) rectangle (11,2);\node at(10.5,2.5){\tiny\( \bullet \)};
\filldraw[red,draw=black] (11,3) rectangle (12,2);\node at(11.5,2.5){\tiny\( \bullet \)};
\filldraw[red,draw=black] (12,3) rectangle (13,2);\node at(12.5,2.5){\tiny\( \bullet \)};
\filldraw[red,draw=black] (13,3) rectangle (14,2);\node at(13.5,2.5){\tiny\( \bullet \)};
\filldraw[red,draw=black] (14,3) rectangle (15,2);\node at(14.5,2.5){\tiny\( \bullet \)};
\draw (0,2) rectangle (1,1);
\draw (1,2) rectangle (2,1);
\draw (2,2) rectangle (3,1);
\draw (3,2) rectangle (4,1);
\draw (4,2) rectangle (5,1);
\draw (5,2) rectangle (6,1);\node at(5.5,1.5){\tiny\( \textcolor{red}{5} \)};
\draw (6,2) rectangle (7,1);
\draw (7,2) rectangle (8,1);
\draw (8,2) rectangle (9,1);
\filldraw[green,draw=black] (9,2) rectangle (10,1);\node at(9.5,1.5){\tiny\( \bullet \)};
\filldraw[red,draw=black] (10,2) rectangle (11,1);\node at(10.5,1.5){\tiny\( \bullet \)};
\filldraw[red,draw=black] (11,2) rectangle (12,1);\node at(11.5,1.5){\tiny\( \bullet \)};
\filldraw[red,draw=black] (12,2) rectangle (13,1);\node at(12.5,1.5){\tiny\( \bullet \)};
\filldraw[red,draw=black] (13,2) rectangle (14,1);\node at(13.5,1.5){\tiny\( \bullet \)};
\filldraw[red,draw=black] (14,2) rectangle (15,1);\node at(14.5,1.5){\tiny\( \bullet \)};
\draw (0,1) rectangle (1,0);
\draw (1,1) rectangle (2,0);
\draw (2,1) rectangle (3,0);
\draw (3,1) rectangle (4,0);
\draw (4,1) rectangle (5,0);
\draw (5,1) rectangle (6,0);
\draw (6,1) rectangle (7,0);\node at(6.5,0.5){\tiny\( \textcolor{red}{5} \)};
\draw (7,1) rectangle (8,0);
\draw (8,1) rectangle (9,0);
\draw (9,1) rectangle (10,0);
\filldraw[green,draw=black] (10,1) rectangle (11,0);\node at(10.5,0.5){\tiny\( \bullet \)};
\filldraw[red,draw=black] (11,1) rectangle (12,0);\node at(11.5,0.5){\tiny\( \bullet \)};
\filldraw[red,draw=black] (12,1) rectangle (13,0);\node at(12.5,0.5){\tiny\( \bullet \)};
\filldraw[red,draw=black] (13,1) rectangle (14,0);\node at(13.5,0.5){\tiny\( \bullet \)};
\filldraw[red,draw=black] (14,1) rectangle (15,0);\node at(14.5,0.5){\tiny\( \bullet \)};
\draw (0,0) rectangle (1,-1);
\draw (1,0) rectangle (2,-1);
\draw (2,0) rectangle (3,-1);
\draw (3,0) rectangle (4,-1);
\draw (4,0) rectangle (5,-1);
\draw (5,0) rectangle (6,-1);
\draw (6,0) rectangle (7,-1);
\draw (7,0) rectangle (8,-1);\node at(7.5,-0.5){\tiny\( 4 \)};
\draw (8,0) rectangle (9,-1);
\draw (9,0) rectangle (10,-1);
\draw (10,0) rectangle (11,-1);
\draw (11,0) rectangle (12,-1);
\filldraw[red,draw=black] (12,0) rectangle (13,-1);\node at(12.5,-0.5){\tiny\( \bullet \)};
\filldraw[red,draw=black] (13,0) rectangle (14,-1);\node at(13.5,-0.5){\tiny\( \bullet \)};
\filldraw[red,draw=black] (14,0) rectangle (15,-1);\node at(14.5,-0.5){\tiny\( \bullet \)};
\draw (0,-1) rectangle (1,-2);
\draw (1,-1) rectangle (2,-2);
\draw (2,-1) rectangle (3,-2);
\draw (3,-1) rectangle (4,-2);
\draw (4,-1) rectangle (5,-2);
\draw (5,-1) rectangle (6,-2);
\draw (6,-1) rectangle (7,-2);
\draw (7,-1) rectangle (8,-2);
\draw (8,-1) rectangle (9,-2);\node at(8.5,-1.5){\tiny\( 4 \)};
\draw (9,-1) rectangle (10,-2);
\draw (10,-1) rectangle (11,-2);
\draw (11,-1) rectangle (12,-2);
\draw (12,-1) rectangle (13,-2);
\filldraw[red,draw=black] (13,-1) rectangle (14,-2);\node at(13.5,-1.5){\tiny\( \bullet \)};
\filldraw[red,draw=black] (14,-1) rectangle (15,-2);\node at(14.5,-1.5){\tiny\( \bullet \)};
\draw (0,-2) rectangle (1,-3);
\draw (1,-2) rectangle (2,-3);
\draw (2,-2) rectangle (3,-3);
\draw (3,-2) rectangle (4,-3);
\draw (4,-2) rectangle (5,-3);
\draw (5,-2) rectangle (6,-3);
\draw (6,-2) rectangle (7,-3);
\draw (7,-2) rectangle (8,-3);
\draw (8,-2) rectangle (9,-3);
\draw (9,-2) rectangle (10,-3);\node at(9.5,-2.5){\tiny\( \textcolor{red}{4} \)};
\draw (10,-2) rectangle (11,-3);
\draw (11,-2) rectangle (12,-3);
\draw (12,-2) rectangle (13,-3);
\draw (13,-2) rectangle (14,-3);
\filldraw[blue!50,draw=black] (14,-2) rectangle (15,-3);\node at(14.5,-2.5){\tiny\( \bullet \)};
\draw (0,-3) rectangle (1,-4);
\draw (1,-3) rectangle (2,-4);
\draw (2,-3) rectangle (3,-4);
\draw (3,-3) rectangle (4,-4);
\draw (4,-3) rectangle (5,-4);
\draw (5,-3) rectangle (6,-4);
\draw (6,-3) rectangle (7,-4);
\draw (7,-3) rectangle (8,-4);
\draw (8,-3) rectangle (9,-4);
\draw (9,-3) rectangle (10,-4);
\draw (10,-3) rectangle (11,-4);\node at(10.5,-3.5){\tiny\( \textcolor{red}{5} \)};
\draw (11,-3) rectangle (12,-4);
\draw (12,-3) rectangle (13,-4);
\draw (13,-3) rectangle (14,-4);
\filldraw[blue!50,draw=black] (14,-3) rectangle (15,-4);\node at(14.5,-3.5){\tiny\( \bullet \)};
\draw (0,-4) rectangle (1,-5);
\draw (1,-4) rectangle (2,-5);
\draw (2,-4) rectangle (3,-5);
\draw (3,-4) rectangle (4,-5);
\draw (4,-4) rectangle (5,-5);
\draw (5,-4) rectangle (6,-5);
\draw (6,-4) rectangle (7,-5);
\draw (7,-4) rectangle (8,-5);
\draw (8,-4) rectangle (9,-5);
\draw (9,-4) rectangle (10,-5);
\draw (10,-4) rectangle (11,-5);
\draw (11,-4) rectangle (12,-5);\node at(11.5,-4.5){\tiny\( 3 \)};
\draw (12,-4) rectangle (13,-5);
\draw (13,-4) rectangle (14,-5);
\draw (14,-4) rectangle (15,-5);
\draw (0,-5) rectangle (1,-6);
\draw (1,-5) rectangle (2,-6);
\draw (2,-5) rectangle (3,-6);
\draw (3,-5) rectangle (4,-6);
\draw (4,-5) rectangle (5,-6);
\draw (5,-5) rectangle (6,-6);
\draw (6,-5) rectangle (7,-6);
\draw (7,-5) rectangle (8,-6);
\draw (8,-5) rectangle (9,-6);
\draw (9,-5) rectangle (10,-6);
\draw (10,-5) rectangle (11,-6);
\draw (11,-5) rectangle (12,-6);
\draw (12,-5) rectangle (13,-6);\node at(12.5,-5.5){\tiny\( 2 \)};
\draw (13,-5) rectangle (14,-6);
\draw (14,-5) rectangle (15,-6);
\draw (0,-6) rectangle (1,-7);
\draw (1,-6) rectangle (2,-7);
\draw (2,-6) rectangle (3,-7);
\draw (3,-6) rectangle (4,-7);
\draw (4,-6) rectangle (5,-7);
\draw (5,-6) rectangle (6,-7);
\draw (6,-6) rectangle (7,-7);
\draw (7,-6) rectangle (8,-7);
\draw (8,-6) rectangle (9,-7);
\draw (9,-6) rectangle (10,-7);
\draw (10,-6) rectangle (11,-7);
\draw (11,-6) rectangle (12,-7);
\draw (12,-6) rectangle (13,-7);
\draw (13,-6) rectangle (14,-7);\node at(13.5,-6.5){\tiny\( 1 \)};
\draw (14,-6) rectangle (15,-7);
\draw (0,-7) rectangle (1,-8);
\draw (1,-7) rectangle (2,-8);
\draw (2,-7) rectangle (3,-8);
\draw (3,-7) rectangle (4,-8);
\draw (4,-7) rectangle (5,-8);
\draw (5,-7) rectangle (6,-8);
\draw (6,-7) rectangle (7,-8);
\draw (7,-7) rectangle (8,-8);
\draw (8,-7) rectangle (9,-8);
\draw (9,-7) rectangle (10,-8);
\draw (10,-7) rectangle (11,-8);
\draw (11,-7) rectangle (12,-8);
\draw (12,-7) rectangle (13,-8);
\draw (13,-7) rectangle (14,-8);
\draw (14,-7) rectangle (15,-8);\node at(14.5,-7.5){\tiny\( 1 \)};
%
\draw[purple,line width=0.8pt] (6,1.5)--(9,1.5);
\draw[purple,line width=0.8pt] (5.5,2)--(5.5,3);
\draw[purple,line width=0.8pt] (4,3.5)--(5,3.5);
\draw[purple,line width=0.8pt] (3.5,4)--(3.5,5);
%
\draw[black,line width=0.8pt,->] (15,-2.5)--(15.8,-2.5);
\node at (16.1,-2.5){ \tiny{\textbf{$z$}} };
\draw[black,line width=0.8pt,->] (15,1.5)--(15.8,1.5);
\node at (17.2,1.5){ \tiny{\textbf{${\rm up}_{\Psi}(z)$}} };
\draw[black,line width=0.8pt,->] (15,3.5)--(15.8,3.5);
\node at (16.1,3.5){ \tiny{\textbf{$y$}} };
\draw[black,line width=0.8pt,->] (3.5,7)--(3.5,7.8);
\node at (3.5,8.2){ \tiny{\textbf{$y$}} };
\draw[black,line width=0.8pt,->] (5.5,7)--(5.5,7.8);
\node at (6,8.2){ \tiny{\textbf{${\rm down}_{\Psi}(y)$}} };
\draw[black,line width=0.8pt,->] (9.5,7)--(9.5,7.8);
\node at (9.5,8.2){ \tiny{\textbf{$z$}} };
\end{tikzpicture}
=0,
\end{equation*}

\begin{align*}
K(\Psi;M_2;\gamma):=
\begin{tikzpicture}[scale=.3,line width=0.5pt,baseline=(a.base)]
\draw (0,7) rectangle (1,6);\node at(0.5,6.5){\tiny\( 7 \)};
\draw (1,7) rectangle (2,6);
\filldraw[red,draw=black] (2,7) rectangle (3,6);\node at(2.5,6.5){\tiny\( \bullet \)};
\filldraw[red,draw=black] (3,7) rectangle (4,6);\node at(3.5,6.5){\tiny\( \bullet \)};
\filldraw[red,draw=black] (4,7) rectangle (5,6);\node at(4.5,6.5){\tiny\( \bullet \)};
\filldraw[red,draw=black] (5,7) rectangle (6,6);\node at(5.5,6.5){\tiny\( \bullet \)};
\filldraw[red,draw=black] (6,7) rectangle (7,6);\node at(6.5,6.5){\tiny\( \bullet \)};
\filldraw[red,draw=black] (7,7) rectangle (8,6);\node at(7.5,6.5){\tiny\( \bullet \)};
\filldraw[red,draw=black] (8,7) rectangle (9,6);\node at(8.5,6.5){\tiny\( \bullet \)};
\filldraw[red,draw=black] (9,7) rectangle (10,6);\node at(9.5,6.5){\tiny\( \bullet \)};
\filldraw[red,draw=black] (10,7) rectangle (11,6);\node at(10.5,6.5){\tiny\( \bullet \)};
\filldraw[red,draw=black] (11,7) rectangle (12,6);\node at(11.5,6.5){\tiny\( \bullet \)};
\filldraw[red,draw=black] (12,7) rectangle (13,6);\node at(12.5,6.5){\tiny\( \bullet \)};
\filldraw[red,draw=black] (13,7) rectangle (14,6);\node at(13.5,6.5){\tiny\( \bullet \)};
\filldraw[red,draw=black] (14,7) rectangle (15,6);\node at(14.5,6.5){\tiny\( \bullet \)};
\draw (0,6) rectangle (1,5);
\draw (1,6) rectangle (2,5);\node at(1.5,5.5){\tiny\( 7 \)};
\draw (2,6) rectangle (3,5);
\filldraw[yellow,draw=black] (3,6) rectangle (4,5);\node at(3.5,5.5){\tiny\( \bullet \)};
\filldraw[yellow,draw=black] (4,6) rectangle (5,5);\node at(4.5,5.5){\tiny\( \bullet \)};
\filldraw[red,draw=black] (5,6) rectangle (6,5);\node at(5.5,5.5){\tiny\( \bullet \)};
\filldraw[red,draw=black] (6,6) rectangle (7,5);\node at(6.5,5.5){\tiny\( \bullet \)};
\filldraw[red,draw=black] (7,6) rectangle (8,5);\node at(7.5,5.5){\tiny\( \bullet \)};
\filldraw[red,draw=black] (8,6) rectangle (9,5);\node at(8.5,5.5){\tiny\( \bullet \)};
\filldraw[red,draw=black] (9,6) rectangle (10,5);\node at(9.5,5.5){\tiny\( \bullet \)};
\filldraw[red,draw=black] (10,6) rectangle (11,5);\node at(10.5,5.5){\tiny\( \bullet \)};
\filldraw[red,draw=black] (11,6) rectangle (12,5);\node at(11.5,5.5){\tiny\( \bullet \)};
\filldraw[red,draw=black] (12,6) rectangle (13,5);\node at(12.5,5.5){\tiny\( \bullet \)};
\filldraw[red,draw=black] (13,6) rectangle (14,5);\node at(13.5,5.5){\tiny\( \bullet \)};
\filldraw[red,draw=black] (14,6) rectangle (15,5);\node at(14.5,5.5){\tiny\( \bullet \)};
\draw (0,5) rectangle (1,4);
\draw (1,5) rectangle (2,4);
\draw (2,5) rectangle (3,4);\node at(2.5,4.5){\tiny\( 6 \)};
\draw (3,5) rectangle (4,4);
\draw (4,5) rectangle (5,4);
\filldraw[red,draw=black] (5,5) rectangle (6,4);\node at(5.5,4.5){\tiny\( \bullet \)};
\filldraw[red,draw=black] (6,5) rectangle (7,4);\node at(6.5,4.5){\tiny\( \bullet \)};
\filldraw[red,draw=black] (7,5) rectangle (8,4);\node at(7.5,4.5){\tiny\( \bullet \)};
\filldraw[red,draw=black] (8,5) rectangle (9,4);\node at(8.5,4.5){\tiny\( \bullet \)};
\filldraw[red,draw=black] (9,5) rectangle (10,4);\node at(9.5,4.5){\tiny\( \bullet \)};
\filldraw[red,draw=black] (10,5) rectangle (11,4);\node at(10.5,4.5){\tiny\( \bullet \)};
\filldraw[red,draw=black] (11,5) rectangle (12,4);\node at(11.5,4.5){\tiny\( \bullet \)};
\filldraw[red,draw=black] (12,5) rectangle (13,4);\node at(12.5,4.5){\tiny\( \bullet \)};
\filldraw[red,draw=black] (13,5) rectangle (14,4);\node at(13.5,4.5){\tiny\( \bullet \)};
\filldraw[red,draw=black] (14,5) rectangle (15,4);\node at(14.5,4.5){\tiny\( \bullet \)};
\draw (0,4) rectangle (1,3);
\draw (1,4) rectangle (2,3);
\draw (2,4) rectangle (3,3);
\draw (3,4) rectangle (4,3);\node at(3.5,3.5){\tiny\( \textcolor{red}{7} \)};
\draw (4,4) rectangle (5,3);
\filldraw[green,draw=black] (5,4) rectangle (6,3);\node at(5.5,3.5){\tiny\( \bullet \)};
\filldraw[red,draw=black] (6,4) rectangle (7,3);\node at(6.5,3.5){\tiny\( \bullet \)};
\filldraw[red,draw=black] (7,4) rectangle (8,3);\node at(7.5,3.5){\tiny\( \bullet \)};
\filldraw[red,draw=black] (8,4) rectangle (9,3);\node at(8.5,3.5){\tiny\( \bullet \)};
\filldraw[red,draw=black] (9,4) rectangle (10,3);\node at(9.5,3.5){\tiny\( \bullet \)};
\filldraw[red,draw=black] (10,4) rectangle (11,3);\node at(10.5,3.5){\tiny\( \bullet \)};
\filldraw[red,draw=black] (11,4) rectangle (12,3);\node at(11.5,3.5){\tiny\( \bullet \)};
\filldraw[red,draw=black] (12,4) rectangle (13,3);\node at(12.5,3.5){\tiny\( \bullet \)};
\filldraw[red,draw=black] (13,4) rectangle (14,3);\node at(13.5,3.5){\tiny\( \bullet \)};
\filldraw[red,draw=black] (14,4) rectangle (15,3);\node at(14.5,3.5){\tiny\( \bullet \)};
\draw (0,3) rectangle (1,2);
\draw (1,3) rectangle (2,2);
\draw (2,3) rectangle (3,2);
\draw (3,3) rectangle (4,2);
\draw (4,3) rectangle (5,2);\node at(4.5,2.5){\tiny\( \textcolor{red}{7} \)};
\draw (5,3) rectangle (6,2);
\filldraw[green,draw=black] (6,3) rectangle (7,2);\node at(6.5,2.5){\tiny\( \bullet \)};
\filldraw[red,draw=black] (7,3) rectangle (8,2);\node at(7.5,2.5){\tiny\( \bullet \)};
\filldraw[red,draw=black] (8,3) rectangle (9,2);\node at(8.5,2.5){\tiny\( \bullet \)};
\filldraw[red,draw=black] (9,3) rectangle (10,2);\node at(9.5,2.5){\tiny\( \bullet \)};
\filldraw[red,draw=black] (10,3) rectangle (11,2);\node at(10.5,2.5){\tiny\( \bullet \)};
\filldraw[red,draw=black] (11,3) rectangle (12,2);\node at(11.5,2.5){\tiny\( \bullet \)};
\filldraw[red,draw=black] (12,3) rectangle (13,2);\node at(12.5,2.5){\tiny\( \bullet \)};
\filldraw[red,draw=black] (13,3) rectangle (14,2);\node at(13.5,2.5){\tiny\( \bullet \)};
\filldraw[red,draw=black] (14,3) rectangle (15,2);\node at(14.5,2.5){\tiny\( \bullet \)};
\draw (0,2) rectangle (1,1);
\draw (1,2) rectangle (2,1);
\draw (2,2) rectangle (3,1);
\draw (3,2) rectangle (4,1);
\draw (4,2) rectangle (5,1);
\draw (5,2) rectangle (6,1);\node at(5.5,1.5){\tiny\( \textcolor{red}{5} \)};
\draw (6,2) rectangle (7,1);
\draw (7,2) rectangle (8,1);
\draw (8,2) rectangle (9,1);
\filldraw[green,draw=black] (9,2) rectangle (10,1);\node at(9.5,1.5){\tiny\( \bullet \)};
\filldraw[red,draw=black] (10,2) rectangle (11,1);\node at(10.5,1.5){\tiny\( \bullet \)};
\filldraw[red,draw=black] (11,2) rectangle (12,1);\node at(11.5,1.5){\tiny\( \bullet \)};
\filldraw[red,draw=black] (12,2) rectangle (13,1);\node at(12.5,1.5){\tiny\( \bullet \)};
\filldraw[red,draw=black] (13,2) rectangle (14,1);\node at(13.5,1.5){\tiny\( \bullet \)};
\filldraw[red,draw=black] (14,2) rectangle (15,1);\node at(14.5,1.5){\tiny\( \bullet \)};
\draw (0,1) rectangle (1,0);
\draw (1,1) rectangle (2,0);
\draw (2,1) rectangle (3,0);
\draw (3,1) rectangle (4,0);
\draw (4,1) rectangle (5,0);
\draw (5,1) rectangle (6,0);
\draw (6,1) rectangle (7,0);\node at(6.5,0.5){\tiny\( \textcolor{red}{5} \)};
\draw (7,1) rectangle (8,0);
\draw (8,1) rectangle (9,0);
\draw (9,1) rectangle (10,0);
\filldraw[green,draw=black] (10,1) rectangle (11,0);\node at(10.5,0.5){\tiny\( \bullet \)};
\filldraw[red,draw=black] (11,1) rectangle (12,0);\node at(11.5,0.5){\tiny\( \bullet \)};
\filldraw[red,draw=black] (12,1) rectangle (13,0);\node at(12.5,0.5){\tiny\( \bullet \)};
\filldraw[red,draw=black] (13,1) rectangle (14,0);\node at(13.5,0.5){\tiny\( \bullet \)};
\filldraw[red,draw=black] (14,1) rectangle (15,0);\node at(14.5,0.5){\tiny\( \bullet \)};
\draw (0,0) rectangle (1,-1);
\draw (1,0) rectangle (2,-1);
\draw (2,0) rectangle (3,-1);
\draw (3,0) rectangle (4,-1);
\draw (4,0) rectangle (5,-1);
\draw (5,0) rectangle (6,-1);
\draw (6,0) rectangle (7,-1);
\draw (7,0) rectangle (8,-1);\node at(7.5,-0.5){\tiny\( 4 \)};
\draw (8,0) rectangle (9,-1);
\draw (9,0) rectangle (10,-1);
\draw (10,0) rectangle (11,-1);
\draw (11,0) rectangle (12,-1);
\filldraw[red,draw=black] (12,0) rectangle (13,-1);\node at(12.5,-0.5){\tiny\( \bullet \)};
\filldraw[red,draw=black] (13,0) rectangle (14,-1);\node at(13.5,-0.5){\tiny\( \bullet \)};
\filldraw[red,draw=black] (14,0) rectangle (15,-1);\node at(14.5,-0.5){\tiny\( \bullet \)};
\draw (0,-1) rectangle (1,-2);
\draw (1,-1) rectangle (2,-2);
\draw (2,-1) rectangle (3,-2);
\draw (3,-1) rectangle (4,-2);
\draw (4,-1) rectangle (5,-2);
\draw (5,-1) rectangle (6,-2);
\draw (6,-1) rectangle (7,-2);
\draw (7,-1) rectangle (8,-2);
\draw (8,-1) rectangle (9,-2);\node at(8.5,-1.5){\tiny\( 4 \)};
\draw (9,-1) rectangle (10,-2);
\draw (10,-1) rectangle (11,-2);
\draw (11,-1) rectangle (12,-2);
\draw (12,-1) rectangle (13,-2);
\filldraw[red,draw=black] (13,-1) rectangle (14,-2);\node at(13.5,-1.5){\tiny\( \bullet \)};
\filldraw[red,draw=black] (14,-1) rectangle (15,-2);\node at(14.5,-1.5){\tiny\( \bullet \)};
\draw (0,-2) rectangle (1,-3);
\draw (1,-2) rectangle (2,-3);
\draw (2,-2) rectangle (3,-3);
\draw (3,-2) rectangle (4,-3);
\draw (4,-2) rectangle (5,-3);
\draw (5,-2) rectangle (6,-3);
\draw (6,-2) rectangle (7,-3);
\draw (7,-2) rectangle (8,-3);
\draw (8,-2) rectangle (9,-3);
\draw (9,-2) rectangle (10,-3);\node at(9.5,-2.5){\tiny\( \textcolor{red}{4} \)};
\draw (10,-2) rectangle (11,-3);
\draw (11,-2) rectangle (12,-3);
\draw (12,-2) rectangle (13,-3);
\draw (13,-2) rectangle (14,-3);
\filldraw[blue!50,draw=black] (14,-2) rectangle (15,-3);\node at(14.5,-2.5){\tiny\( \bullet \)};
\draw (0,-3) rectangle (1,-4);
\draw (1,-3) rectangle (2,-4);
\draw (2,-3) rectangle (3,-4);
\draw (3,-3) rectangle (4,-4);
\draw (4,-3) rectangle (5,-4);
\draw (5,-3) rectangle (6,-4);
\draw (6,-3) rectangle (7,-4);
\draw (7,-3) rectangle (8,-4);
\draw (8,-3) rectangle (9,-4);
\draw (9,-3) rectangle (10,-4);
\draw (10,-3) rectangle (11,-4);\node at(10.5,-3.5){\tiny\( \textcolor{red}{5} \)};
\draw (11,-3) rectangle (12,-4);
\draw (12,-3) rectangle (13,-4);
\draw (13,-3) rectangle (14,-4);
\filldraw[blue!50,draw=black] (14,-3) rectangle (15,-4);\node at(14.5,-3.5){\tiny\( \bullet \)};
\draw (0,-4) rectangle (1,-5);
\draw (1,-4) rectangle (2,-5);
\draw (2,-4) rectangle (3,-5);
\draw (3,-4) rectangle (4,-5);
\draw (4,-4) rectangle (5,-5);
\draw (5,-4) rectangle (6,-5);
\draw (6,-4) rectangle (7,-5);
\draw (7,-4) rectangle (8,-5);
\draw (8,-4) rectangle (9,-5);
\draw (9,-4) rectangle (10,-5);
\draw (10,-4) rectangle (11,-5);
\draw (11,-4) rectangle (12,-5);\node at(11.5,-4.5){\tiny\( 3 \)};
\draw (12,-4) rectangle (13,-5);
\draw (13,-4) rectangle (14,-5);
\draw (14,-4) rectangle (15,-5);
\draw (0,-5) rectangle (1,-6);
\draw (1,-5) rectangle (2,-6);
\draw (2,-5) rectangle (3,-6);
\draw (3,-5) rectangle (4,-6);
\draw (4,-5) rectangle (5,-6);
\draw (5,-5) rectangle (6,-6);
\draw (6,-5) rectangle (7,-6);
\draw (7,-5) rectangle (8,-6);
\draw (8,-5) rectangle (9,-6);
\draw (9,-5) rectangle (10,-6);
\draw (10,-5) rectangle (11,-6);
\draw (11,-5) rectangle (12,-6);
\draw (12,-5) rectangle (13,-6);\node at(12.5,-5.5){\tiny\( 2 \)};
\draw (13,-5) rectangle (14,-6);
\draw (14,-5) rectangle (15,-6);
\draw (0,-6) rectangle (1,-7);
\draw (1,-6) rectangle (2,-7);
\draw (2,-6) rectangle (3,-7);
\draw (3,-6) rectangle (4,-7);
\draw (4,-6) rectangle (5,-7);
\draw (5,-6) rectangle (6,-7);
\draw (6,-6) rectangle (7,-7);
\draw (7,-6) rectangle (8,-7);
\draw (8,-6) rectangle (9,-7);
\draw (9,-6) rectangle (10,-7);
\draw (10,-6) rectangle (11,-7);
\draw (11,-6) rectangle (12,-7);
\draw (12,-6) rectangle (13,-7);
\draw (13,-6) rectangle (14,-7);\node at(13.5,-6.5){\tiny\( 1 \)};
\draw (14,-6) rectangle (15,-7);
\draw (0,-7) rectangle (1,-8);
\draw (1,-7) rectangle (2,-8);
\draw (2,-7) rectangle (3,-8);
\draw (3,-7) rectangle (4,-8);
\draw (4,-7) rectangle (5,-8);
\draw (5,-7) rectangle (6,-8);
\draw (6,-7) rectangle (7,-8);
\draw (7,-7) rectangle (8,-8);
\draw (8,-7) rectangle (9,-8);
\draw (9,-7) rectangle (10,-8);
\draw (10,-7) rectangle (11,-8);
\draw (11,-7) rectangle (12,-8);
\draw (12,-7) rectangle (13,-8);
\draw (13,-7) rectangle (14,-8);
\draw (14,-7) rectangle (15,-8);\node at(14.5,-7.5){\tiny\( 1 \)};
%
\draw[purple,line width=0.8pt] (6,1.5)--(9,1.5);
\draw[purple,line width=0.8pt] (5.5,2)--(5.5,3);
\draw[purple,line width=0.8pt] (4,3.5)--(5,3.5);
\draw[purple,line width=0.8pt] (3.5,4)--(3.5,5);
%
\draw[black,line width=0.8pt,->] (15,-2.5)--(15.8,-2.5);
\node at (16.1,-2.5){ \tiny{\textbf{$z$}} };
\draw[black,line width=0.8pt,->] (15,1.5)--(15.8,1.5);
\node at (17.2,1.5){ \tiny{\textbf{${\rm up}_{\Psi}(z)$}} };
\draw[black,line width=0.8pt,->] (15,3.5)--(15.8,3.5);
\node at (16.1,3.5){ \tiny{\textbf{$y$}} };
\draw[black,line width=0.8pt,->] (3.5,7)--(3.5,7.8);
\node at (3.5,8.2){ \tiny{\textbf{$y$}} };
\draw[black,line width=0.8pt,->] (5.5,7)--(5.5,7.8);
\node at (6,8.2){ \tiny{\textbf{${\rm down}_{\Psi}(y)$}} };
\draw[black,line width=0.8pt,->] (9.5,7)--(9.5,7.8);
\node at (9.5,8.2){ \tiny{\textbf{$z$}} };
\end{tikzpicture}
=&\
\begin{tikzpicture}[scale=.3,line width=0.5pt,baseline=(a.base)]
\draw (0,7) rectangle (1,6);\node at(0.5,6.5){\tiny\( 7 \)};
\draw (1,7) rectangle (2,6);
\filldraw[red,draw=black] (2,7) rectangle (3,6);\node at(2.5,6.5){\tiny\( \bullet \)};
\filldraw[red,draw=black] (3,7) rectangle (4,6);\node at(3.5,6.5){\tiny\( \bullet \)};
\filldraw[red,draw=black] (4,7) rectangle (5,6);\node at(4.5,6.5){\tiny\( \bullet \)};
\filldraw[red,draw=black] (5,7) rectangle (6,6);\node at(5.5,6.5){\tiny\( \bullet \)};
\filldraw[red,draw=black] (6,7) rectangle (7,6);\node at(6.5,6.5){\tiny\( \bullet \)};
\filldraw[red,draw=black] (7,7) rectangle (8,6);\node at(7.5,6.5){\tiny\( \bullet \)};
\filldraw[red,draw=black] (8,7) rectangle (9,6);\node at(8.5,6.5){\tiny\( \bullet \)};
\filldraw[red,draw=black] (9,7) rectangle (10,6);\node at(9.5,6.5){\tiny\( \bullet \)};
\filldraw[red,draw=black] (10,7) rectangle (11,6);\node at(10.5,6.5){\tiny\( \bullet \)};
\filldraw[red,draw=black] (11,7) rectangle (12,6);\node at(11.5,6.5){\tiny\( \bullet \)};
\filldraw[red,draw=black] (12,7) rectangle (13,6);\node at(12.5,6.5){\tiny\( \bullet \)};
\filldraw[red,draw=black] (13,7) rectangle (14,6);\node at(13.5,6.5){\tiny\( \bullet \)};
\filldraw[red,draw=black] (14,7) rectangle (15,6);\node at(14.5,6.5){\tiny\( \bullet \)};
\draw (0,6) rectangle (1,5);
\draw (1,6) rectangle (2,5);\node at(1.5,5.5){\tiny\( 7 \)};
\draw (2,6) rectangle (3,5);
\filldraw[red,draw=black] (3,6) rectangle (4,5);\node at(3.5,5.5){\tiny\( \bullet \)};
\filldraw[red,draw=black] (4,6) rectangle (5,5);\node at(4.5,5.5){\tiny\( \bullet \)};
\filldraw[red,draw=black] (5,6) rectangle (6,5);\node at(5.5,5.5){\tiny\( \bullet \)};
\filldraw[red,draw=black] (6,6) rectangle (7,5);\node at(6.5,5.5){\tiny\( \bullet \)};
\filldraw[red,draw=black] (7,6) rectangle (8,5);\node at(7.5,5.5){\tiny\( \bullet \)};
\filldraw[red,draw=black] (8,6) rectangle (9,5);\node at(8.5,5.5){\tiny\( \bullet \)};
\filldraw[red,draw=black] (9,6) rectangle (10,5);\node at(9.5,5.5){\tiny\( \bullet \)};
\filldraw[red,draw=black] (10,6) rectangle (11,5);\node at(10.5,5.5){\tiny\( \bullet \)};
\filldraw[red,draw=black] (11,6) rectangle (12,5);\node at(11.5,5.5){\tiny\( \bullet \)};
\filldraw[red,draw=black] (12,6) rectangle (13,5);\node at(12.5,5.5){\tiny\( \bullet \)};
\filldraw[red,draw=black] (13,6) rectangle (14,5);\node at(13.5,5.5){\tiny\( \bullet \)};
\filldraw[red,draw=black] (14,6) rectangle (15,5);\node at(14.5,5.5){\tiny\( \bullet \)};
\draw (0,5) rectangle (1,4);
\draw (1,5) rectangle (2,4);
\draw (2,5) rectangle (3,4);\node at(2.5,4.5){\tiny\( 6 \)};
\draw (3,5) rectangle (4,4);
\draw (4,5) rectangle (5,4);
\filldraw[red,draw=black] (5,5) rectangle (6,4);\node at(5.5,4.5){\tiny\( \bullet \)};
\filldraw[red,draw=black] (6,5) rectangle (7,4);\node at(6.5,4.5){\tiny\( \bullet \)};
\filldraw[red,draw=black] (7,5) rectangle (8,4);\node at(7.5,4.5){\tiny\( \bullet \)};
\filldraw[red,draw=black] (8,5) rectangle (9,4);\node at(8.5,4.5){\tiny\( \bullet \)};
\filldraw[red,draw=black] (9,5) rectangle (10,4);\node at(9.5,4.5){\tiny\( \bullet \)};
\filldraw[red,draw=black] (10,5) rectangle (11,4);\node at(10.5,4.5){\tiny\( \bullet \)};
\filldraw[red,draw=black] (11,5) rectangle (12,4);\node at(11.5,4.5){\tiny\( \bullet \)};
\filldraw[red,draw=black] (12,5) rectangle (13,4);\node at(12.5,4.5){\tiny\( \bullet \)};
\filldraw[red,draw=black] (13,5) rectangle (14,4);\node at(13.5,4.5){\tiny\( \bullet \)};
\filldraw[red,draw=black] (14,5) rectangle (15,4);\node at(14.5,4.5){\tiny\( \bullet \)};
\draw (0,4) rectangle (1,3);
\draw (1,4) rectangle (2,3);
\draw (2,4) rectangle (3,3);
\draw (3,4) rectangle (4,3);\node at(3.5,3.5){\tiny\( 7 \)};
\draw (4,4) rectangle (5,3);
\filldraw[red,draw=black] (5,4) rectangle (6,3);\node at(5.5,3.5){\tiny\( \bullet \)};
\filldraw[red,draw=black] (6,4) rectangle (7,3);\node at(6.5,3.5){\tiny\( \bullet \)};
\filldraw[red,draw=black] (7,4) rectangle (8,3);\node at(7.5,3.5){\tiny\( \bullet \)};
\filldraw[red,draw=black] (8,4) rectangle (9,3);\node at(8.5,3.5){\tiny\( \bullet \)};
\filldraw[red,draw=black] (9,4) rectangle (10,3);\node at(9.5,3.5){\tiny\( \bullet \)};
\filldraw[red,draw=black] (10,4) rectangle (11,3);\node at(10.5,3.5){\tiny\( \bullet \)};
\filldraw[red,draw=black] (11,4) rectangle (12,3);\node at(11.5,3.5){\tiny\( \bullet \)};
\filldraw[red,draw=black] (12,4) rectangle (13,3);\node at(12.5,3.5){\tiny\( \bullet \)};
\filldraw[red,draw=black] (13,4) rectangle (14,3);\node at(13.5,3.5){\tiny\( \bullet \)};
\filldraw[red,draw=black] (14,4) rectangle (15,3);\node at(14.5,3.5){\tiny\( \bullet \)};
\draw (0,3) rectangle (1,2);
\draw (1,3) rectangle (2,2);
\draw (2,3) rectangle (3,2);
\draw (3,3) rectangle (4,2);
\draw (4,3) rectangle (5,2);\node at(4.5,2.5){\tiny\( 7 \)};
\draw (5,3) rectangle (6,2);
\filldraw[red,draw=black] (6,3) rectangle (7,2);\node at(6.5,2.5){\tiny\( \bullet \)};
\filldraw[red,draw=black] (7,3) rectangle (8,2);\node at(7.5,2.5){\tiny\( \bullet \)};
\filldraw[red,draw=black] (8,3) rectangle (9,2);\node at(8.5,2.5){\tiny\( \bullet \)};
\filldraw[red,draw=black] (9,3) rectangle (10,2);\node at(9.5,2.5){\tiny\( \bullet \)};
\filldraw[red,draw=black] (10,3) rectangle (11,2);\node at(10.5,2.5){\tiny\( \bullet \)};
\filldraw[red,draw=black] (11,3) rectangle (12,2);\node at(11.5,2.5){\tiny\( \bullet \)};
\filldraw[red,draw=black] (12,3) rectangle (13,2);\node at(12.5,2.5){\tiny\( \bullet \)};
\filldraw[red,draw=black] (13,3) rectangle (14,2);\node at(13.5,2.5){\tiny\( \bullet \)};
\filldraw[red,draw=black] (14,3) rectangle (15,2);\node at(14.5,2.5){\tiny\( \bullet \)};
\draw (0,2) rectangle (1,1);
\draw (1,2) rectangle (2,1);
\draw (2,2) rectangle (3,1);
\draw (3,2) rectangle (4,1);
\draw (4,2) rectangle (5,1);
\draw (5,2) rectangle (6,1);\node at(5.5,1.5){\tiny\( 5 \)};
\draw (6,2) rectangle (7,1);
\draw (7,2) rectangle (8,1);
\draw (8,2) rectangle (9,1);
\filldraw[red,draw=black] (9,2) rectangle (10,1);\node at(9.5,1.5){\tiny\( \bullet \)};
\filldraw[red,draw=black] (10,2) rectangle (11,1);\node at(10.5,1.5){\tiny\( \bullet \)};
\filldraw[red,draw=black] (11,2) rectangle (12,1);\node at(11.5,1.5){\tiny\( \bullet \)};
\filldraw[red,draw=black] (12,2) rectangle (13,1);\node at(12.5,1.5){\tiny\( \bullet \)};
\filldraw[red,draw=black] (13,2) rectangle (14,1);\node at(13.5,1.5){\tiny\( \bullet \)};
\filldraw[red,draw=black] (14,2) rectangle (15,1);\node at(14.5,1.5){\tiny\( \bullet \)};
\draw (0,1) rectangle (1,0);
\draw (1,1) rectangle (2,0);
\draw (2,1) rectangle (3,0);
\draw (3,1) rectangle (4,0);
\draw (4,1) rectangle (5,0);
\draw (5,1) rectangle (6,0);
\draw (6,1) rectangle (7,0);\node at(6.5,0.5){\tiny\( 5 \)};
\draw (7,1) rectangle (8,0);
\draw (8,1) rectangle (9,0);
\draw (9,1) rectangle (10,0);
\filldraw[red,draw=black] (10,1) rectangle (11,0);\node at(10.5,0.5){\tiny\( \bullet \)};
\filldraw[red,draw=black] (11,1) rectangle (12,0);\node at(11.5,0.5){\tiny\( \bullet \)};
\filldraw[red,draw=black] (12,1) rectangle (13,0);\node at(12.5,0.5){\tiny\( \bullet \)};
\filldraw[red,draw=black] (13,1) rectangle (14,0);\node at(13.5,0.5){\tiny\( \bullet \)};
\filldraw[red,draw=black] (14,1) rectangle (15,0);\node at(14.5,0.5){\tiny\( \bullet \)};
\draw (0,0) rectangle (1,-1);
\draw (1,0) rectangle (2,-1);
\draw (2,0) rectangle (3,-1);
\draw (3,0) rectangle (4,-1);
\draw (4,0) rectangle (5,-1);
\draw (5,0) rectangle (6,-1);
\draw (6,0) rectangle (7,-1);
\draw (7,0) rectangle (8,-1);\node at(7.5,-0.5){\tiny\( 4 \)};
\draw (8,0) rectangle (9,-1);
\draw (9,0) rectangle (10,-1);
\draw (10,0) rectangle (11,-1);
\draw (11,0) rectangle (12,-1);
\filldraw[red,draw=black] (12,0) rectangle (13,-1);\node at(12.5,-0.5){\tiny\( \bullet \)};
\filldraw[red,draw=black] (13,0) rectangle (14,-1);\node at(13.5,-0.5){\tiny\( \bullet \)};
\filldraw[red,draw=black] (14,0) rectangle (15,-1);\node at(14.5,-0.5){\tiny\( \bullet \)};
\draw (0,-1) rectangle (1,-2);
\draw (1,-1) rectangle (2,-2);
\draw (2,-1) rectangle (3,-2);
\draw (3,-1) rectangle (4,-2);
\draw (4,-1) rectangle (5,-2);
\draw (5,-1) rectangle (6,-2);
\draw (6,-1) rectangle (7,-2);
\draw (7,-1) rectangle (8,-2);
\draw (8,-1) rectangle (9,-2);\node at(8.5,-1.5){\tiny\( 4 \)};
\draw (9,-1) rectangle (10,-2);
\draw (10,-1) rectangle (11,-2);
\draw (11,-1) rectangle (12,-2);
\draw (12,-1) rectangle (13,-2);
\filldraw[red,draw=black] (13,-1) rectangle (14,-2);\node at(13.5,-1.5){\tiny\( \bullet \)};
\filldraw[red,draw=black] (14,-1) rectangle (15,-2);\node at(14.5,-1.5){\tiny\( \bullet \)};
\draw (0,-2) rectangle (1,-3);
\draw (1,-2) rectangle (2,-3);
\draw (2,-2) rectangle (3,-3);
\draw (3,-2) rectangle (4,-3);
\draw (4,-2) rectangle (5,-3);
\draw (5,-2) rectangle (6,-3);
\draw (6,-2) rectangle (7,-3);
\draw (7,-2) rectangle (8,-3);
\draw (8,-2) rectangle (9,-3);
\draw (9,-2) rectangle (10,-3);\node at(9.5,-2.5){\tiny\( 4 \)};
\draw (10,-2) rectangle (11,-3);
\draw (11,-2) rectangle (12,-3);
\draw (12,-2) rectangle (13,-3);
\draw (13,-2) rectangle (14,-3);
\filldraw[red,draw=black] (14,-2) rectangle (15,-3);\node at(14.5,-2.5){\tiny\( \bullet \)};
\draw (0,-3) rectangle (1,-4);
\draw (1,-3) rectangle (2,-4);
\draw (2,-3) rectangle (3,-4);
\draw (3,-3) rectangle (4,-4);
\draw (4,-3) rectangle (5,-4);
\draw (5,-3) rectangle (6,-4);
\draw (6,-3) rectangle (7,-4);
\draw (7,-3) rectangle (8,-4);
\draw (8,-3) rectangle (9,-4);
\draw (9,-3) rectangle (10,-4);
\draw (10,-3) rectangle (11,-4);\node at(10.5,-3.5){\tiny\( \textcolor{green}{4} \)};
\draw (11,-3) rectangle (12,-4);
\draw (12,-3) rectangle (13,-4);
\draw (13,-3) rectangle (14,-4);
\filldraw[red,draw=black] (14,-3) rectangle (15,-4);\node at(14.5,-3.5){\tiny\( \bullet \)};
\draw (0,-4) rectangle (1,-5);
\draw (1,-4) rectangle (2,-5);
\draw (2,-4) rectangle (3,-5);
\draw (3,-4) rectangle (4,-5);
\draw (4,-4) rectangle (5,-5);
\draw (5,-4) rectangle (6,-5);
\draw (6,-4) rectangle (7,-5);
\draw (7,-4) rectangle (8,-5);
\draw (8,-4) rectangle (9,-5);
\draw (9,-4) rectangle (10,-5);
\draw (10,-4) rectangle (11,-5);
\draw (11,-4) rectangle (12,-5);\node at(11.5,-4.5){\tiny\( 3 \)};
\draw (12,-4) rectangle (13,-5);
\draw (13,-4) rectangle (14,-5);
\draw (14,-4) rectangle (15,-5);
\draw (0,-5) rectangle (1,-6);
\draw (1,-5) rectangle (2,-6);
\draw (2,-5) rectangle (3,-6);
\draw (3,-5) rectangle (4,-6);
\draw (4,-5) rectangle (5,-6);
\draw (5,-5) rectangle (6,-6);
\draw (6,-5) rectangle (7,-6);
\draw (7,-5) rectangle (8,-6);
\draw (8,-5) rectangle (9,-6);
\draw (9,-5) rectangle (10,-6);
\draw (10,-5) rectangle (11,-6);
\draw (11,-5) rectangle (12,-6);
\draw (12,-5) rectangle (13,-6);\node at(12.5,-5.5){\tiny\( 2 \)};
\draw (13,-5) rectangle (14,-6);
\draw (14,-5) rectangle (15,-6);
\draw (0,-6) rectangle (1,-7);
\draw (1,-6) rectangle (2,-7);
\draw (2,-6) rectangle (3,-7);
\draw (3,-6) rectangle (4,-7);
\draw (4,-6) rectangle (5,-7);
\draw (5,-6) rectangle (6,-7);
\draw (6,-6) rectangle (7,-7);
\draw (7,-6) rectangle (8,-7);
\draw (8,-6) rectangle (9,-7);
\draw (9,-6) rectangle (10,-7);
\draw (10,-6) rectangle (11,-7);
\draw (11,-6) rectangle (12,-7);
\draw (12,-6) rectangle (13,-7);
\draw (13,-6) rectangle (14,-7);\node at(13.5,-6.5){\tiny\( 1 \)};
\draw (14,-6) rectangle (15,-7);
\draw (0,-7) rectangle (1,-8);
\draw (1,-7) rectangle (2,-8);
\draw (2,-7) rectangle (3,-8);
\draw (3,-7) rectangle (4,-8);
\draw (4,-7) rectangle (5,-8);
\draw (5,-7) rectangle (6,-8);
\draw (6,-7) rectangle (7,-8);
\draw (7,-7) rectangle (8,-8);
\draw (8,-7) rectangle (9,-8);
\draw (9,-7) rectangle (10,-8);
\draw (10,-7) rectangle (11,-8);
\draw (11,-7) rectangle (12,-8);
\draw (12,-7) rectangle (13,-8);
\draw (13,-7) rectangle (14,-8);
\draw (14,-7) rectangle (15,-8);\node at(14.5,-7.5){\tiny\( 1 \)};
%
\draw[black,line width=0.8pt,->] (15,-3.5)--(15.8,-3.5);
\node at (16.8,-3.5){ \tiny{\textbf{$z+1$}} };
\end{tikzpicture}\\
& \\
=&\ K(\Psi;M_2;\gamma-\epsilon_{z+1}),
\end{align*}
where $M_1\sqcup y = M_2$.
Here, the red, yellow, green, and blue roots together form the root ideal $\Psi$. Among them, the yellow  (respectively, blue) roots represents a ceiling (respectively, a wall), while the green roots correspond to mirrors. The red path is the bounce path $(2, 4, 6, 10)$
passing through $y=4$ and $z=10$.
\hspace*{\fill}$\square$
\end{exam}

\subsection{Weighted $K$-$k$-Schur functions}
In this subsection, we first review Katalan formulas of $K$-$k$-Schur functions and closed $k$-Schur Katalan functions. Next we define a new subfamily of Katalan functions, named weighted $K$-$k$-Schur functions, which are used as a transition between $K$-$k$-Schur functions and closed $k$-Schur Katalan functions.

For $\lambda\in\text{P}_{\ell}^{k}$, define a root ideal
\begin{equation}
	\Delta^{k}(\lambda):=\{(i,j)\in\Delta_{\ell}^{+} \mid k-\lambda_{i}+i<j \} \subseteq \Delta_{\ell}^{+}.
	\mlabel{eq:dkl}
\end{equation}
For $\mathcal{R}\subseteq\Delta_{\ell}^{+}$, denote by
\begin{equation}
L(\mathcal{R}):=\bigsqcup_{(i,j)\in\mathcal{R}}\{j\}
\mlabel{eq:Lroot}
\end{equation}
the multiset of second components of elements in $\mathcal{R}$.
In particular, if $\mathcal{R}$ is a root ideal, we denote $$K(\Psi;\mathcal{R};\gamma):= K(\Psi;L(\mathcal{R});\gamma).$$
For later reference, we compile several simple yet useful facts about the root ideal $\dkl$. We abbreviate $\bdo{\lambda}{x} := \down{\dkl}{x}$ and denote by $\bott_\lambda$ the bottom of $\dkl$.

\begin{remark}(\cite[Remark~5.1]{BMS}, \cite[Remark~3.1]{FG})
Let $\lambda\in\pkl$.
\begin{enumerate}
\item For each $x\in[\bott_\lambda]$, the notation $\bdo{\lambda}{x}$ is defined and the root $(x,\bdo{\lambda}{x})$ is removable in $\dkl$. Thus, $\text{bot}_{\dkl}(x)\in[\bott_\lambda+1,\ell]$ for each $x\in[\ell]$.
\item If $x\in[\bott_\lambda-1]$ and $k>\lambda_{x}=\lambda_{x+1}$, then there is a mirror in rows $x,x+1$ in $\dkl$.
\item If $x\in[\bott_\lambda-1]$ and $\lambda_{x}>\lambda_{x+1}$, then there is a ceiling in columns
    $\bdo{\lambda}{x},\bdo{\lambda}{x}+1$ in $\dkl$.
\end{enumerate}
\mlabel{re:rootfact}
\end{remark}

The notion of the $K$-$k$-Schur function was introduced by Lam et al.~\cite{LSS}, and later, Morse~\cite{Mor} provided Pieri rules that offer an effective characterization. For our purposes, we recall here only its Katalan formula.

\begin{lemma}(\cite[Theorem 2.6]{BMS})
For $\lambda\in\pkl$, the $K$-$k$-Schur function $g_{\lambda}^{(k)}$ satisfies
\begin{equation*}
g_{\lambda}^{(k)} = K(\Delta^{k}(\lambda);\Delta^{k+1}(\lambda);\lambda).
\end{equation*}
\mlabel{lem:KkSch}
\end{lemma}

The following is another family of Katalan functions, named
closed $k$-Schur Katalan function.

\begin{defn}(\cite[Definition 2.11]{BMS})
For $\lambda\in\text{P}_{\ell}^{k}$, the {\em closed $k$-Schur Katalan function} is defined by
\begin{equation*}
\fg{\lambda}{k} :=K(\Delta^{k}(\lambda);\Delta^{k}(\lambda);\lambda).
\end{equation*}
\mlabel{defn:cKataFunc}
\end{defn}

The key subtlety between the Katalan formulas for the $K$-$k$-Schur function $g_{\lambda}^{(k)}$ and the closed $k$-Schur Katalan function $\fg{\lambda}{k}$ lies in the choice of multiset.
By~(\mref{eq:dkl}) and Remark~\mref{re:rootfact},
$$
\De{k+1}{\lambda} = \De{k}{\lambda}\setminus \{ (x, \bdo{\lambda}{x}) \mid x\in[\bott_\lambda] \},
$$
and so
\begin{equation}
L\big(\De{k+1}{\lambda}\big) \overset{\eqref{eq:Lroot}}{=} L\big(\De{k}{\lambda}\big)\setminus \{ \bdo{\lambda}{x} \mid x\in[\bott_\lambda] \}.
\mlabel{eq:Dk1toDk}
\end{equation}

We use the following example to more clearly illustrate the difference between $K$-$k$-Schur functions and closed $k$-Schur Katalan functions.

\begin{exam}
Let
\[
k = 7, \quad \ell = 6,\quad \lambda = (7,6,6,6,4,3)\in{\rm P}_6^7.
\]
Then
\begin{align*}
\dkl =&\ \big\{(1,2),(1,3),(1,4),(1,5),(1,6),(2,4),(2,5),(2,6),(3,5),(3,6),(4,6)  \big\},\\
\De{k+1}{\lambda}=&\ \big\{(1,3),(1,4),(1,5),(1,6),(2,5),(2,6),(3,6)\big\}.
\end{align*}
By~(\mref{eq:Lroot}),
\begin{equation}
L\big(\dkl\big) = \{ 2,3,4,4,5,5,5,6,6,6,6 \},\quad L\big(\De{k+1}{\lambda}\big) =\{ 3,4,5,5,6,6,6 \}.
\mlabel{eq:twoL}
\end{equation}
We represent the corresponding $K$-$k$-Schur function and closed $k$-Schur Katalan function as
\begin{equation*}
g_\lambda^{(k)} =
\begin{tikzpicture}[scale=.4,line width=0.5pt,baseline=(a.base)]
\draw (0,3) rectangle (1,2);\node at(0.5,2.5){\scriptsize\( 7 \)};
\filldraw[red,draw=black] (1,3) rectangle (2,2);
\filldraw[red,draw=black] (2,3) rectangle (3,2);\node at(2.5,2.5){\scriptsize\( \bullet \)};
\filldraw[red,draw=black] (3,3) rectangle (4,2);\node at(3.5,2.5){\scriptsize\( \bullet \)};
\filldraw[red,draw=black] (4,3) rectangle (5,2);\node at(4.5,2.5){\scriptsize\( \bullet \)};
\filldraw[red,draw=black] (5,3) rectangle (6,2);\node at(5.5,2.5){\scriptsize\( \bullet \)};
\draw (0,2) rectangle (1,1);
\draw (1,2) rectangle (2,1);\node at(1.5,1.5){\scriptsize\( 6 \)};
\draw (2,2) rectangle (3,1);
\filldraw[red,draw=black] (3,2) rectangle (4,1);
\filldraw[red,draw=black] (4,2) rectangle (5,1);\node at(4.5,1.5){\scriptsize\( \bullet \)};
\filldraw[red,draw=black] (5,2) rectangle (6,1);\node at(5.5,1.5){\scriptsize\( \bullet \)};
\draw (0,1) rectangle (1,0);
\draw (1,1) rectangle (2,0);
\draw (2,1) rectangle (3,0);\node at(2.5,0.5){\scriptsize\( 6 \)};
\draw (3,1) rectangle (4,0);
\filldraw[red,draw=black] (4,1) rectangle (5,0);
\filldraw[red,draw=black] (5,1) rectangle (6,0);\node at(5.5,0.5){\scriptsize\( \bullet \)};
\draw (0,0) rectangle (1,-1);
\draw (1,0) rectangle (2,-1);
\draw (2,0) rectangle (3,-1);
\draw (3,0) rectangle (4,-1);\node at(3.5,-0.5){\scriptsize\( 6 \)};
\draw (4,0) rectangle (5,-1);
\filldraw[red,draw=black] (5,0) rectangle (6,-1);
\draw (0,-1) rectangle (1,-2);
\draw (1,-1) rectangle (2,-2);
\draw (2,-1) rectangle (3,-2);
\draw (3,-1) rectangle (4,-2);
\draw (4,-1) rectangle (5,-2);\node at(4.5,-1.5){\scriptsize\( 4 \)};
\draw (5,-1) rectangle (6,-2);
\draw (0,-2) rectangle (1,-3);
\draw (1,-2) rectangle (2,-3);
\draw (2,-2) rectangle (3,-3);
\draw (3,-2) rectangle (4,-3);
\draw (4,-2) rectangle (5,-3);
\draw (5,-2) rectangle (6,-3);\node at(5.5,-2.5){\scriptsize\( 3 \)};
\end{tikzpicture}
\, , \quad \quad \quad
\fg{\lambda}{k} =
\begin{tikzpicture}[scale=.4,line width=0.5pt,baseline=(a.base)]
\draw (0,3) rectangle (1,2);\node at(0.5,2.5){\scriptsize\( 7 \)};
\filldraw[red,draw=black] (1,3) rectangle (2,2);\node at(1.5,2.5){\scriptsize\( \bullet \)};
\filldraw[red,draw=black] (2,3) rectangle (3,2);\node at(2.5,2.5){\scriptsize\( \bullet \)};
\filldraw[red,draw=black] (3,3) rectangle (4,2);\node at(3.5,2.5){\scriptsize\( \bullet \)};
\filldraw[red,draw=black] (4,3) rectangle (5,2);\node at(4.5,2.5){\scriptsize\( \bullet \)};
\filldraw[red,draw=black] (5,3) rectangle (6,2);\node at(5.5,2.5){\scriptsize\( \bullet \)};
\draw (0,2) rectangle (1,1);
\draw (1,2) rectangle (2,1);\node at(1.5,1.5){\scriptsize\( 6 \)};
\draw (2,2) rectangle (3,1);
\filldraw[red,draw=black] (3,2) rectangle (4,1);\node at(3.5,1.5){\scriptsize\( \bullet \)};
\filldraw[red,draw=black] (4,2) rectangle (5,1);\node at(4.5,1.5){\scriptsize\( \bullet \)};
\filldraw[red,draw=black] (5,2) rectangle (6,1);\node at(5.5,1.5){\scriptsize\( \bullet \)};
\draw (0,1) rectangle (1,0);
\draw (1,1) rectangle (2,0);
\draw (2,1) rectangle (3,0);\node at(2.5,0.5){\scriptsize\( 6 \)};
\draw (3,1) rectangle (4,0);
\filldraw[red,draw=black] (4,1) rectangle (5,0);\node at(4.5,0.5){\scriptsize\( \bullet \)};
\filldraw[red,draw=black] (5,1) rectangle (6,0);\node at(5.5,0.5){\scriptsize\( \bullet \)};
\draw (0,0) rectangle (1,-1);
\draw (1,0) rectangle (2,-1);
\draw (2,0) rectangle (3,-1);
\draw (3,0) rectangle (4,-1);\node at(3.5,-0.5){\scriptsize\( 6 \)};
\draw (4,0) rectangle (5,-1);
\filldraw[red,draw=black] (5,0) rectangle (6,-1);\node at(5.5,-0.5){\scriptsize\( \bullet \)};
\draw (0,-1) rectangle (1,-2);
\draw (1,-1) rectangle (2,-2);
\draw (2,-1) rectangle (3,-2);
\draw (3,-1) rectangle (4,-2);
\draw (4,-1) rectangle (5,-2);\node at(4.5,-1.5){\scriptsize\( 4 \)};
\draw (5,-1) rectangle (6,-2);
\draw (0,-2) rectangle (1,-3);
\draw (1,-2) rectangle (2,-3);
\draw (2,-2) rectangle (3,-3);
\draw (3,-2) rectangle (4,-3);
\draw (4,-2) rectangle (5,-3);
\draw (5,-2) rectangle (6,-3);\node at(5.5,-2.5){\scriptsize\( 3 \)};
\end{tikzpicture}
\, .
\end{equation*}
Then $$\bott_\lambda = 4, \quad  \bdo{\lambda}{1} = 2, \quad\bdo{\lambda}{2} = 4, \quad\bdo{\lambda}{3} = 5, \quad\bdo{\lambda}{4} = 6,$$
and so
$$L\big(\De{k+1}{\lambda}\big) = L\big(\dkl\big)\setminus \{ 2,4,5,6\} = L\big(\dkl\big)\setminus\{ \bdo{\lambda}{x} \mid x\in[\bott_\lambda] \},$$
which exactly satisfies~(\mref{eq:Dk1toDk}).
\hspace*{\fill}$\square$
\mlabel{exam:twog}
\end{exam}

We now introduce a new subfamily of Katalan functions that simultaneously generalize both the $K$-$k$-Schur functions $g_\lambda^{(k)}$ and the closed $k$-Schur Katalan functions $\fg{\lambda}{k}$.

\begin{defn}
Let $\lambda\in\pkl$. For $z\in[\bott_\lambda+1]$, define the {\em $K$-$k$-Schur function of weight $z$} to be
\begin{equation}
\sg{\lambda}{z}:= K\Big(\dkl;L\big(\dkl\big)\setminus \{ \bdo{\lambda}{x} \mid x\in[z,\bott_\lambda] \};\lambda\Big).
\mlabel{eq:wKk}
\end{equation}
\mlabel{def:semiKk}
\end{defn}

Note that the cases $z > \bott_\lambda + 1$ and $z = \bott_\lambda + 1$ have the same effect in~(\mref{eq:wKk}). Therefore, we restrict to $z \in [\bott_\lambda + 1]$ in the above definition.

\begin{prop}
Let $\lambda\in\pkl$. Then
\begin{equation}
\sg{\lambda}{1} = g_\lambda^{(k)}, \quad \sg{\lambda}{\bott_\lambda+1} = \fg{\lambda}{k}.
\mlabel{eq:sgtofg}
\end{equation}
\mlabel{prop:threeg}
\end{prop}

\begin{proof}
By Definition~\mref{def:semiKk},
\begin{align*}
\sg{\lambda}{1} =&\ K\Big(\dkl;L\big(\dkl\big)\setminus \{ \bdo{\lambda}{x} \mid x\in[\bott_\lambda] \};\lambda\Big) \hspace{1cm} (\text{by $[\bott_\lambda]:= [1,\bott_\lambda]$})\\
=&\ K(\dkl;\De{k+1}{\lambda};\lambda) \hspace{1cm} (\text{by~(\ref{eq:Dk1toDk})})\\
=&\ g_\lambda^{(k)} \hspace{1cm} (\text{by Lemma~\ref{lem:KkSch}}).
\end{align*}
Similarly,
\begin{align*}
\sg{\lambda}{\bott_\lambda+1} =&\ K\Big(\dkl;L\big(\dkl\big)\setminus \{ \bdo{\lambda}{x} \mid x\in[\bott_\lambda+1,\bott_\lambda] \};\lambda\Big) \\
=&\ K(\dkl;\dkl;\lambda) \hspace{1cm} (\text{by $[\bott_\lambda+1,\bott_\lambda] = \emptyset$})\\
=&\ \fg{\lambda}{k} \hspace{1cm} (\text{by Definition~\ref{defn:cKataFunc}}).
\end{align*}
This completes the proof.
\end{proof}

Let us end this section with an example for better understanding.

\begin{exam}
Let
\[
k = 7, \quad \ell = 6, \quad \lambda = (7,6,6,6,4,3)\in{\rm P}_6^7.
 \]
It follows from~Example~\mref{exam:twog} that
\[
\bott_\lambda = 4, \quad \bdo{\lambda}{1} = 2, \quad\bdo{\lambda}{2} = 4, \quad\bdo{\lambda}{3} = 5, \quad\bdo{\lambda}{4} = 6.
\]
Suppose $z = 2$. Then
\begin{align*}
L\big(\dkl\big)\setminus \{ \bdo{\lambda}{x} \mid x\in[z,\bott_\lambda] \}
=&\ \{ 2,3,4,4,5,5,5,6,6,6,6 \}\setminus \{ 4,5,6 \} \hspace{1cm} (\text{by~(\ref{eq:twoL})}) \\
=&\ \{ 2,3,4,5,5,6,6,6 \}.
\end{align*}
Hence, we represent the associated $K$-$k$-Schur function of weight 2 as
\begin{equation*}
\sg{\lambda}{2} =
\begin{tikzpicture}[scale=.4,line width=0.5pt,baseline=(a.base)]
\draw (0,3) rectangle (1,2);\node at(0.5,2.5){\scriptsize\( 7 \)};
\filldraw[red,draw=black] (1,3) rectangle (2,2);\node at(1.5,2.5){\scriptsize\( \bullet \)};
\filldraw[red,draw=black] (2,3) rectangle (3,2);\node at(2.5,2.5){\scriptsize\( \bullet \)};
\filldraw[red,draw=black] (3,3) rectangle (4,2);\node at(3.5,2.5){\scriptsize\( \bullet \)};
\filldraw[red,draw=black] (4,3) rectangle (5,2);\node at(4.5,2.5){\scriptsize\( \bullet \)};
\filldraw[red,draw=black] (5,3) rectangle (6,2);\node at(5.5,2.5){\scriptsize\( \bullet \)};
\draw (0,2) rectangle (1,1);
\draw (1,2) rectangle (2,1);\node at(1.5,1.5){\scriptsize\( 6 \)};
\draw (2,2) rectangle (3,1);
\filldraw[red,draw=black] (3,2) rectangle (4,1);
\filldraw[red,draw=black] (4,2) rectangle (5,1);\node at(4.5,1.5){\scriptsize\( \bullet \)};
\filldraw[red,draw=black] (5,2) rectangle (6,1);\node at(5.5,1.5){\scriptsize\( \bullet \)};
\draw (0,1) rectangle (1,0);
\draw (1,1) rectangle (2,0);
\draw (2,1) rectangle (3,0);\node at(2.5,0.5){\scriptsize\( 6 \)};
\draw (3,1) rectangle (4,0);
\filldraw[red,draw=black] (4,1) rectangle (5,0);
\filldraw[red,draw=black] (5,1) rectangle (6,0);\node at(5.5,0.5){\scriptsize\( \bullet \)};
\draw (0,0) rectangle (1,-1);
\draw (1,0) rectangle (2,-1);
\draw (2,0) rectangle (3,-1);
\draw (3,0) rectangle (4,-1);\node at(3.5,-0.5){\scriptsize\( 6 \)};
\draw (4,0) rectangle (5,-1);
\filldraw[red,draw=black] (5,0) rectangle (6,-1);
\draw (0,-1) rectangle (1,-2);
\draw (1,-1) rectangle (2,-2);
\draw (2,-1) rectangle (3,-2);
\draw (3,-1) rectangle (4,-2);
\draw (4,-1) rectangle (5,-2);\node at(4.5,-1.5){\scriptsize\( 4 \)};
\draw (5,-1) rectangle (6,-2);
\draw (0,-2) rectangle (1,-3);
\draw (1,-2) rectangle (2,-3);
\draw (2,-2) rectangle (3,-3);
\draw (3,-2) rectangle (4,-3);
\draw (4,-2) rectangle (5,-3);
\draw (5,-2) rectangle (6,-3);\node at(5.5,-2.5){\scriptsize\( 3 \)};
\draw[black,line width=0.8pt,->] (6,1.5)--(6.8,1.5);
\node at (7.1,1.5){ \scriptsize{\text{$2$}} };
\end{tikzpicture}
.
\end{equation*}
\hspace*{\fill}$\square$
\end{exam}

\section{A recursive formula for weighted $K$-$k$-Schur functions}
\mlabel{sec:lower}
In this section, motivated by Proposition~\mref{prop:threeg}, we investigate the coefficients $b_{\lambda\mu}$ in Conjecture~\mref{conj:aim} by analyzing how the $K$-$k$-Schur functions $\sg{\lambda}{z+1}$ of weight $z+1$ decompose linearly in terms of $\sg{\mu}{z}$ for $\lambda, \mu \in \pkl$ and $z \in [\bott_\lambda]$ (Proposition~\mref{prop:ind}).
We begin with the following result.

\begin{lemma}
Let $\lambda\in\pkl$ and $z\in[\bott_\lambda]$. Then
\begin{equation*}
\sg{\lambda}{z+1} =\sg{\lambda}{z} - L_{\bdo{\lambda}{z}}\sg{\lambda}{z}.
\end{equation*}
\mlabel{lem:ztoz1}
\end{lemma}
\begin{proof}
We have,
\begin{align*}
\sg{\lambda}{z+1} = &\ K\Big(\dkl;L\big(\dkl\big)\setminus \{ \bdo{\lambda}{x} \mid x\in[z+1,\bott_\lambda] \};\lambda\Big) \hspace{1cm} (\text{by Definition~\ref{def:semiKk}})\\
=&\ K\Big(\dkl;\Big(L\big(\dkl\big)\setminus \{ \bdo{\lambda}{x} \mid x\in[z+1,\bott_\lambda] \}\Big)\setminus \bdo{\lambda}{z};\lambda\Big)\\
& - K\Big(\dkl;\Big(L\big(\dkl\big)\setminus \{ \bdo{\lambda}{x} \mid x\in[z+1,\bott_\lambda] \}\Big)\setminus \bdo{\lambda}{z};\lambda-\epsilon_{\bdo{\lambda}{z}}\Big)\\
& \hspace{8cm} (\text{by Lemma~\ref{lem:relk}~(\ref{it:relk3})})\\
=&\ K\Big(\dkl;L\big(\dkl\big)\setminus \{ \bdo{\lambda}{x} \mid x\in[z,\bott_\lambda] \};\lambda\Big)\\
& - K\Big(\dkl;L\big(\dkl\big)\setminus \{ \bdo{\lambda}{x} \mid x\in[z,\bott_\lambda] \};\lambda-\epsilon_{\bdo{\lambda}{z}}\Big)\\
=&\ K\Big(\dkl;L\big(\dkl\big)\setminus \{ \bdo{\lambda}{x} \mid x\in[z,\bott_\lambda] \};\lambda\Big)\\
& - L_{\bdo{\lambda}{z}}K\Big(\dkl;L\big(\dkl\big)\setminus \{ \bdo{\lambda}{x} \mid x\in[z,\bott_\lambda] \};\lambda\Big)\\
& \hspace{4.5cm} (\text{by Lemma~\ref{lem:LK} for the second summand})\\
=&\ \sg{\lambda}{z} - L_{\bdo{\lambda}{z}}\sg{\lambda}{z} \hspace{1cm} (\text{by Definition~\ref{def:semiKk}}),
\end{align*}
as required.
\end{proof}

\subsection{Lowering operators on weighted $K$-$k$-Schur functions} \mlabel{ss:LsKk}
In this subsection, we analyze the term $-L_{\bdo{\lambda}{z}}\sg{\lambda}{z}$ in Lemma~\mref{lem:ztoz1}. To this end, we generalize $L_{\bdo{\lambda}{z}}\sg{\lambda}{z}$ to $L_{\bDo{\lambda}{a}{z}}\sg{\lambda}{z}$ for any $a \in \ZZ_{\geq 1}$ such that $\bDo{\lambda}{a}{z}$ is defined, and distinguish two cases: $\bDo{\lambda}{a}{z} \in [\bott_\lambda + 1, \ell]$ (addressed in Proposition~\mref{prop:bigd}) and $\bDo{\lambda}{a}{z} \in [\bott_\lambda]$ (treated in Proposition~\mref{prop:smalld}).

\begin{prop}
Let $\lambda \in \pkl$, $z \in [\bott_\lambda]$, and $a \in \ZZ_{\geq 1}$ such that $\bDo{\lambda}{a-1}{z} \in [\bott_\lambda]$ and $\bDo{\lambda}{a}{z} \in [\bott_\lambda + 1, \ell]$. Define $\gamma := \lambda - \epsilon_{\bDo{\lambda}{a}{z}}$.
\begin{enumerate}[label={(\rm \roman*)}]
\item If $\gamma \in \pkl$, then
$
L_{\bDo{\lambda}{a}{z}} \sg{\lambda}{z} = \sg{\gamma}{z}.
$\mlabel{it:bigd1}

\item If $\gamma \notin \pkl$, and assuming $\lambda_z > \lambda_{z+1}$ when $\bDo{\lambda}{a-1}{z} \in [\bott_\lambda - 1]$, then
$
L_{\bDo{\lambda}{a}{z}} \sg{\lambda}{z} = 0.
$\mlabel{it:bigd2}
\end{enumerate}
\mlabel{prop:bigd}
\end{prop}


\begin{proof}
\mref{it:bigd1} Suppose $\gamma= \lambda-\epsilon_{\bDo{\lambda}{a}{z}}\in\pkl$. By $\bDo{\lambda}{a}{z}\in[\bott_\lambda+1,\ell]$,
\begin{equation}
\dkl =\De{k}{\lambda-\epsilon_{\bDo{\lambda}{a}{z}}} = \De{k}{\gamma}, \quad
\bott_\lambda = \bott_\gamma,\quad
\bdo{\lambda}{x} = \bdo{\gamma}{x}, \quad \forall x\in[\bott_\lambda].
\mlabel{eq:bigdC1}
\end{equation}
It follows that
\begin{align*}
L_{\bDo{\lambda}{a}{z}}\sg{\lambda}{z}
=&\ L_{\bDo{\lambda}{a}{z}}K\Big(\dkl;L\big(\dkl\big)\setminus \{ \bdo{\lambda}{x} \mid x\in[z,\bott_\lambda] \};\lambda\Big)\hspace{1cm}
(\text{by Definition~\ref{def:semiKk}}) \\
=&\ K\Big(\dkl;L\big(\dkl\big)\setminus \{ \bdo{\lambda}{x} \mid x\in[z,\bott_\lambda] \};\lambda-\epsilon_{\bDo{\lambda}{a}{z}}\Big)\hspace{1cm} (\text{by Lemma~\ref{lem:LK}})\\
=&\ K\Big(\De{k}{\gamma};
L\big(\De{k}{\gamma}\big)\setminus \{ \bdo{\gamma}{x}
\mid x\in[z,\bott_{\gamma}] \};\gamma\Big)\hspace{2.3cm} (\text{by~(\ref{eq:bigdC1})})\\
=&\ \sg{\gamma}{z} \hspace{3cm} (\text{by Definition~\ref{def:semiKk}}).
\end{align*}

\mref{it:bigd2} Suppose $\gamma = \lambda-\epsilon_{\bDo{\lambda}{a}{z}}\notin\pkl$, that is, $\bDo{\lambda}{a}{z}<\ell$ and $\lambda_{\bDo{\lambda}{a}{z}} = \lambda_{\bDo{\lambda}{a}{z}+1}$.
Consider first $\bDo{\lambda}{a-1}{z} = \bott_\lambda$. Then
$$
\bdo{\lambda}{\bott_\lambda} = \bdo{\lambda}{\bDo{\lambda}{a-1}{z}} = \bDo{\lambda}{a}{z}<\ell,
$$
which implies that
there are at least two roots in row $\bott_\lambda$, and so
there is a ceiling in columns
\begin{equation}
\bDo{\lambda}{a}{z},\quad  \bDo{\lambda}{a}{z}+1 \,\text{ in }\, \dkl.
\mlabel{eq:ceil1}
\end{equation}

Consider next $\bDo{\lambda}{a-1}{z}\in[\bott_\lambda-1]$.
Then $\lambda_{z}>\lambda_{z+1}$ by the assumption hypothesis.
Define $c\in[0,a-1]$ such that
\begin{equation}
\lambda_{\bDo{\lambda}{c}{z}}>\lambda_{\bDo{\lambda}{c}{z}+1} \,\text{ and }\, \lambda_{\bDo{\lambda}{x}{z}} = \lambda_{\bDo{\lambda}{x}{z}+1}, \quad\forall x\in[c+1,a-1].
\mlabel{eq:csmallest}
\end{equation}
The existence of $c$ follows from the inequality $\lambda_z>\lambda_{z+1}$, which ensures that
$c$ can be taken to be at least zero, while its uniqueness is guaranteed by the definition~(\mref{eq:csmallest}).
By Remark~\mref{re:rootfact}~(3), there is a ceiling in columns
\begin{equation}
\bDo{\lambda}{c+1}{z}=\bdo{\lambda}{\bDo{\lambda}{c}{z}}, \quad \bDo{\lambda}{c+1}{z}+1 \,\text{ in }\, \dkl.
\mlabel{eq:ceil2}
\end{equation}
Since
\[
k\geq\lambda_{\bDo{\lambda}{c}{z}}>\lambda_{\bDo{\lambda}{c}{z}+1}
\geq \lambda_{\bDo{\lambda}{c+1}{z}} \geq \lambda_{\bDo{\lambda}{x}{z}},
\]
we have $$k> \lambda_{\bDo{\lambda}{x}{z}} = \lambda_{\bDo{\lambda}{x}{z}+1}.$$
Then, by Remark~\mref{re:rootfact}~(2),
there is a mirror in rows $\bDo{\lambda}{x}{z},\bDo{\lambda}{x}{z}+1$ for each $x\in[c+1,a-1]$. In other words, there is a mirror in rows \begin{equation}
x,\quad x+1,\,\text{ for each }\,  x\in{\rm path}_{\dkl}(\bDo{\lambda}{c+1}{z},\bDo{\lambda}{a-1}{z}).
\mlabel{eq:ceil3}
\end{equation}

By taking $c':= a$ in~(\mref{eq:ceil1}) or $c':= c+1$ in~(\mref{eq:ceil2}) and (\mref{eq:ceil3}), we conclude that there is a (unique) $c'\in[a]$ such that
\begin{enumerate}[label={(\alph*)}]
\item there is a ceiling in columns $\bDo{\lambda}{c'}{z}, \bDo{\lambda}{c'}{z}+1$ in $\dkl$;  \mlabel{it:cei}

\item there is a mirror in rows $x, x+1$ for each $x\in{\rm path}_{\dkl}(\bDo{\lambda}{c'}{z},\bDo{\lambda}{a-1}{z})$. \mlabel{it:mirror}
\end{enumerate}
Note that when $c'= a$, the path ${\rm path}_{\dkl}(\bDo{\lambda}{a}{z},\bDo{\lambda}{a-1}{z})$ is empty, so condition~\mref{it:mirror} holds vacuously.

Next, we are going to apply Mirror Lemma~\ref{lem:mirr1} with
$$\Psi := \dkl, \quad M:=L\big(\dkl\big)\setminus \{ \bdo{\lambda}{x} \mid x\in[z,\bott_\lambda] \}.$$
Indeed,
\begin{enumerate}
\item[(i)] $\dkl$ has a ceiling in columns $\bDo{\lambda}{c'}{z}, \bDo{\lambda}{c'}{z}+1$ for some $c'\in[a]$ by~\mref{it:cei};
\item[(ii)] $\dkl$ has a mirror in rows $x, x+1$ for each $x\in{\rm path}_{\dkl}(\bDo{\lambda}{c'}{z},\bDo{\lambda}{a-1}{z})$ by~\mref{it:mirror};
\item[(iii)] $\dkl$ has a wall in rows $\bDo{\lambda}{a}{z}, \bDo{\lambda}{a}{z}+1$, since there are no roots in rows $\bDo{\lambda}{a}{z}, \bDo{\lambda}{a}{z}+1$ by $\bDo{\lambda}{a}{z}, \bDo{\lambda}{a}{z}+1 > \bott_\lambda$;
\item[(iv)] $\gamma_x = \lambda_x = \lambda_{x+1} = \gamma_{x+1}$ for all $x\in{\rm path}_{\dkl}(\bDo{\lambda}{c'}{z},\bDo{\lambda}{a-1}{z})$;
\item[(v)] $\gamma_{\bDo{\lambda}{a}{z}}+1
= \lambda_{\bDo{\lambda}{a}{z}}-1 +1= \lambda_{\bDo{\lambda}{a}{z}+1}
= \gamma_{\bDo{\lambda}{a}{z}+1}$;
\item[(vi)] $m_M(x)+1 = m_M(x+1)$ for all $x\in{\rm path}_{\dkl}(\bDo{\lambda}{c'+1}{z},\bDo{\lambda}{a}{z})$;
\item[(vii)] $m_M(\bDo{\lambda}{c'}{z})+1 = m_M(\bDo{\lambda}{c'}{z}+1)$.
\end{enumerate}
Thus
\begin{align*}
L_{\bDo{\lambda}{a}{z}}\sg{\lambda}{z}
=&\ L_{\bDo{\lambda}{a}{z}}K\Big(\dkl;L\big(\dkl\big)\setminus \{ \bdo{\lambda}{x} \mid x\in[z,\bott_\lambda] \};\lambda\Big)\hspace{0.5cm}
(\text{by Definition~\ref{def:semiKk}}) \\
=&\ K\Big(\dkl;L\big(\dkl\big)\setminus \{ \bdo{\lambda}{x} \mid x\in[z,\bott_\lambda] \};\gamma\Big)\hspace{1cm} (\text{by Lemma~\ref{lem:LK}})\\
=&\ 0 \hspace{3cm}(\text{by Mirror Lemma~\ref{lem:mirr1})}).
\end{align*}
This completes the proof.
\end{proof}

We now illustrate Proposition~\mref{prop:bigd} with the following example.

\begin{exam}
Let
\[
k=7, \quad \ell=13, \quad \lambda = (7,6,5,5,4,4,4,3,3,3,2,1,1)\in{\rm P}_{13}^7.
\]
Then $\bott_\lambda = 8$. Taking $z=2\in[\bott_\lambda]$ and $a = 3$, we obtain
\begin{align*}
& (1) \qquad \bDo{\lambda}{3}{z} = 11\in [\bott_\lambda+1,\ell],\\
& (2) \qquad \bDo{\lambda}{2}{z} = 7 \in [\bott_\lambda-1] \,\text{ and }\, \lambda_2=6>5=\lambda_3,\\
& (3) \qquad \gamma = \lambda-\epsilon_{\bDo{\lambda}{3}{z}} = (7,6,5,5,4,4,4,3,3,3,\textcolor{green}{1},1,1)\in{\rm P}_{13}^7,
\end{align*}
and so
\begin{equation*}
L_{\bDo{\lambda}{3}{z}}\sg{\lambda}{z}=
L_{11}
\begin{tikzpicture}[scale=.3,line width=0.5pt,baseline=(a.base)]
\draw (0,7) rectangle (1,6);\node at(0.5,6.5){\tiny\( 7 \)};
\filldraw[red,draw=black] (1,7) rectangle (2,6);\node at(1.5,6.5){\tiny\( \bullet \)};
\filldraw[red,draw=black] (2,7) rectangle (3,6);\node at(2.5,6.5){\tiny\( \bullet \)};
\filldraw[red,draw=black] (3,7) rectangle (4,6);\node at(3.5,6.5){\tiny\( \bullet \)};
\filldraw[red,draw=black] (4,7) rectangle (5,6);\node at(4.5,6.5){\tiny\( \bullet \)};
\filldraw[red,draw=black] (5,7) rectangle (6,6);\node at(5.5,6.5){\tiny\( \bullet \)};
\filldraw[red,draw=black] (6,7) rectangle (7,6);\node at(6.5,6.5){\tiny\( \bullet \)};
\filldraw[red,draw=black] (7,7) rectangle (8,6);\node at(7.5,6.5){\tiny\( \bullet \)};
\filldraw[red,draw=black] (8,7) rectangle (9,6);\node at(8.5,6.5){\tiny\( \bullet \)};
\filldraw[red,draw=black] (9,7) rectangle (10,6);\node at(9.5,6.5){\tiny\( \bullet \)};
\filldraw[red,draw=black] (10,7) rectangle (11,6);\node at(10.5,6.5){\tiny\( \bullet \)};
\filldraw[red,draw=black] (11,7) rectangle (12,6);\node at(11.5,6.5){\tiny\( \bullet \)};
\filldraw[red,draw=black] (12,7) rectangle (13,6);\node at(12.5,6.5){\tiny\( \bullet \)};
\draw (0,6) rectangle (1,5);
\draw (1,6) rectangle (2,5);\node at(1.5,5.5){\tiny\( 6 \)};
\draw (2,6) rectangle (3,5);
\filldraw[red,draw=black] (3,6) rectangle (4,5);
\filldraw[red,draw=black] (4,6) rectangle (5,5);\node at(4.5,5.5){\tiny\( \bullet \)};
\filldraw[red,draw=black] (5,6) rectangle (6,5);\node at(5.5,5.5){\tiny\( \bullet \)};
\filldraw[red,draw=black] (6,6) rectangle (7,5);\node at(6.5,5.5){\tiny\( \bullet \)};
\filldraw[red,draw=black] (7,6) rectangle (8,5);\node at(7.5,5.5){\tiny\( \bullet \)};
\filldraw[red,draw=black] (8,6) rectangle (9,5);\node at(8.5,5.5){\tiny\( \bullet \)};
\filldraw[red,draw=black] (9,6) rectangle (10,5);\node at(9.5,5.5){\tiny\( \bullet \)};
\filldraw[red,draw=black] (10,6) rectangle (11,5);\node at(10.5,5.5){\tiny\( \bullet \)};
\filldraw[red,draw=black] (11,6) rectangle (12,5);\node at(11.5,5.5){\tiny\( \bullet \)};
\filldraw[red,draw=black] (12,6) rectangle (13,5);\node at(12.5,5.5){\tiny\( \bullet \)};
\draw (0,5) rectangle (1,4);
\draw (1,5) rectangle (2,4);
\draw (2,5) rectangle (3,4);\node at(2.5,4.5){\tiny\( 5 \)};
\draw (3,5) rectangle (4,4);
\draw (4,5) rectangle (5,4);
\filldraw[red,draw=black] (5,5) rectangle (6,4);
\filldraw[red,draw=black] (6,5) rectangle (7,4);\node at(6.5,4.5){\tiny\( \bullet \)};
\filldraw[red,draw=black] (7,5) rectangle (8,4);\node at(7.5,4.5){\tiny\( \bullet \)};
\filldraw[red,draw=black] (8,5) rectangle (9,4);\node at(8.5,4.5){\tiny\( \bullet \)};
\filldraw[red,draw=black] (9,5) rectangle (10,4);\node at(9.5,4.5){\tiny\( \bullet \)};
\filldraw[red,draw=black] (10,5) rectangle (11,4);\node at(10.5,4.5){\tiny\( \bullet \)};
\filldraw[red,draw=black] (11,5) rectangle (12,4);\node at(11.5,4.5){\tiny\( \bullet \)};
\filldraw[red,draw=black] (12,5) rectangle (13,4);\node at(12.5,4.5){\tiny\( \bullet \)};
\draw (0,4) rectangle (1,3);
\draw (1,4) rectangle (2,3);
\draw (2,4) rectangle (3,3);
\draw (3,4) rectangle (4,3);\node at(3.5,3.5){\tiny\( 5 \)};
\draw (4,4) rectangle (5,3);
\draw (5,4) rectangle (6,3);
\filldraw[red,draw=black] (6,4) rectangle (7,3);
\filldraw[red,draw=black] (7,4) rectangle (8,3);\node at(7.5,3.5){\tiny\( \bullet \)};
\filldraw[red,draw=black] (8,4) rectangle (9,3);\node at(8.5,3.5){\tiny\( \bullet \)};
\filldraw[red,draw=black] (9,4) rectangle (10,3);\node at(9.5,3.5){\tiny\( \bullet \)};
\filldraw[red,draw=black] (10,4) rectangle (11,3);\node at(10.5,3.5){\tiny\( \bullet \)};
\filldraw[red,draw=black] (11,4) rectangle (12,3);\node at(11.5,3.5){\tiny\( \bullet \)};
\filldraw[red,draw=black] (12,4) rectangle (13,3);\node at(12.5,3.5){\tiny\( \bullet \)};
\draw (0,3) rectangle (1,2);
\draw (1,3) rectangle (2,2);
\draw (2,3) rectangle (3,2);
\draw (3,3) rectangle (4,2);
\draw (4,3) rectangle (5,2);\node at(4.5,2.5){\tiny\( 4 \)};
\draw (5,3) rectangle (6,2);
\draw (6,3) rectangle (7,2);
\draw (7,3) rectangle (8,2);
\filldraw[red,draw=black] (8,3) rectangle (9,2);
\filldraw[red,draw=black] (9,3) rectangle (10,2);\node at(9.5,2.5){\tiny\( \bullet \)};
\filldraw[red,draw=black] (10,3) rectangle (11,2);\node at(10.5,2.5){\tiny\( \bullet \)};
\filldraw[red,draw=black] (11,3) rectangle (12,2);\node at(11.5,2.5){\tiny\( \bullet \)};
\filldraw[red,draw=black] (12,3) rectangle (13,2);\node at(12.5,2.5){\tiny\( \bullet \)};
\draw (0,2) rectangle (1,1);
\draw (1,2) rectangle (2,1);
\draw (2,2) rectangle (3,1);
\draw (3,2) rectangle (4,1);
\draw (4,2) rectangle (5,1);
\draw (5,2) rectangle (6,1);\node at(5.5,1.5){\tiny\( 4 \)};
\draw (6,2) rectangle (7,1);
\draw (7,2) rectangle (8,1);
\draw (8,2) rectangle (9,1);
\filldraw[red,draw=black] (9,2) rectangle (10,1);4
\filldraw[red,draw=black] (10,2) rectangle (11,1);\node at(10.5,1.5){\tiny\( \bullet \)};
\filldraw[red,draw=black] (11,2) rectangle (12,1);\node at(11.5,1.5){\tiny\( \bullet \)};
\filldraw[red,draw=black] (12,2) rectangle (13,1);\node at(12.5,1.5){\tiny\( \bullet \)};
\draw (0,1) rectangle (1,0);
\draw (1,1) rectangle (2,0);
\draw (2,1) rectangle (3,0);
\draw (3,1) rectangle (4,0);
\draw (4,1) rectangle (5,0);
\draw (5,1) rectangle (6,0);
\draw (6,1) rectangle (7,0);\node at(6.5,0.5){\tiny\( 4 \)};
\draw (7,1) rectangle (8,0);
\draw (8,1) rectangle (9,0);
\draw (9,1) rectangle (10,0);
\filldraw[red,draw=black] (10,1) rectangle (11,0);
\filldraw[red,draw=black] (11,1) rectangle (12,0);\node at(11.5,0.5){\tiny\( \bullet \)};
\filldraw[red,draw=black] (12,1) rectangle (13,0);\node at(12.5,0.5){\tiny\( \bullet \)};
\draw (0,0) rectangle (1,-1);
\draw (1,0) rectangle (2,-1);
\draw (2,0) rectangle (3,-1);
\draw (3,0) rectangle (4,-1);
\draw (4,0) rectangle (5,-1);
\draw (5,0) rectangle (6,-1);
\draw (6,0) rectangle (7,-1);
\draw (7,0) rectangle (8,-1);\node at(7.5,-0.5){\tiny\( 3 \)};
\draw (8,0) rectangle (9,-1);
\draw (9,0) rectangle (10,-1);
\draw (10,0) rectangle (11,-1);
\draw (11,0) rectangle (12,-1);
\filldraw[red,draw=black] (12,0) rectangle (13,-1);
\draw (0,-1) rectangle (1,-2);
\draw (1,-1) rectangle (2,-2);
\draw (2,-1) rectangle (3,-2);
\draw (3,-1) rectangle (4,-2);
\draw (4,-1) rectangle (5,-2);
\draw (5,-1) rectangle (6,-2);
\draw (6,-1) rectangle (7,-2);
\draw (7,-1) rectangle (8,-2);
\draw (8,-1) rectangle (9,-2);\node at(8.5,-1.5){\tiny\( 3 \)};
\draw (9,-1) rectangle (10,-2);
\draw (10,-1) rectangle (11,-2);
\draw (11,-1) rectangle (12,-2);
\draw (12,-1) rectangle (13,-2);
\draw (0,-2) rectangle (1,-3);
\draw (1,-2) rectangle (2,-3);
\draw (2,-2) rectangle (3,-3);
\draw (3,-2) rectangle (4,-3);
\draw (4,-2) rectangle (5,-3);
\draw (5,-2) rectangle (6,-3);
\draw (6,-2) rectangle (7,-3);
\draw (7,-2) rectangle (8,-3);
\draw (8,-2) rectangle (9,-3);
\draw (9,-2) rectangle (10,-3);\node at(9.5,-2.5){\tiny\( 3 \)};
\draw (10,-2) rectangle (11,-3);
\draw (11,-2) rectangle (12,-3);
\draw (12,-2) rectangle (13,-3);
\draw (0,-3) rectangle (1,-4);
\draw (1,-3) rectangle (2,-4);
\draw (2,-3) rectangle (3,-4);
\draw (3,-3) rectangle (4,-4);
\draw (4,-3) rectangle (5,-4);
\draw (5,-3) rectangle (6,-4);
\draw (6,-3) rectangle (7,-4);
\draw (7,-3) rectangle (8,-4);
\draw (8,-3) rectangle (9,-4);
\draw (9,-3) rectangle (10,-4);
\draw (10,-3) rectangle (11,-4);\node at(10.5,-3.5){\tiny\( 2 \)};
\draw (11,-3) rectangle (12,-4);
\draw (12,-3) rectangle (13,-4);
\draw (0,-4) rectangle (1,-5);
\draw (1,-4) rectangle (2,-5);
\draw (2,-4) rectangle (3,-5);
\draw (3,-4) rectangle (4,-5);
\draw (4,-4) rectangle (5,-5);
\draw (5,-4) rectangle (6,-5);
\draw (6,-4) rectangle (7,-5);
\draw (7,-4) rectangle (8,-5);
\draw (8,-4) rectangle (9,-5);
\draw (9,-4) rectangle (10,-5);
\draw (10,-4) rectangle (11,-5);
\draw (11,-4) rectangle (12,-5);\node at(11.5,-4.5){\tiny\( 1 \)};
\draw (12,-4) rectangle (13,-5);
\draw (0,-5) rectangle (1,-6);
\draw (1,-5) rectangle (2,-6);
\draw (2,-5) rectangle (3,-6);
\draw (3,-5) rectangle (4,-6);
\draw (4,-5) rectangle (5,-6);
\draw (5,-5) rectangle (6,-6);
\draw (6,-5) rectangle (7,-6);
\draw (7,-5) rectangle (8,-6);
\draw (8,-5) rectangle (9,-6);
\draw (9,-5) rectangle (10,-6);
\draw (10,-5) rectangle (11,-6);
\draw (11,-5) rectangle (12,-6);
\draw (12,-5) rectangle (13,-6);\node at(12.5,-5.5){\tiny\( 1 \)};
\draw[purple,line width=0.8pt] (4,3.5)--(6,3.5);
\draw[purple,line width=0.8pt] (3.5,4)--(3.5,5);
\draw[purple,line width=0.8pt] (6.5,3)--(6.5,1);
\draw[purple,line width=0.8pt] (4,3.5)--(6,3.5);
\draw[purple,line width=0.8pt] (7,0.5)--(10,0.5);
%
\draw[black,line width=0.8pt,->] (13,-3.5)--(13.8,-3.5);
\node at (14.2,-3.5){ \tiny{\textbf{$11$}} };
\draw[black,line width=0.8pt,->] (10.5,7)--(10.5,7.8);
\node at (10.5,8.2){ \tiny{\textbf{$11$}} };
\end{tikzpicture}
= \,
\begin{tikzpicture}[scale=.3,line width=0.5pt,baseline=(a.base)]
\draw (0,7) rectangle (1,6);\node at(0.5,6.5){\tiny\( 7 \)};
\filldraw[red,draw=black] (1,7) rectangle (2,6);\node at(1.5,6.5){\tiny\( \bullet \)};
\filldraw[red,draw=black] (2,7) rectangle (3,6);\node at(2.5,6.5){\tiny\( \bullet \)};
\filldraw[red,draw=black] (3,7) rectangle (4,6);\node at(3.5,6.5){\tiny\( \bullet \)};
\filldraw[red,draw=black] (4,7) rectangle (5,6);\node at(4.5,6.5){\tiny\( \bullet \)};
\filldraw[red,draw=black] (5,7) rectangle (6,6);\node at(5.5,6.5){\tiny\( \bullet \)};
\filldraw[red,draw=black] (6,7) rectangle (7,6);\node at(6.5,6.5){\tiny\( \bullet \)};
\filldraw[red,draw=black] (7,7) rectangle (8,6);\node at(7.5,6.5){\tiny\( \bullet \)};
\filldraw[red,draw=black] (8,7) rectangle (9,6);\node at(8.5,6.5){\tiny\( \bullet \)};
\filldraw[red,draw=black] (9,7) rectangle (10,6);\node at(9.5,6.5){\tiny\( \bullet \)};
\filldraw[red,draw=black] (10,7) rectangle (11,6);\node at(10.5,6.5){\tiny\( \bullet \)};
\filldraw[red,draw=black] (11,7) rectangle (12,6);\node at(11.5,6.5){\tiny\( \bullet \)};
\filldraw[red,draw=black] (12,7) rectangle (13,6);\node at(12.5,6.5){\tiny\( \bullet \)};
\draw (0,6) rectangle (1,5);
\draw (1,6) rectangle (2,5);\node at(1.5,5.5){\tiny\( 6 \)};
\draw (2,6) rectangle (3,5);
\filldraw[red,draw=black] (3,6) rectangle (4,5);
\filldraw[red,draw=black] (4,6) rectangle (5,5);\node at(4.5,5.5){\tiny\( \bullet \)};
\filldraw[red,draw=black] (5,6) rectangle (6,5);\node at(5.5,5.5){\tiny\( \bullet \)};
\filldraw[red,draw=black] (6,6) rectangle (7,5);\node at(6.5,5.5){\tiny\( \bullet \)};
\filldraw[red,draw=black] (7,6) rectangle (8,5);\node at(7.5,5.5){\tiny\( \bullet \)};
\filldraw[red,draw=black] (8,6) rectangle (9,5);\node at(8.5,5.5){\tiny\( \bullet \)};
\filldraw[red,draw=black] (9,6) rectangle (10,5);\node at(9.5,5.5){\tiny\( \bullet \)};
\filldraw[red,draw=black] (10,6) rectangle (11,5);\node at(10.5,5.5){\tiny\( \bullet \)};
\filldraw[red,draw=black] (11,6) rectangle (12,5);\node at(11.5,5.5){\tiny\( \bullet \)};
\filldraw[red,draw=black] (12,6) rectangle (13,5);\node at(12.5,5.5){\tiny\( \bullet \)};
\draw (0,5) rectangle (1,4);
\draw (1,5) rectangle (2,4);
\draw (2,5) rectangle (3,4);\node at(2.5,4.5){\tiny\( 5 \)};
\draw (3,5) rectangle (4,4);
\draw (4,5) rectangle (5,4);
\filldraw[red,draw=black] (5,5) rectangle (6,4);
\filldraw[red,draw=black] (6,5) rectangle (7,4);\node at(6.5,4.5){\tiny\( \bullet \)};
\filldraw[red,draw=black] (7,5) rectangle (8,4);\node at(7.5,4.5){\tiny\( \bullet \)};
\filldraw[red,draw=black] (8,5) rectangle (9,4);\node at(8.5,4.5){\tiny\( \bullet \)};
\filldraw[red,draw=black] (9,5) rectangle (10,4);\node at(9.5,4.5){\tiny\( \bullet \)};
\filldraw[red,draw=black] (10,5) rectangle (11,4);\node at(10.5,4.5){\tiny\( \bullet \)};
\filldraw[red,draw=black] (11,5) rectangle (12,4);\node at(11.5,4.5){\tiny\( \bullet \)};
\filldraw[red,draw=black] (12,5) rectangle (13,4);\node at(12.5,4.5){\tiny\( \bullet \)};
\draw (0,4) rectangle (1,3);
\draw (1,4) rectangle (2,3);
\draw (2,4) rectangle (3,3);
\draw (3,4) rectangle (4,3);\node at(3.5,3.5){\tiny\( 5 \)};
\draw (4,4) rectangle (5,3);
\draw (5,4) rectangle (6,3);
\filldraw[red,draw=black] (6,4) rectangle (7,3);
\filldraw[red,draw=black] (7,4) rectangle (8,3);\node at(7.5,3.5){\tiny\( \bullet \)};
\filldraw[red,draw=black] (8,4) rectangle (9,3);\node at(8.5,3.5){\tiny\( \bullet \)};
\filldraw[red,draw=black] (9,4) rectangle (10,3);\node at(9.5,3.5){\tiny\( \bullet \)};
\filldraw[red,draw=black] (10,4) rectangle (11,3);\node at(10.5,3.5){\tiny\( \bullet \)};
\filldraw[red,draw=black] (11,4) rectangle (12,3);\node at(11.5,3.5){\tiny\( \bullet \)};
\filldraw[red,draw=black] (12,4) rectangle (13,3);\node at(12.5,3.5){\tiny\( \bullet \)};
\draw (0,3) rectangle (1,2);
\draw (1,3) rectangle (2,2);
\draw (2,3) rectangle (3,2);
\draw (3,3) rectangle (4,2);
\draw (4,3) rectangle (5,2);\node at(4.5,2.5){\tiny\( 4 \)};
\draw (5,3) rectangle (6,2);
\draw (6,3) rectangle (7,2);
\draw (7,3) rectangle (8,2);
\filldraw[red,draw=black] (8,3) rectangle (9,2);
\filldraw[red,draw=black] (9,3) rectangle (10,2);\node at(9.5,2.5){\tiny\( \bullet \)};
\filldraw[red,draw=black] (10,3) rectangle (11,2);\node at(10.5,2.5){\tiny\( \bullet \)};
\filldraw[red,draw=black] (11,3) rectangle (12,2);\node at(11.5,2.5){\tiny\( \bullet \)};
\filldraw[red,draw=black] (12,3) rectangle (13,2);\node at(12.5,2.5){\tiny\( \bullet \)};
\draw (0,2) rectangle (1,1);
\draw (1,2) rectangle (2,1);
\draw (2,2) rectangle (3,1);
\draw (3,2) rectangle (4,1);
\draw (4,2) rectangle (5,1);
\draw (5,2) rectangle (6,1);\node at(5.5,1.5){\tiny\( 4 \)};
\draw (6,2) rectangle (7,1);
\draw (7,2) rectangle (8,1);
\draw (8,2) rectangle (9,1);
\filldraw[red,draw=black] (9,2) rectangle (10,1);4
\filldraw[red,draw=black] (10,2) rectangle (11,1);\node at(10.5,1.5){\tiny\( \bullet \)};
\filldraw[red,draw=black] (11,2) rectangle (12,1);\node at(11.5,1.5){\tiny\( \bullet \)};
\filldraw[red,draw=black] (12,2) rectangle (13,1);\node at(12.5,1.5){\tiny\( \bullet \)};
\draw (0,1) rectangle (1,0);
\draw (1,1) rectangle (2,0);
\draw (2,1) rectangle (3,0);
\draw (3,1) rectangle (4,0);
\draw (4,1) rectangle (5,0);
\draw (5,1) rectangle (6,0);
\draw (6,1) rectangle (7,0);\node at(6.5,0.5){\tiny\( 4 \)};
\draw (7,1) rectangle (8,0);
\draw (8,1) rectangle (9,0);
\draw (9,1) rectangle (10,0);
\filldraw[red,draw=black] (10,1) rectangle (11,0);
\filldraw[red,draw=black] (11,1) rectangle (12,0);\node at(11.5,0.5){\tiny\( \bullet \)};
\filldraw[red,draw=black] (12,1) rectangle (13,0);\node at(12.5,0.5){\tiny\( \bullet \)};
\draw (0,0) rectangle (1,-1);
\draw (1,0) rectangle (2,-1);
\draw (2,0) rectangle (3,-1);
\draw (3,0) rectangle (4,-1);
\draw (4,0) rectangle (5,-1);
\draw (5,0) rectangle (6,-1);
\draw (6,0) rectangle (7,-1);
\draw (7,0) rectangle (8,-1);\node at(7.5,-0.5){\tiny\( 3 \)};
\draw (8,0) rectangle (9,-1);
\draw (9,0) rectangle (10,-1);
\draw (10,0) rectangle (11,-1);
\draw (11,0) rectangle (12,-1);
\filldraw[red,draw=black] (12,0) rectangle (13,-1);
\draw (0,-1) rectangle (1,-2);
\draw (1,-1) rectangle (2,-2);
\draw (2,-1) rectangle (3,-2);
\draw (3,-1) rectangle (4,-2);
\draw (4,-1) rectangle (5,-2);
\draw (5,-1) rectangle (6,-2);
\draw (6,-1) rectangle (7,-2);
\draw (7,-1) rectangle (8,-2);
\draw (8,-1) rectangle (9,-2);\node at(8.5,-1.5){\tiny\( 3 \)};
\draw (9,-1) rectangle (10,-2);
\draw (10,-1) rectangle (11,-2);
\draw (11,-1) rectangle (12,-2);
\draw (12,-1) rectangle (13,-2);
\draw (0,-2) rectangle (1,-3);
\draw (1,-2) rectangle (2,-3);
\draw (2,-2) rectangle (3,-3);
\draw (3,-2) rectangle (4,-3);
\draw (4,-2) rectangle (5,-3);
\draw (5,-2) rectangle (6,-3);
\draw (6,-2) rectangle (7,-3);
\draw (7,-2) rectangle (8,-3);
\draw (8,-2) rectangle (9,-3);
\draw (9,-2) rectangle (10,-3);\node at(9.5,-2.5){\tiny\( 3 \)};
\draw (10,-2) rectangle (11,-3);
\draw (11,-2) rectangle (12,-3);
\draw (12,-2) rectangle (13,-3);
\draw (0,-3) rectangle (1,-4);
\draw (1,-3) rectangle (2,-4);
\draw (2,-3) rectangle (3,-4);
\draw (3,-3) rectangle (4,-4);
\draw (4,-3) rectangle (5,-4);
\draw (5,-3) rectangle (6,-4);
\draw (6,-3) rectangle (7,-4);
\draw (7,-3) rectangle (8,-4);
\draw (8,-3) rectangle (9,-4);
\draw (9,-3) rectangle (10,-4);
\draw (10,-3) rectangle (11,-4);\node at(10.5,-3.5){\tiny\( \textcolor{green}{1} \)};
\draw (11,-3) rectangle (12,-4);
\draw (12,-3) rectangle (13,-4);
\draw (0,-4) rectangle (1,-5);
\draw (1,-4) rectangle (2,-5);
\draw (2,-4) rectangle (3,-5);
\draw (3,-4) rectangle (4,-5);
\draw (4,-4) rectangle (5,-5);
\draw (5,-4) rectangle (6,-5);
\draw (6,-4) rectangle (7,-5);
\draw (7,-4) rectangle (8,-5);
\draw (8,-4) rectangle (9,-5);
\draw (9,-4) rectangle (10,-5);
\draw (10,-4) rectangle (11,-5);
\draw (11,-4) rectangle (12,-5);\node at(11.5,-4.5){\tiny\( 1 \)};
\draw (12,-4) rectangle (13,-5);
\draw (0,-5) rectangle (1,-6);
\draw (1,-5) rectangle (2,-6);
\draw (2,-5) rectangle (3,-6);
\draw (3,-5) rectangle (4,-6);
\draw (4,-5) rectangle (5,-6);
\draw (5,-5) rectangle (6,-6);
\draw (6,-5) rectangle (7,-6);
\draw (7,-5) rectangle (8,-6);
\draw (8,-5) rectangle (9,-6);
\draw (9,-5) rectangle (10,-6);
\draw (10,-5) rectangle (11,-6);
\draw (11,-5) rectangle (12,-6);
\draw (12,-5) rectangle (13,-6);\node at(12.5,-5.5){\tiny\( 1 \)};
\end{tikzpicture}
\, .
\end{equation*}
For $\lambda = (7,6,5,5,4,4,4,3,3,3,2,\textcolor{red}{2},1)\in{\rm P}_{13}^7$, $$\gamma = \lambda-\epsilon_{11} = (7,6,5,5,4,4,4,3,3,3,\textcolor{green}{1},2,1)\notin{\rm P}_{13}^7.$$
Thus
\begin{equation*}
L_{\bDo{\lambda}{3}{z}}\sg{\lambda}{z}=
L_{11}
\begin{tikzpicture}[scale=.3,line width=0.5pt,baseline=(a.base)]
\draw (0,7) rectangle (1,6);\node at(0.5,6.5){\tiny\( 7 \)};
\filldraw[red,draw=black] (1,7) rectangle (2,6);\node at(1.5,6.5){\tiny\( \bullet \)};
\filldraw[red,draw=black] (2,7) rectangle (3,6);\node at(2.5,6.5){\tiny\( \bullet \)};
\filldraw[red,draw=black] (3,7) rectangle (4,6);\node at(3.5,6.5){\tiny\( \bullet \)};
\filldraw[red,draw=black] (4,7) rectangle (5,6);\node at(4.5,6.5){\tiny\( \bullet \)};
\filldraw[red,draw=black] (5,7) rectangle (6,6);\node at(5.5,6.5){\tiny\( \bullet \)};
\filldraw[red,draw=black] (6,7) rectangle (7,6);\node at(6.5,6.5){\tiny\( \bullet \)};
\filldraw[red,draw=black] (7,7) rectangle (8,6);\node at(7.5,6.5){\tiny\( \bullet \)};
\filldraw[red,draw=black] (8,7) rectangle (9,6);\node at(8.5,6.5){\tiny\( \bullet \)};
\filldraw[red,draw=black] (9,7) rectangle (10,6);\node at(9.5,6.5){\tiny\( \bullet \)};
\filldraw[red,draw=black] (10,7) rectangle (11,6);\node at(10.5,6.5){\tiny\( \bullet \)};
\filldraw[red,draw=black] (11,7) rectangle (12,6);\node at(11.5,6.5){\tiny\( \bullet \)};
\filldraw[red,draw=black] (12,7) rectangle (13,6);\node at(12.5,6.5){\tiny\( \bullet \)};
\draw (0,6) rectangle (1,5);
\draw (1,6) rectangle (2,5);\node at(1.5,5.5){\tiny\( 6 \)};
\draw (2,6) rectangle (3,5);
\filldraw[red,draw=black] (3,6) rectangle (4,5);
\filldraw[red,draw=black] (4,6) rectangle (5,5);\node at(4.5,5.5){\tiny\( \bullet \)};
\filldraw[red,draw=black] (5,6) rectangle (6,5);\node at(5.5,5.5){\tiny\( \bullet \)};
\filldraw[red,draw=black] (6,6) rectangle (7,5);\node at(6.5,5.5){\tiny\( \bullet \)};
\filldraw[red,draw=black] (7,6) rectangle (8,5);\node at(7.5,5.5){\tiny\( \bullet \)};
\filldraw[red,draw=black] (8,6) rectangle (9,5);\node at(8.5,5.5){\tiny\( \bullet \)};
\filldraw[red,draw=black] (9,6) rectangle (10,5);\node at(9.5,5.5){\tiny\( \bullet \)};
\filldraw[red,draw=black] (10,6) rectangle (11,5);\node at(10.5,5.5){\tiny\( \bullet \)};
\filldraw[red,draw=black] (11,6) rectangle (12,5);\node at(11.5,5.5){\tiny\( \bullet \)};
\filldraw[red,draw=black] (12,6) rectangle (13,5);\node at(12.5,5.5){\tiny\( \bullet \)};
\draw (0,5) rectangle (1,4);
\draw (1,5) rectangle (2,4);
\draw (2,5) rectangle (3,4);\node at(2.5,4.5){\tiny\( 5 \)};
\draw (3,5) rectangle (4,4);
\draw (4,5) rectangle (5,4);
\filldraw[red,draw=black] (5,5) rectangle (6,4);
\filldraw[red,draw=black] (6,5) rectangle (7,4);\node at(6.5,4.5){\tiny\( \bullet \)};
\filldraw[red,draw=black] (7,5) rectangle (8,4);\node at(7.5,4.5){\tiny\( \bullet \)};
\filldraw[red,draw=black] (8,5) rectangle (9,4);\node at(8.5,4.5){\tiny\( \bullet \)};
\filldraw[red,draw=black] (9,5) rectangle (10,4);\node at(9.5,4.5){\tiny\( \bullet \)};
\filldraw[red,draw=black] (10,5) rectangle (11,4);\node at(10.5,4.5){\tiny\( \bullet \)};
\filldraw[red,draw=black] (11,5) rectangle (12,4);\node at(11.5,4.5){\tiny\( \bullet \)};
\filldraw[red,draw=black] (12,5) rectangle (13,4);\node at(12.5,4.5){\tiny\( \bullet \)};
\draw (0,4) rectangle (1,3);
\draw (1,4) rectangle (2,3);
\draw (2,4) rectangle (3,3);
\draw (3,4) rectangle (4,3);\node at(3.5,3.5){\tiny\( 5 \)};
\draw (4,4) rectangle (5,3);
\draw (5,4) rectangle (6,3);
\filldraw[red,draw=black] (6,4) rectangle (7,3);
\filldraw[red,draw=black] (7,4) rectangle (8,3);\node at(7.5,3.5){\tiny\( \bullet \)};
\filldraw[red,draw=black] (8,4) rectangle (9,3);\node at(8.5,3.5){\tiny\( \bullet \)};
\filldraw[red,draw=black] (9,4) rectangle (10,3);\node at(9.5,3.5){\tiny\( \bullet \)};
\filldraw[red,draw=black] (10,4) rectangle (11,3);\node at(10.5,3.5){\tiny\( \bullet \)};
\filldraw[red,draw=black] (11,4) rectangle (12,3);\node at(11.5,3.5){\tiny\( \bullet \)};
\filldraw[red,draw=black] (12,4) rectangle (13,3);\node at(12.5,3.5){\tiny\( \bullet \)};
\draw (0,3) rectangle (1,2);
\draw (1,3) rectangle (2,2);
\draw (2,3) rectangle (3,2);
\draw (3,3) rectangle (4,2);
\draw (4,3) rectangle (5,2);\node at(4.5,2.5){\tiny\( 4 \)};
\draw (5,3) rectangle (6,2);
\draw (6,3) rectangle (7,2);
\draw (7,3) rectangle (8,2);
\filldraw[red,draw=black] (8,3) rectangle (9,2);
\filldraw[red,draw=black] (9,3) rectangle (10,2);\node at(9.5,2.5){\tiny\( \bullet \)};
\filldraw[red,draw=black] (10,3) rectangle (11,2);\node at(10.5,2.5){\tiny\( \bullet \)};
\filldraw[red,draw=black] (11,3) rectangle (12,2);\node at(11.5,2.5){\tiny\( \bullet \)};
\filldraw[red,draw=black] (12,3) rectangle (13,2);\node at(12.5,2.5){\tiny\( \bullet \)};
\draw (0,2) rectangle (1,1);
\draw (1,2) rectangle (2,1);
\draw (2,2) rectangle (3,1);
\draw (3,2) rectangle (4,1);
\draw (4,2) rectangle (5,1);
\draw (5,2) rectangle (6,1);\node at(5.5,1.5){\tiny\( 4 \)};
\draw (6,2) rectangle (7,1);
\draw (7,2) rectangle (8,1);
\draw (8,2) rectangle (9,1);
\filldraw[red,draw=black] (9,2) rectangle (10,1);4
\filldraw[red,draw=black] (10,2) rectangle (11,1);\node at(10.5,1.5){\tiny\( \bullet \)};
\filldraw[red,draw=black] (11,2) rectangle (12,1);\node at(11.5,1.5){\tiny\( \bullet \)};
\filldraw[red,draw=black] (12,2) rectangle (13,1);\node at(12.5,1.5){\tiny\( \bullet \)};
\draw (0,1) rectangle (1,0);
\draw (1,1) rectangle (2,0);
\draw (2,1) rectangle (3,0);
\draw (3,1) rectangle (4,0);
\draw (4,1) rectangle (5,0);
\draw (5,1) rectangle (6,0);
\draw (6,1) rectangle (7,0);\node at(6.5,0.5){\tiny\( 4 \)};
\draw (7,1) rectangle (8,0);
\draw (8,1) rectangle (9,0);
\draw (9,1) rectangle (10,0);
\filldraw[red,draw=black] (10,1) rectangle (11,0);
\filldraw[red,draw=black] (11,1) rectangle (12,0);\node at(11.5,0.5){\tiny\( \bullet \)};
\filldraw[red,draw=black] (12,1) rectangle (13,0);\node at(12.5,0.5){\tiny\( \bullet \)};
\draw (0,0) rectangle (1,-1);
\draw (1,0) rectangle (2,-1);
\draw (2,0) rectangle (3,-1);
\draw (3,0) rectangle (4,-1);
\draw (4,0) rectangle (5,-1);
\draw (5,0) rectangle (6,-1);
\draw (6,0) rectangle (7,-1);
\draw (7,0) rectangle (8,-1);\node at(7.5,-0.5){\tiny\( 3 \)};
\draw (8,0) rectangle (9,-1);
\draw (9,0) rectangle (10,-1);
\draw (10,0) rectangle (11,-1);
\draw (11,0) rectangle (12,-1);
\filldraw[red,draw=black] (12,0) rectangle (13,-1);
\draw (0,-1) rectangle (1,-2);
\draw (1,-1) rectangle (2,-2);
\draw (2,-1) rectangle (3,-2);
\draw (3,-1) rectangle (4,-2);
\draw (4,-1) rectangle (5,-2);
\draw (5,-1) rectangle (6,-2);
\draw (6,-1) rectangle (7,-2);
\draw (7,-1) rectangle (8,-2);
\draw (8,-1) rectangle (9,-2);\node at(8.5,-1.5){\tiny\( 3 \)};
\draw (9,-1) rectangle (10,-2);
\draw (10,-1) rectangle (11,-2);
\draw (11,-1) rectangle (12,-2);
\draw (12,-1) rectangle (13,-2);
\draw (0,-2) rectangle (1,-3);
\draw (1,-2) rectangle (2,-3);
\draw (2,-2) rectangle (3,-3);
\draw (3,-2) rectangle (4,-3);
\draw (4,-2) rectangle (5,-3);
\draw (5,-2) rectangle (6,-3);
\draw (6,-2) rectangle (7,-3);
\draw (7,-2) rectangle (8,-3);
\draw (8,-2) rectangle (9,-3);
\draw (9,-2) rectangle (10,-3);\node at(9.5,-2.5){\tiny\( 3 \)};
\draw (10,-2) rectangle (11,-3);
\draw (11,-2) rectangle (12,-3);
\draw (12,-2) rectangle (13,-3);
\draw (0,-3) rectangle (1,-4);
\draw (1,-3) rectangle (2,-4);
\draw (2,-3) rectangle (3,-4);
\draw (3,-3) rectangle (4,-4);
\draw (4,-3) rectangle (5,-4);
\draw (5,-3) rectangle (6,-4);
\draw (6,-3) rectangle (7,-4);
\draw (7,-3) rectangle (8,-4);
\draw (8,-3) rectangle (9,-4);
\draw (9,-3) rectangle (10,-4);
\draw (10,-3) rectangle (11,-4);\node at(10.5,-3.5){\tiny\( 2 \)};
\draw (11,-3) rectangle (12,-4);
\draw (12,-3) rectangle (13,-4);
\draw (0,-4) rectangle (1,-5);
\draw (1,-4) rectangle (2,-5);
\draw (2,-4) rectangle (3,-5);
\draw (3,-4) rectangle (4,-5);
\draw (4,-4) rectangle (5,-5);
\draw (5,-4) rectangle (6,-5);
\draw (6,-4) rectangle (7,-5);
\draw (7,-4) rectangle (8,-5);
\draw (8,-4) rectangle (9,-5);
\draw (9,-4) rectangle (10,-5);
\draw (10,-4) rectangle (11,-5);
\draw (11,-4) rectangle (12,-5);\node at(11.5,-4.5){\tiny\( 2 \)};
\draw (12,-4) rectangle (13,-5);
\draw (0,-5) rectangle (1,-6);
\draw (1,-5) rectangle (2,-6);
\draw (2,-5) rectangle (3,-6);
\draw (3,-5) rectangle (4,-6);
\draw (4,-5) rectangle (5,-6);
\draw (5,-5) rectangle (6,-6);
\draw (6,-5) rectangle (7,-6);
\draw (7,-5) rectangle (8,-6);
\draw (8,-5) rectangle (9,-6);
\draw (9,-5) rectangle (10,-6);
\draw (10,-5) rectangle (11,-6);
\draw (11,-5) rectangle (12,-6);
\draw (12,-5) rectangle (13,-6);\node at(12.5,-5.5){\tiny\( 1 \)};
\draw[purple,line width=0.8pt] (4,3.5)--(6,3.5);
\draw[purple,line width=0.8pt] (3.5,4)--(3.5,5);
\draw[purple,line width=0.8pt] (6.5,3)--(6.5,1);
\draw[purple,line width=0.8pt] (4,3.5)--(6,3.5);
\draw[purple,line width=0.8pt] (7,0.5)--(10,0.5);
%
\draw[black,line width=0.8pt,->] (13,-3.5)--(13.8,-3.5);
\node at (14.2,-3.5){ \tiny{\textbf{$11$}} };
\draw[black,line width=0.8pt,->] (10.5,7)--(10.5,7.8);
\node at (10.5,8.2){ \tiny{\textbf{$11$}} };
\end{tikzpicture}
= 0.
\end{equation*}
\hspace*{\fill}$\square$
\end{exam}

By replacing the assumption $\bDo{\lambda}{a}{z} \in [\bott_\lambda + 1, \ell]$ in Proposition~\mref{prop:bigd} with $\bDo{\lambda}{a}{z} \in [\bott_\lambda]$, we obtain the following result.

\begin{prop}
Let $\lambda\in\pkl$, $z\in[\bott_\lambda-1]$, and $a\in\ZZ_{\geq1}$ such that $\bDo{\lambda}{a}{z}\in[\bott_\lambda]$.
Denote $\gamma:= \lambda-\epsilon_{\bDo{\lambda}{a}{z}}$.
\begin{enumerate}[label=(\roman*)]
\item If $\gamma\in\pkl$, then
$$
L_{\bDo{\lambda}{a}{z}}\sg{\lambda}{z} = \big(1-L_{\bDo{\gamma}{a+1}{z}}\big)\sg{\gamma}{z} + L_{\bDo{\lambda}{a+1}{z}}\sg{\lambda}{z}.
$$
Here we use the convention that $L_{\bDo{\gamma}{a+1}{z}} := 0$ when $\bDo{\gamma}{a+1}{z}$ is undefined.
\mlabel{it:smalld1}

\item If $\gamma\notin\pkl$, and assuming $\lambda_{z}>\lambda_{z+1}$, then
$$
L_{\bDo{\lambda}{a}{z}}\sg{\lambda}{z} = L_{\bDo{\lambda}{a+1}{z}}\sg{\lambda}{z}.
$$
\mlabel{it:smalld2}
\end{enumerate}
\mlabel{prop:smalld}
\end{prop}

\begin{proof}
Denote
\begin{equation}
\beta:= \big(\bDo{\lambda}{a}{z},\bDo{\lambda}{a+1}{z}\big).
\mlabel{eq:beta}
\end{equation}
Then $\beta\in\dkl$ and $\beta$ is removable to $\dkl$ by Remark~\mref{re:rootfact}~(1).
Hence,
\begin{align}
L_{\bDo{\lambda}{a}{z}}\sg{\lambda}{z} =&\ L_{\bDo{\lambda}{a}{z}}K\Big(\dkl;L\big(\dkl\big)\setminus \{ \bdo{\lambda}{x} \mid x\in[z,\bott_\lambda] \};\lambda\Big)\hspace{1cm}
(\text{by Definition~\ref{def:semiKk}}) \notag\\
=&\ K\Big(\dkl;L\big(\dkl\big)\setminus \{ \bdo{\lambda}{x} \mid x\in[z,\bott_\lambda] \};\lambda-\epsilon_{\bDo{\lambda}{a}{z}}\Big)\hspace{1cm} (\text{by Lemma~\ref{lem:LK}})\notag\\
=&\ K\Big(\dkl\setminus \beta;L\big(\dkl\big)\setminus \{ \bdo{\lambda}{x} \mid x\in[z,\bott_\lambda] \};\lambda-\epsilon_{\bDo{\lambda}{a}{z}}\Big)\notag\\
&\ + K\Big(\dkl;L\big(\dkl\big)\setminus \{ \bdo{\lambda}{x} \mid x\in[z,\bott_\lambda] \};\lambda-\epsilon_{\bDo{\lambda}{a}{z}}+\varepsilon_{\beta}\Big)
\hspace{1cm} (\text{by Lemma~\ref{lem:relk}~(\ref{it:relk1})})\notag\\
=&\ K\Big(\dkl\setminus \beta;L\big(\dkl\big)\setminus \{ \bdo{\lambda}{x} \mid x\in[z,\bott_\lambda] \};\gamma\Big)\notag\\
&\ + K\Big(\dkl;L\big(\dkl\big)\setminus \{ \bdo{\lambda}{x} \mid x\in[z,\bott_\lambda] \};\lambda-\epsilon_{\bDo{\lambda}{a}{z}}+\varepsilon_{\beta}\Big)\notag\\
& \hspace{4cm} (\text{by $\gamma := \lambda-\epsilon_{\bDo{\lambda}{a}{z}}$ for the first summand})\notag\\
=&\ K\Big(\dkl\setminus \beta;L\big(\dkl\big)\setminus \{ \bdo{\lambda}{x} \mid x\in[z,\bott_\lambda] \};\gamma\Big)\notag\\
&\ + K\Big(\dkl;L\big(\dkl\big)\setminus \{ \bdo{\lambda}{x} \mid x\in[z,\bott_\lambda] \};\lambda-\epsilon_{\bDo{\lambda}{a+1}{z}}\Big)\notag\\
& \hspace{2cm} \big(\text{by~(\ref{eq:beta}) and so $\varepsilon_{\beta} = \epsilon_{\bDo{\lambda}{a}{z}} - \epsilon_{\bDo{\lambda}{a+1}{z}} $ for the second summand}\big)\notag\\
=&\ K\Big(\dkl\setminus \beta;L\big(\dkl\big)\setminus \{ \bdo{\lambda}{x} \mid x\in[z,\bott_\lambda] \};\gamma\Big) + L_{\bDo{\lambda}{a+1}{z}} \sg{\lambda}{z}\label{eq:smalldini}\\
& \hspace{3.5cm} (\text{by Lemma~\ref{lem:LK} and Definition~\ref{def:semiKk} for the second summand}).\notag
\end{align}

\mref{it:smalld1} Suppose $\gamma = \lambda-\epsilon_{\bDo{\lambda}{a}{z}}\in\pkl$. We divide the proof into the following two cases.

{\bf Case~1.} $\bDo{\lambda}{a+1}{z}+1\leq\ell$. In this case, we have
\begin{equation}
\dkl \setminus \beta \overset{(\ref{eq:beta})}{=} \De{k}{\lambda-\epsilon_{\bDo{\lambda}{a}{z}}} = \dkga,
\mlabel{eq:smalldDel}
\end{equation}
\begin{equation}
\bDo{\lambda}{a+1}{z}+1 = \bDo{\gamma}{a+1}{z},
\mlabel{eq:smalldC12}
\end{equation}
\begin{equation}
\Big(L\big(\dkl\big)\setminus \{ \bdo{\lambda}{x} \mid x\in[z,\bott_\lambda] \}\Big)\setminus\big(\bDo{\lambda}{a+1}{z}+1\big)
\xlongequal{(\ref{eq:smalldDel}),\eqref{eq:smalldC12}}
L\big(\dkga\big)\setminus \{ \bdo{\gamma}{x} \mid x\in[z,\bott_\gamma] \},
\label{eq:smalldC11}
\end{equation}
whence
\begin{align*}
& K\Big(\dkl\setminus \beta;L\big(\dkl\big)\setminus \{ \bdo{\lambda}{x} \mid x\in[z,\bott_\lambda] \};\gamma\Big)\\
=&\ K\Big(\De{k}{\gamma};L\big(\dkl\big)\setminus \{ \bdo{\lambda}{x} \mid x\in[z,\bott_\lambda] \};\gamma\Big)\hspace{1cm} (\text{by~(\ref{eq:smalldDel})})\\
=&\ K\Big(\De{k}{\gamma};\Big(L\big(\dkl\big)\setminus \{ \bdo{\lambda}{x} \mid x\in[z,\bott_\lambda] \}\Big)\setminus\big(\bDo{\lambda}{a+1}{z}+1\big);\gamma\Big)\\
&\ - K\Big(\De{k}{\gamma};\Big(L\big(\dkl\big)\setminus \{ \bdo{\lambda}{x} \mid x\in[z,\bott_\lambda] \}\Big)\setminus\big(\bDo{\lambda}{a+1}{z}+1\big);
\gamma-\epsilon_{\bDo{\lambda}{a+1}{z}+1}\Big)\\
& \hspace{8.5cm} (\text{by Lemma~\ref{lem:relk}~(\ref{it:relk3})})\\
=&\ K\Big(\De{k}{\gamma};L\big(\dkga\big)\setminus \{ \bdo{\gamma}{x} \mid x\in[z,\bott_\gamma] \};\gamma\Big)\\
&\ - K\Big(\De{k}{\gamma};L\big(\dkga\big)\setminus \{ \bdo{\gamma}{x} \mid x\in[z,\bott_\gamma] \};
\gamma-\epsilon_{\bDo{\lambda}{a+1}{z}+1}\Big)\hspace{1cm}
(\text{by~(\ref{eq:smalldC11})})\\
=&\ K\Big(\De{k}{\gamma};L\big(\dkga\big)\setminus \{ \bdo{\gamma}{x} \mid x\in[z,\bott_\gamma] \};\gamma\Big)\\
&\ - K\Big(\De{k}{\gamma};L\big(\dkga\big)\setminus \{ \bdo{\gamma}{x} \mid x\in[z,\bott_\gamma] \};
\gamma-\epsilon_{\bDo{\gamma}{a+1}{z}}\Big)\\
& \hspace{3cm}
(\text{by~(\ref{eq:smalldC12}) for $\epsilon_{\bDo{\lambda}{a+1}{z}+1}$ in the second summand})\\
=&\ K\Big(\De{k}{\gamma};L\big(\dkga\big)\setminus \{ \bdo{\gamma}{x} \mid x\in[z,\bott_\gamma] \};\gamma\Big)\\
&\ -L_{\bDo{\gamma}{a+1}{z}} K\Big(\De{k}{\gamma};L\big(\dkga\big)\setminus \{ \bdo{\gamma}{x} \mid x\in[z,\bott_\gamma] \};
\gamma\Big)\\
& \hspace{4cm} (\text{by Lemma~\ref{lem:LK} for the second summand})\\
=&\ \big(1-L_{\bDo{\gamma}{a+1}{z}}\big)\sg{\gamma}{z} \hspace{1cm} (\text{by Definition~\ref{def:semiKk}}).
\end{align*}
Substituting the above equation into~(\mref{eq:smalldini}) gives the required result.

{\bf Case~2.} $\bDo{\lambda}{a+1}{z}+1 > \ell$. Then $\bott_\lambda = \bDo{\lambda}{a}{z}$ and row $\bott_\lambda$ has exactly one root $\beta =\big(\bDo{\lambda}{a}{z}, \bDo{\lambda}{a+1}{z}\big)$ in the root ideal $\dkl$. Thus
$$\bDo{\lambda}{a+1}{z} = \bdo{\lambda}{\bott_\lambda} = \ell,$$
and so row $\bDo{\gamma}{a}{z} = \bDo{\lambda}{a}{z}$ has no root in root ideal $\Delta^k(\gamma)$. This implies that
$$\bDo{\gamma}{a}{z}>\bott_\gamma, \quad \bDo{\gamma}{a+1}{z}\,\text{ is undefined}, \quad   L_{\bDo{\gamma}{a+1}{z}}=0.$$
Hence
\begin{equation}
\begin{aligned}
\dkl \setminus \beta =&\ \De{k}{\lambda-\epsilon_{\bDo{\lambda}{a}{z}}} = \dkga,\\
L\big(\dkl\big)\setminus \{ \bdo{\lambda}{x} \mid x\in[z,\bott_\lambda] \}
=&\ L\big(\dkga\big)\setminus \{ \bdo{\gamma}{x} \mid x\in[z,\bott_\gamma] \},
\end{aligned}
\mlabel{eq:smalldC2}
\end{equation}
and further
\begin{align*}
&\ K\Big(\dkl\setminus \beta;L\big(\dkl\big)\setminus \{ \bdo{\lambda}{x} \mid x\in[z,\bott_\lambda] \};\gamma\Big)\\
=&\ K\Big(\dkga;L\big(\dkga\big)\setminus \{ \bdo{\gamma}{x} \mid x\in[z,\bott_\gamma] \};\gamma\Big)\hspace{1cm}(\text{by~(\ref{eq:smalldC2})})\\
=&\ \sg{\gamma}{z} \hspace{4cm} (\text{by Definition~\ref{def:semiKk}})\\
=&\ (1-L_{\bDo{\gamma}{a+1}{z}}) \sg{\gamma}{z}   \hspace{2cm} (\text{by $L_{\bDo{\gamma}{a+1}{z}} = 0$}).
\end{align*}
Substituting the above equation into~(\mref{eq:smalldini}) yields the required result.

\mref{it:smalld2} Suppose
\[
\gamma = \lambda-\epsilon_{\bDo{\lambda}{a}{z}}\notin\pkl,\, \text{ that is},\ \, \bDo{\lambda}{a}{z}<\ell\,\text{ and }\, \lambda_{\bDo{\lambda}{a}{z}} = \lambda_{\bDo{\lambda}{a}{z}+1}.
\]
Take the unique $c\in[0,a-1]$ such that~(\mref{eq:csmallest}) holds. We are going to apply Mirror Lemma~\ref{lem:mirr1} with
$$\Psi := \dkl\setminus\beta, \quad M:=L\big(\dkl\big)\setminus \{ \bdo{\lambda}{x} \mid x\in[z,\bott_\lambda] \}.$$
For this, let us examine the conditions in Mirror Lemma~\ref{lem:mirr1} one by one.
\begin{enumerate}
\item[(i)] From the inequality $\lambda_{\bDo{\lambda}{c}{z}} > \lambda_{\bDo{\lambda}{c}{z}+1}$ and Remark~\ref{re:rootfact}~(3), it follows that the root ideal $\dkl$ has a ceiling in columns
$$\bDo{\lambda}{c+1}{z},\quad \bDo{\lambda}{c+1}{z}+1.$$
Since~\eqref{eq:beta} and $\bDo{\lambda}{c+1}{z} \leq \bDo{\lambda}{a}{z}$ ensure that removing $\beta$ does not affect this ceiling, the root ideal $\dkl \setminus \beta$ also retains a ceiling in the same columns.

\item[(ii)] Since Remark~\ref{re:rootfact}~(b) and
    $$k>\lambda_{\bDo{\lambda}{x}{z}} = \lambda_{\bDo{\lambda}{x}{z}+1}\, \text{ for }\,x\in[c+1,a-1],$$
the root ideal $\dkl$ has a mirror in rows
$$x,\  x+1\,\text{ for each }\, x\in{\rm path}_{\dkl}\big(\bDo{\lambda}{c+1}{z}, \bDo{\lambda}{a-1}{z}\big).$$
As a result, the root ideal $\dkl\setminus \beta$ has a mirror in rows
$$x,\  x+1 \,\text{ for each }\, x\in{\rm path}_{\dkl\setminus \beta}\big(\bDo{\lambda}{c+1}{z}, \bDo{\lambda}{a-1}{z}\big),\,\text{ where }\, \bDo{\lambda}{a-1}{z} = {\rm up}_{\dkl\setminus\beta}(\bDo{\lambda}{a}{z}).$$

\item[(iii)] By $\lambda_{\bDo{\lambda}{a}{z}} = \lambda_{\bDo{\lambda}{a}{z}+1}$ and~(\mref{eq:beta}), there is a wall in rows $\bDo{\lambda}{a}{z},\bDo{\lambda}{a}{z}+1$ in the root ideal $\dkl\setminus\beta$.
\item[(iv)] $\gamma_x = \lambda_x = \lambda_{x+1} =  \gamma_{x+1}$ for all $x\in{\rm path}_{\dkl\setminus\beta}\big(\bDo{\lambda}{c+1}{z}, \bDo{\lambda}{a-1}{z}\big)$.
\item[(v)] $\gamma_{\bDo{\lambda}{a}{z}}+1 = \big(\lambda_{\bDo{\lambda}{a}{z}}-1\big)+1 = \lambda_{\bDo{\lambda}{a}{z}} = \lambda_{\bDo{\lambda}{a}{z}+1} =
\big( \lambda-\epsilon_{\bDo{\lambda}{a}{z}} \big)_{\bDo{\lambda}{a}{z}+1} =\gamma_{\bDo{\lambda}{a}{z}+1}$.
\item[(vi)]$m_M(x)+1 = m_M(x+1)$ for all $x\in{\rm path}_{\dkl\setminus\beta}\big(\bDo{\lambda}{c+2}{z}, \bDo{\lambda}{a}{z}\big)$.
\item[(vii)] $m_M\big(\bDo{\lambda}{c+1}{z}\big)+1 = m_M\big(\bDo{\lambda}{c+1}{z}+1\big)$.
\end{enumerate}
By Mirror Lemma~\mref{lem:mirr1},
$$ K\Big(\dkl\setminus \beta;L\big(\dkl\big)\setminus \{ \bdo{\lambda}{x} \mid x\in[z,\bott_\lambda] \};\gamma\Big) = 0, $$
and the result holds by~(\mref{eq:smalldini}).
This completes the proof.
\end{proof}

The following example illustrates Proposition~\mref{prop:smalld}.

\begin{exam}
Let
\[
k=7,\quad \ell=13,\quad \lambda = (7,6,5,5,4,4,4,3,3,3,2,2,1)\in{\rm P}_{13}^7.
\]
Then $\bott_\lambda = 8$. Taking $z=2\in[\bott_\lambda-1]$ and $a=2$, we have
\begin{align*}
& (1) \qquad \bDo{\lambda}{2}{z} = 7 \in [\bott_\lambda], \quad \lambda_2 = 6 >5 =  \lambda_3,\\
& (2) \qquad \gamma = \lambda - \epsilon_{\bDo{\lambda}{2}{z}} = (7,6,5,5,4,4,\textcolor{green}{3},3,3,3,2,2,1) \in {\rm P}_{13}^7,\\
& (3) \qquad \bDo{\gamma}{a+1}{z} = \bDo{\gamma}{3}{z} = 12, \quad \bDo{\lambda}{a+1}{z} = \bDo{\lambda}{3}{z}=11,
\end{align*}
and so
\begin{equation*}
\begin{aligned}
L_{\bDo{\lambda}{2}{z}}\sg{\lambda}{z} =
L_7 \,
\begin{tikzpicture}[scale=.3,line width=0.5pt,baseline=(a.base)]
\draw (0,7) rectangle (1,6);\node at(0.5,6.5){\tiny\( 7 \)};
\filldraw[red,draw=black] (1,7) rectangle (2,6);\node at(1.5,6.5){\tiny\( \bullet \)};
\filldraw[red,draw=black] (2,7) rectangle (3,6);\node at(2.5,6.5){\tiny\( \bullet \)};
\filldraw[red,draw=black] (3,7) rectangle (4,6);\node at(3.5,6.5){\tiny\( \bullet \)};
\filldraw[red,draw=black] (4,7) rectangle (5,6);\node at(4.5,6.5){\tiny\( \bullet \)};
\filldraw[red,draw=black] (5,7) rectangle (6,6);\node at(5.5,6.5){\tiny\( \bullet \)};
\filldraw[red,draw=black] (6,7) rectangle (7,6);\node at(6.5,6.5){\tiny\( \bullet \)};
\filldraw[red,draw=black] (7,7) rectangle (8,6);\node at(7.5,6.5){\tiny\( \bullet \)};
\filldraw[red,draw=black] (8,7) rectangle (9,6);\node at(8.5,6.5){\tiny\( \bullet \)};
\filldraw[red,draw=black] (9,7) rectangle (10,6);\node at(9.5,6.5){\tiny\( \bullet \)};
\filldraw[red,draw=black] (10,7) rectangle (11,6);\node at(10.5,6.5){\tiny\( \bullet \)};
\filldraw[red,draw=black] (11,7) rectangle (12,6);\node at(11.5,6.5){\tiny\( \bullet \)};
\filldraw[red,draw=black] (12,7) rectangle (13,6);\node at(12.5,6.5){\tiny\( \bullet \)};
\draw (0,6) rectangle (1,5);
\draw (1,6) rectangle (2,5);\node at(1.5,5.5){\tiny\( 6 \)};
\draw (2,6) rectangle (3,5);
\filldraw[red,draw=black] (3,6) rectangle (4,5);
\filldraw[red,draw=black] (4,6) rectangle (5,5);\node at(4.5,5.5){\tiny\( \bullet \)};
\filldraw[red,draw=black] (5,6) rectangle (6,5);\node at(5.5,5.5){\tiny\( \bullet \)};
\filldraw[red,draw=black] (6,6) rectangle (7,5);\node at(6.5,5.5){\tiny\( \bullet \)};
\filldraw[red,draw=black] (7,6) rectangle (8,5);\node at(7.5,5.5){\tiny\( \bullet \)};
\filldraw[red,draw=black] (8,6) rectangle (9,5);\node at(8.5,5.5){\tiny\( \bullet \)};
\filldraw[red,draw=black] (9,6) rectangle (10,5);\node at(9.5,5.5){\tiny\( \bullet \)};
\filldraw[red,draw=black] (10,6) rectangle (11,5);\node at(10.5,5.5){\tiny\( \bullet \)};
\filldraw[red,draw=black] (11,6) rectangle (12,5);\node at(11.5,5.5){\tiny\( \bullet \)};
\filldraw[red,draw=black] (12,6) rectangle (13,5);\node at(12.5,5.5){\tiny\( \bullet \)};
\draw (0,5) rectangle (1,4);
\draw (1,5) rectangle (2,4);
\draw (2,5) rectangle (3,4);\node at(2.5,4.5){\tiny\( 5 \)};
\draw (3,5) rectangle (4,4);
\draw (4,5) rectangle (5,4);
\filldraw[red,draw=black] (5,5) rectangle (6,4);
\filldraw[red,draw=black] (6,5) rectangle (7,4);\node at(6.5,4.5){\tiny\( \bullet \)};
\filldraw[red,draw=black] (7,5) rectangle (8,4);\node at(7.5,4.5){\tiny\( \bullet \)};
\filldraw[red,draw=black] (8,5) rectangle (9,4);\node at(8.5,4.5){\tiny\( \bullet \)};
\filldraw[red,draw=black] (9,5) rectangle (10,4);\node at(9.5,4.5){\tiny\( \bullet \)};
\filldraw[red,draw=black] (10,5) rectangle (11,4);\node at(10.5,4.5){\tiny\( \bullet \)};
\filldraw[red,draw=black] (11,5) rectangle (12,4);\node at(11.5,4.5){\tiny\( \bullet \)};
\filldraw[red,draw=black] (12,5) rectangle (13,4);\node at(12.5,4.5){\tiny\( \bullet \)};
\draw (0,4) rectangle (1,3);
\draw (1,4) rectangle (2,3);
\draw (2,4) rectangle (3,3);
\draw (3,4) rectangle (4,3);\node at(3.5,3.5){\tiny\( 5 \)};
\draw (4,4) rectangle (5,3);
\draw (5,4) rectangle (6,3);
\filldraw[red,draw=black] (6,4) rectangle (7,3);
\filldraw[red,draw=black] (7,4) rectangle (8,3);\node at(7.5,3.5){\tiny\( \bullet \)};
\filldraw[red,draw=black] (8,4) rectangle (9,3);\node at(8.5,3.5){\tiny\( \bullet \)};
\filldraw[red,draw=black] (9,4) rectangle (10,3);\node at(9.5,3.5){\tiny\( \bullet \)};
\filldraw[red,draw=black] (10,4) rectangle (11,3);\node at(10.5,3.5){\tiny\( \bullet \)};
\filldraw[red,draw=black] (11,4) rectangle (12,3);\node at(11.5,3.5){\tiny\( \bullet \)};
\filldraw[red,draw=black] (12,4) rectangle (13,3);\node at(12.5,3.5){\tiny\( \bullet \)};
\draw (0,3) rectangle (1,2);
\draw (1,3) rectangle (2,2);
\draw (2,3) rectangle (3,2);
\draw (3,3) rectangle (4,2);
\draw (4,3) rectangle (5,2);\node at(4.5,2.5){\tiny\( 4 \)};
\draw (5,3) rectangle (6,2);
\draw (6,3) rectangle (7,2);
\draw (7,3) rectangle (8,2);
\filldraw[red,draw=black] (8,3) rectangle (9,2);
\filldraw[red,draw=black] (9,3) rectangle (10,2);\node at(9.5,2.5){\tiny\( \bullet \)};
\filldraw[red,draw=black] (10,3) rectangle (11,2);\node at(10.5,2.5){\tiny\( \bullet \)};
\filldraw[red,draw=black] (11,3) rectangle (12,2);\node at(11.5,2.5){\tiny\( \bullet \)};
\filldraw[red,draw=black] (12,3) rectangle (13,2);\node at(12.5,2.5){\tiny\( \bullet \)};
\draw (0,2) rectangle (1,1);
\draw (1,2) rectangle (2,1);
\draw (2,2) rectangle (3,1);
\draw (3,2) rectangle (4,1);
\draw (4,2) rectangle (5,1);
\draw (5,2) rectangle (6,1);\node at(5.5,1.5){\tiny\( 4 \)};
\draw (6,2) rectangle (7,1);
\draw (7,2) rectangle (8,1);
\draw (8,2) rectangle (9,1);
\filldraw[red,draw=black] (9,2) rectangle (10,1);4
\filldraw[red,draw=black] (10,2) rectangle (11,1);\node at(10.5,1.5){\tiny\( \bullet \)};
\filldraw[red,draw=black] (11,2) rectangle (12,1);\node at(11.5,1.5){\tiny\( \bullet \)};
\filldraw[red,draw=black] (12,2) rectangle (13,1);\node at(12.5,1.5){\tiny\( \bullet \)};
\draw (0,1) rectangle (1,0);
\draw (1,1) rectangle (2,0);
\draw (2,1) rectangle (3,0);
\draw (3,1) rectangle (4,0);
\draw (4,1) rectangle (5,0);
\draw (5,1) rectangle (6,0);
\draw (6,1) rectangle (7,0);\node at(6.5,0.5){\tiny\( 4 \)};
\draw (7,1) rectangle (8,0);
\draw (8,1) rectangle (9,0);
\draw (9,1) rectangle (10,0);
\filldraw[red,draw=black] (10,1) rectangle (11,0);
\filldraw[red,draw=black] (11,1) rectangle (12,0);\node at(11.5,0.5){\tiny\( \bullet \)};
\filldraw[red,draw=black] (12,1) rectangle (13,0);\node at(12.5,0.5){\tiny\( \bullet \)};
\draw (0,0) rectangle (1,-1);
\draw (1,0) rectangle (2,-1);
\draw (2,0) rectangle (3,-1);
\draw (3,0) rectangle (4,-1);
\draw (4,0) rectangle (5,-1);
\draw (5,0) rectangle (6,-1);
\draw (6,0) rectangle (7,-1);
\draw (7,0) rectangle (8,-1);\node at(7.5,-0.5){\tiny\( 3 \)};
\draw (8,0) rectangle (9,-1);
\draw (9,0) rectangle (10,-1);
\draw (10,0) rectangle (11,-1);
\draw (11,0) rectangle (12,-1);
\filldraw[red,draw=black] (12,0) rectangle (13,-1);
\draw (0,-1) rectangle (1,-2);
\draw (1,-1) rectangle (2,-2);
\draw (2,-1) rectangle (3,-2);
\draw (3,-1) rectangle (4,-2);
\draw (4,-1) rectangle (5,-2);
\draw (5,-1) rectangle (6,-2);
\draw (6,-1) rectangle (7,-2);
\draw (7,-1) rectangle (8,-2);
\draw (8,-1) rectangle (9,-2);\node at(8.5,-1.5){\tiny\( 3 \)};
\draw (9,-1) rectangle (10,-2);
\draw (10,-1) rectangle (11,-2);
\draw (11,-1) rectangle (12,-2);
\draw (12,-1) rectangle (13,-2);
\draw (0,-2) rectangle (1,-3);
\draw (1,-2) rectangle (2,-3);
\draw (2,-2) rectangle (3,-3);
\draw (3,-2) rectangle (4,-3);
\draw (4,-2) rectangle (5,-3);
\draw (5,-2) rectangle (6,-3);
\draw (6,-2) rectangle (7,-3);
\draw (7,-2) rectangle (8,-3);
\draw (8,-2) rectangle (9,-3);
\draw (9,-2) rectangle (10,-3);\node at(9.5,-2.5){\tiny\( 3 \)};
\draw (10,-2) rectangle (11,-3);
\draw (11,-2) rectangle (12,-3);
\draw (12,-2) rectangle (13,-3);
\draw (0,-3) rectangle (1,-4);
\draw (1,-3) rectangle (2,-4);
\draw (2,-3) rectangle (3,-4);
\draw (3,-3) rectangle (4,-4);
\draw (4,-3) rectangle (5,-4);
\draw (5,-3) rectangle (6,-4);
\draw (6,-3) rectangle (7,-4);
\draw (7,-3) rectangle (8,-4);
\draw (8,-3) rectangle (9,-4);
\draw (9,-3) rectangle (10,-4);
\draw (10,-3) rectangle (11,-4);\node at(10.5,-3.5){\tiny\( 2 \)};
\draw (11,-3) rectangle (12,-4);
\draw (12,-3) rectangle (13,-4);
\draw (0,-4) rectangle (1,-5);
\draw (1,-4) rectangle (2,-5);
\draw (2,-4) rectangle (3,-5);
\draw (3,-4) rectangle (4,-5);
\draw (4,-4) rectangle (5,-5);
\draw (5,-4) rectangle (6,-5);
\draw (6,-4) rectangle (7,-5);
\draw (7,-4) rectangle (8,-5);
\draw (8,-4) rectangle (9,-5);
\draw (9,-4) rectangle (10,-5);
\draw (10,-4) rectangle (11,-5);
\draw (11,-4) rectangle (12,-5);\node at(11.5,-4.5){\tiny\( 2 \)};
\draw (12,-4) rectangle (13,-5);
\draw (0,-5) rectangle (1,-6);
\draw (1,-5) rectangle (2,-6);
\draw (2,-5) rectangle (3,-6);
\draw (3,-5) rectangle (4,-6);
\draw (4,-5) rectangle (5,-6);
\draw (5,-5) rectangle (6,-6);
\draw (6,-5) rectangle (7,-6);
\draw (7,-5) rectangle (8,-6);
\draw (8,-5) rectangle (9,-6);
\draw (9,-5) rectangle (10,-6);
\draw (10,-5) rectangle (11,-6);
\draw (11,-5) rectangle (12,-6);
\draw (12,-5) rectangle (13,-6);\node at(12.5,-5.5){\tiny\( 1 \)};
\draw[purple,line width=0.8pt] (4,3.5)--(6,3.5);
\draw[purple,line width=0.8pt] (3.5,4)--(3.5,5);
\draw[purple,line width=0.8pt] (6.5,3)--(6.5,1);
\draw[purple,line width=0.8pt] (4,3.5)--(6,3.5);
\draw[purple,line width=0.8pt] (7,0.5)--(10,0.5);
%
\draw[black,line width=0.8pt,->] (13,5.5)--(13.8,5.5);
\node at (14.1,5.5){ \tiny{\textbf{$2$}} };
\draw[black,line width=0.8pt,->] (13,0.5)--(13.8,0.5);
\node at (14.1,0.5){ \tiny{\textbf{$7$}} };
\draw[black,line width=0.8pt,->] (10.5,7)--(10.5,7.8);
\node at (10.5,8.2){ \tiny{\textbf{$11$}} };
\end{tikzpicture}
=&\
(1-L_{12})\,
\begin{tikzpicture}[scale=.3,line width=0.5pt,baseline=(a.base)]
\draw (0,7) rectangle (1,6);\node at(0.5,6.5){\tiny\( 7 \)};
\filldraw[red,draw=black] (1,7) rectangle (2,6);\node at(1.5,6.5){\tiny\( \bullet \)};
\filldraw[red,draw=black] (2,7) rectangle (3,6);\node at(2.5,6.5){\tiny\( \bullet \)};
\filldraw[red,draw=black] (3,7) rectangle (4,6);\node at(3.5,6.5){\tiny\( \bullet \)};
\filldraw[red,draw=black] (4,7) rectangle (5,6);\node at(4.5,6.5){\tiny\( \bullet \)};
\filldraw[red,draw=black] (5,7) rectangle (6,6);\node at(5.5,6.5){\tiny\( \bullet \)};
\filldraw[red,draw=black] (6,7) rectangle (7,6);\node at(6.5,6.5){\tiny\( \bullet \)};
\filldraw[red,draw=black] (7,7) rectangle (8,6);\node at(7.5,6.5){\tiny\( \bullet \)};
\filldraw[red,draw=black] (8,7) rectangle (9,6);\node at(8.5,6.5){\tiny\( \bullet \)};
\filldraw[red,draw=black] (9,7) rectangle (10,6);\node at(9.5,6.5){\tiny\( \bullet \)};
\filldraw[red,draw=black] (10,7) rectangle (11,6);\node at(10.5,6.5){\tiny\( \bullet \)};
\filldraw[red,draw=black] (11,7) rectangle (12,6);\node at(11.5,6.5){\tiny\( \bullet \)};
\filldraw[red,draw=black] (12,7) rectangle (13,6);\node at(12.5,6.5){\tiny\( \bullet \)};
\draw (0,6) rectangle (1,5);
\draw (1,6) rectangle (2,5);\node at(1.5,5.5){\tiny\( 6 \)};
\draw (2,6) rectangle (3,5);
\filldraw[red,draw=black] (3,6) rectangle (4,5);
\filldraw[red,draw=black] (4,6) rectangle (5,5);\node at(4.5,5.5){\tiny\( \bullet \)};
\filldraw[red,draw=black] (5,6) rectangle (6,5);\node at(5.5,5.5){\tiny\( \bullet \)};
\filldraw[red,draw=black] (6,6) rectangle (7,5);\node at(6.5,5.5){\tiny\( \bullet \)};
\filldraw[red,draw=black] (7,6) rectangle (8,5);\node at(7.5,5.5){\tiny\( \bullet \)};
\filldraw[red,draw=black] (8,6) rectangle (9,5);\node at(8.5,5.5){\tiny\( \bullet \)};
\filldraw[red,draw=black] (9,6) rectangle (10,5);\node at(9.5,5.5){\tiny\( \bullet \)};
\filldraw[red,draw=black] (10,6) rectangle (11,5);\node at(10.5,5.5){\tiny\( \bullet \)};
\filldraw[red,draw=black] (11,6) rectangle (12,5);\node at(11.5,5.5){\tiny\( \bullet \)};
\filldraw[red,draw=black] (12,6) rectangle (13,5);\node at(12.5,5.5){\tiny\( \bullet \)};
\draw (0,5) rectangle (1,4);
\draw (1,5) rectangle (2,4);
\draw (2,5) rectangle (3,4);\node at(2.5,4.5){\tiny\( 5 \)};
\draw (3,5) rectangle (4,4);
\draw (4,5) rectangle (5,4);
\filldraw[red,draw=black] (5,5) rectangle (6,4);
\filldraw[red,draw=black] (6,5) rectangle (7,4);\node at(6.5,4.5){\tiny\( \bullet \)};
\filldraw[red,draw=black] (7,5) rectangle (8,4);\node at(7.5,4.5){\tiny\( \bullet \)};
\filldraw[red,draw=black] (8,5) rectangle (9,4);\node at(8.5,4.5){\tiny\( \bullet \)};
\filldraw[red,draw=black] (9,5) rectangle (10,4);\node at(9.5,4.5){\tiny\( \bullet \)};
\filldraw[red,draw=black] (10,5) rectangle (11,4);\node at(10.5,4.5){\tiny\( \bullet \)};
\filldraw[red,draw=black] (11,5) rectangle (12,4);\node at(11.5,4.5){\tiny\( \bullet \)};
\filldraw[red,draw=black] (12,5) rectangle (13,4);\node at(12.5,4.5){\tiny\( \bullet \)};
\draw (0,4) rectangle (1,3);
\draw (1,4) rectangle (2,3);
\draw (2,4) rectangle (3,3);
\draw (3,4) rectangle (4,3);\node at(3.5,3.5){\tiny\( 5 \)};
\draw (4,4) rectangle (5,3);
\draw (5,4) rectangle (6,3);
\filldraw[red,draw=black] (6,4) rectangle (7,3);
\filldraw[red,draw=black] (7,4) rectangle (8,3);\node at(7.5,3.5){\tiny\( \bullet \)};
\filldraw[red,draw=black] (8,4) rectangle (9,3);\node at(8.5,3.5){\tiny\( \bullet \)};
\filldraw[red,draw=black] (9,4) rectangle (10,3);\node at(9.5,3.5){\tiny\( \bullet \)};
\filldraw[red,draw=black] (10,4) rectangle (11,3);\node at(10.5,3.5){\tiny\( \bullet \)};
\filldraw[red,draw=black] (11,4) rectangle (12,3);\node at(11.5,3.5){\tiny\( \bullet \)};
\filldraw[red,draw=black] (12,4) rectangle (13,3);\node at(12.5,3.5){\tiny\( \bullet \)};
\draw (0,3) rectangle (1,2);
\draw (1,3) rectangle (2,2);
\draw (2,3) rectangle (3,2);
\draw (3,3) rectangle (4,2);
\draw (4,3) rectangle (5,2);\node at(4.5,2.5){\tiny\( 4 \)};
\draw (5,3) rectangle (6,2);
\draw (6,3) rectangle (7,2);
\draw (7,3) rectangle (8,2);
\filldraw[red,draw=black] (8,3) rectangle (9,2);
\filldraw[red,draw=black] (9,3) rectangle (10,2);\node at(9.5,2.5){\tiny\( \bullet \)};
\filldraw[red,draw=black] (10,3) rectangle (11,2);\node at(10.5,2.5){\tiny\( \bullet \)};
\filldraw[red,draw=black] (11,3) rectangle (12,2);\node at(11.5,2.5){\tiny\( \bullet \)};
\filldraw[red,draw=black] (12,3) rectangle (13,2);\node at(12.5,2.5){\tiny\( \bullet \)};
\draw (0,2) rectangle (1,1);
\draw (1,2) rectangle (2,1);
\draw (2,2) rectangle (3,1);
\draw (3,2) rectangle (4,1);
\draw (4,2) rectangle (5,1);
\draw (5,2) rectangle (6,1);\node at(5.5,1.5){\tiny\( 4 \)};
\draw (6,2) rectangle (7,1);
\draw (7,2) rectangle (8,1);
\draw (8,2) rectangle (9,1);
\filldraw[red,draw=black] (9,2) rectangle (10,1);4
\filldraw[red,draw=black] (10,2) rectangle (11,1);\node at(10.5,1.5){\tiny\( \bullet \)};
\filldraw[red,draw=black] (11,2) rectangle (12,1);\node at(11.5,1.5){\tiny\( \bullet \)};
\filldraw[red,draw=black] (12,2) rectangle (13,1);\node at(12.5,1.5){\tiny\( \bullet \)};
\draw (0,1) rectangle (1,0);
\draw (1,1) rectangle (2,0);
\draw (2,1) rectangle (3,0);
\draw (3,1) rectangle (4,0);
\draw (4,1) rectangle (5,0);
\draw (5,1) rectangle (6,0);
\draw (6,1) rectangle (7,0);\node at(6.5,0.5){\tiny\( \textcolor{green}{3} \)};
\draw (7,1) rectangle (8,0);
\draw (8,1) rectangle (9,0);
\draw (9,1) rectangle (10,0);
\draw (10,1) rectangle (11,0);
\filldraw[red,draw=black] (11,1) rectangle (12,0);
\filldraw[red,draw=black] (12,1) rectangle (13,0);\node at(12.5,0.5){\tiny\( \bullet \)};
\draw (0,0) rectangle (1,-1);
\draw (1,0) rectangle (2,-1);
\draw (2,0) rectangle (3,-1);
\draw (3,0) rectangle (4,-1);
\draw (4,0) rectangle (5,-1);
\draw (5,0) rectangle (6,-1);
\draw (6,0) rectangle (7,-1);
\draw (7,0) rectangle (8,-1);\node at(7.5,-0.5){\tiny\( 3 \)};
\draw (8,0) rectangle (9,-1);
\draw (9,0) rectangle (10,-1);
\draw (10,0) rectangle (11,-1);
\draw (11,0) rectangle (12,-1);
\filldraw[red,draw=black] (12,0) rectangle (13,-1);
\draw (0,-1) rectangle (1,-2);
\draw (1,-1) rectangle (2,-2);
\draw (2,-1) rectangle (3,-2);
\draw (3,-1) rectangle (4,-2);
\draw (4,-1) rectangle (5,-2);
\draw (5,-1) rectangle (6,-2);
\draw (6,-1) rectangle (7,-2);
\draw (7,-1) rectangle (8,-2);
\draw (8,-1) rectangle (9,-2);\node at(8.5,-1.5){\tiny\( 3 \)};
\draw (9,-1) rectangle (10,-2);
\draw (10,-1) rectangle (11,-2);
\draw (11,-1) rectangle (12,-2);
\draw (12,-1) rectangle (13,-2);
\draw (0,-2) rectangle (1,-3);
\draw (1,-2) rectangle (2,-3);
\draw (2,-2) rectangle (3,-3);
\draw (3,-2) rectangle (4,-3);
\draw (4,-2) rectangle (5,-3);
\draw (5,-2) rectangle (6,-3);
\draw (6,-2) rectangle (7,-3);
\draw (7,-2) rectangle (8,-3);
\draw (8,-2) rectangle (9,-3);
\draw (9,-2) rectangle (10,-3);\node at(9.5,-2.5){\tiny\( 3 \)};
\draw (10,-2) rectangle (11,-3);
\draw (11,-2) rectangle (12,-3);
\draw (12,-2) rectangle (13,-3);
\draw (0,-3) rectangle (1,-4);
\draw (1,-3) rectangle (2,-4);
\draw (2,-3) rectangle (3,-4);
\draw (3,-3) rectangle (4,-4);
\draw (4,-3) rectangle (5,-4);
\draw (5,-3) rectangle (6,-4);
\draw (6,-3) rectangle (7,-4);
\draw (7,-3) rectangle (8,-4);
\draw (8,-3) rectangle (9,-4);
\draw (9,-3) rectangle (10,-4);
\draw (10,-3) rectangle (11,-4);\node at(10.5,-3.5){\tiny\( 2 \)};
\draw (11,-3) rectangle (12,-4);
\draw (12,-3) rectangle (13,-4);
\draw (0,-4) rectangle (1,-5);
\draw (1,-4) rectangle (2,-5);
\draw (2,-4) rectangle (3,-5);
\draw (3,-4) rectangle (4,-5);
\draw (4,-4) rectangle (5,-5);
\draw (5,-4) rectangle (6,-5);
\draw (6,-4) rectangle (7,-5);
\draw (7,-4) rectangle (8,-5);
\draw (8,-4) rectangle (9,-5);
\draw (9,-4) rectangle (10,-5);
\draw (10,-4) rectangle (11,-5);
\draw (11,-4) rectangle (12,-5);\node at(11.5,-4.5){\tiny\( 2 \)};
\draw (12,-4) rectangle (13,-5);
\draw (0,-5) rectangle (1,-6);
\draw (1,-5) rectangle (2,-6);
\draw (2,-5) rectangle (3,-6);
\draw (3,-5) rectangle (4,-6);
\draw (4,-5) rectangle (5,-6);
\draw (5,-5) rectangle (6,-6);
\draw (6,-5) rectangle (7,-6);
\draw (7,-5) rectangle (8,-6);
\draw (8,-5) rectangle (9,-6);
\draw (9,-5) rectangle (10,-6);
\draw (10,-5) rectangle (11,-6);
\draw (11,-5) rectangle (12,-6);
\draw (12,-5) rectangle (13,-6);\node at(12.5,-5.5){\tiny\( 1 \)};
\draw[purple,line width=0.8pt] (4,3.5)--(6,3.5);
\draw[purple,line width=0.8pt] (3.5,4)--(3.5,5);
\draw[purple,line width=0.8pt] (6.5,3)--(6.5,1);
\draw[purple,line width=0.8pt] (4,3.5)--(6,3.5);
\draw[purple,line width=0.8pt] (7,0.5)--(11,0.5);
%
\draw[black,line width=0.8pt,->] (13,5.5)--(13.8,5.5);
\node at (14.1,5.5){ \tiny{\textbf{$2$}} };
\draw[black,line width=0.8pt,->] (13,0.5)--(13.8,0.5);
\node at (14.1,0.5){ \tiny{\textbf{$7$}} };
\draw[black,line width=0.8pt,->] (11.5,7)--(11.5,7.8);
\node at (11.5,8.2){ \tiny{\textbf{$12$}} };
\end{tikzpicture}
\\
&\ +
L_{11} \,
\begin{tikzpicture}[scale=.3,line width=0.5pt,baseline=(a.base)]
\draw (0,7) rectangle (1,6);\node at(0.5,6.5){\tiny\( 7 \)};
\filldraw[red,draw=black] (1,7) rectangle (2,6);\node at(1.5,6.5){\tiny\( \bullet \)};
\filldraw[red,draw=black] (2,7) rectangle (3,6);\node at(2.5,6.5){\tiny\( \bullet \)};
\filldraw[red,draw=black] (3,7) rectangle (4,6);\node at(3.5,6.5){\tiny\( \bullet \)};
\filldraw[red,draw=black] (4,7) rectangle (5,6);\node at(4.5,6.5){\tiny\( \bullet \)};
\filldraw[red,draw=black] (5,7) rectangle (6,6);\node at(5.5,6.5){\tiny\( \bullet \)};
\filldraw[red,draw=black] (6,7) rectangle (7,6);\node at(6.5,6.5){\tiny\( \bullet \)};
\filldraw[red,draw=black] (7,7) rectangle (8,6);\node at(7.5,6.5){\tiny\( \bullet \)};
\filldraw[red,draw=black] (8,7) rectangle (9,6);\node at(8.5,6.5){\tiny\( \bullet \)};
\filldraw[red,draw=black] (9,7) rectangle (10,6);\node at(9.5,6.5){\tiny\( \bullet \)};
\filldraw[red,draw=black] (10,7) rectangle (11,6);\node at(10.5,6.5){\tiny\( \bullet \)};
\filldraw[red,draw=black] (11,7) rectangle (12,6);\node at(11.5,6.5){\tiny\( \bullet \)};
\filldraw[red,draw=black] (12,7) rectangle (13,6);\node at(12.5,6.5){\tiny\( \bullet \)};
\draw (0,6) rectangle (1,5);
\draw (1,6) rectangle (2,5);\node at(1.5,5.5){\tiny\( 6 \)};
\draw (2,6) rectangle (3,5);
\filldraw[red,draw=black] (3,6) rectangle (4,5);
\filldraw[red,draw=black] (4,6) rectangle (5,5);\node at(4.5,5.5){\tiny\( \bullet \)};
\filldraw[red,draw=black] (5,6) rectangle (6,5);\node at(5.5,5.5){\tiny\( \bullet \)};
\filldraw[red,draw=black] (6,6) rectangle (7,5);\node at(6.5,5.5){\tiny\( \bullet \)};
\filldraw[red,draw=black] (7,6) rectangle (8,5);\node at(7.5,5.5){\tiny\( \bullet \)};
\filldraw[red,draw=black] (8,6) rectangle (9,5);\node at(8.5,5.5){\tiny\( \bullet \)};
\filldraw[red,draw=black] (9,6) rectangle (10,5);\node at(9.5,5.5){\tiny\( \bullet \)};
\filldraw[red,draw=black] (10,6) rectangle (11,5);\node at(10.5,5.5){\tiny\( \bullet \)};
\filldraw[red,draw=black] (11,6) rectangle (12,5);\node at(11.5,5.5){\tiny\( \bullet \)};
\filldraw[red,draw=black] (12,6) rectangle (13,5);\node at(12.5,5.5){\tiny\( \bullet \)};
\draw (0,5) rectangle (1,4);
\draw (1,5) rectangle (2,4);
\draw (2,5) rectangle (3,4);\node at(2.5,4.5){\tiny\( 5 \)};
\draw (3,5) rectangle (4,4);
\draw (4,5) rectangle (5,4);
\filldraw[red,draw=black] (5,5) rectangle (6,4);
\filldraw[red,draw=black] (6,5) rectangle (7,4);\node at(6.5,4.5){\tiny\( \bullet \)};
\filldraw[red,draw=black] (7,5) rectangle (8,4);\node at(7.5,4.5){\tiny\( \bullet \)};
\filldraw[red,draw=black] (8,5) rectangle (9,4);\node at(8.5,4.5){\tiny\( \bullet \)};
\filldraw[red,draw=black] (9,5) rectangle (10,4);\node at(9.5,4.5){\tiny\( \bullet \)};
\filldraw[red,draw=black] (10,5) rectangle (11,4);\node at(10.5,4.5){\tiny\( \bullet \)};
\filldraw[red,draw=black] (11,5) rectangle (12,4);\node at(11.5,4.5){\tiny\( \bullet \)};
\filldraw[red,draw=black] (12,5) rectangle (13,4);\node at(12.5,4.5){\tiny\( \bullet \)};
\draw (0,4) rectangle (1,3);
\draw (1,4) rectangle (2,3);
\draw (2,4) rectangle (3,3);
\draw (3,4) rectangle (4,3);\node at(3.5,3.5){\tiny\( 5 \)};
\draw (4,4) rectangle (5,3);
\draw (5,4) rectangle (6,3);
\filldraw[red,draw=black] (6,4) rectangle (7,3);
\filldraw[red,draw=black] (7,4) rectangle (8,3);\node at(7.5,3.5){\tiny\( \bullet \)};
\filldraw[red,draw=black] (8,4) rectangle (9,3);\node at(8.5,3.5){\tiny\( \bullet \)};
\filldraw[red,draw=black] (9,4) rectangle (10,3);\node at(9.5,3.5){\tiny\( \bullet \)};
\filldraw[red,draw=black] (10,4) rectangle (11,3);\node at(10.5,3.5){\tiny\( \bullet \)};
\filldraw[red,draw=black] (11,4) rectangle (12,3);\node at(11.5,3.5){\tiny\( \bullet \)};
\filldraw[red,draw=black] (12,4) rectangle (13,3);\node at(12.5,3.5){\tiny\( \bullet \)};
\draw (0,3) rectangle (1,2);
\draw (1,3) rectangle (2,2);
\draw (2,3) rectangle (3,2);
\draw (3,3) rectangle (4,2);
\draw (4,3) rectangle (5,2);\node at(4.5,2.5){\tiny\( 4 \)};
\draw (5,3) rectangle (6,2);
\draw (6,3) rectangle (7,2);
\draw (7,3) rectangle (8,2);
\filldraw[red,draw=black] (8,3) rectangle (9,2);
\filldraw[red,draw=black] (9,3) rectangle (10,2);\node at(9.5,2.5){\tiny\( \bullet \)};
\filldraw[red,draw=black] (10,3) rectangle (11,2);\node at(10.5,2.5){\tiny\( \bullet \)};
\filldraw[red,draw=black] (11,3) rectangle (12,2);\node at(11.5,2.5){\tiny\( \bullet \)};
\filldraw[red,draw=black] (12,3) rectangle (13,2);\node at(12.5,2.5){\tiny\( \bullet \)};
\draw (0,2) rectangle (1,1);
\draw (1,2) rectangle (2,1);
\draw (2,2) rectangle (3,1);
\draw (3,2) rectangle (4,1);
\draw (4,2) rectangle (5,1);
\draw (5,2) rectangle (6,1);\node at(5.5,1.5){\tiny\( 4 \)};
\draw (6,2) rectangle (7,1);
\draw (7,2) rectangle (8,1);
\draw (8,2) rectangle (9,1);
\filldraw[red,draw=black] (9,2) rectangle (10,1);4
\filldraw[red,draw=black] (10,2) rectangle (11,1);\node at(10.5,1.5){\tiny\( \bullet \)};
\filldraw[red,draw=black] (11,2) rectangle (12,1);\node at(11.5,1.5){\tiny\( \bullet \)};
\filldraw[red,draw=black] (12,2) rectangle (13,1);\node at(12.5,1.5){\tiny\( \bullet \)};
\draw (0,1) rectangle (1,0);
\draw (1,1) rectangle (2,0);
\draw (2,1) rectangle (3,0);
\draw (3,1) rectangle (4,0);
\draw (4,1) rectangle (5,0);
\draw (5,1) rectangle (6,0);
\draw (6,1) rectangle (7,0);\node at(6.5,0.5){\tiny\( 4 \)};
\draw (7,1) rectangle (8,0);
\draw (8,1) rectangle (9,0);
\draw (9,1) rectangle (10,0);
\filldraw[red,draw=black] (10,1) rectangle (11,0);
\filldraw[red,draw=black] (11,1) rectangle (12,0);\node at(11.5,0.5){\tiny\( \bullet \)};
\filldraw[red,draw=black] (12,1) rectangle (13,0);\node at(12.5,0.5){\tiny\( \bullet \)};
\draw (0,0) rectangle (1,-1);
\draw (1,0) rectangle (2,-1);
\draw (2,0) rectangle (3,-1);
\draw (3,0) rectangle (4,-1);
\draw (4,0) rectangle (5,-1);
\draw (5,0) rectangle (6,-1);
\draw (6,0) rectangle (7,-1);
\draw (7,0) rectangle (8,-1);\node at(7.5,-0.5){\tiny\( 3 \)};
\draw (8,0) rectangle (9,-1);
\draw (9,0) rectangle (10,-1);
\draw (10,0) rectangle (11,-1);
\draw (11,0) rectangle (12,-1);
\filldraw[red,draw=black] (12,0) rectangle (13,-1);
\draw (0,-1) rectangle (1,-2);
\draw (1,-1) rectangle (2,-2);
\draw (2,-1) rectangle (3,-2);
\draw (3,-1) rectangle (4,-2);
\draw (4,-1) rectangle (5,-2);
\draw (5,-1) rectangle (6,-2);
\draw (6,-1) rectangle (7,-2);
\draw (7,-1) rectangle (8,-2);
\draw (8,-1) rectangle (9,-2);\node at(8.5,-1.5){\tiny\( 3 \)};
\draw (9,-1) rectangle (10,-2);
\draw (10,-1) rectangle (11,-2);
\draw (11,-1) rectangle (12,-2);
\draw (12,-1) rectangle (13,-2);
\draw (0,-2) rectangle (1,-3);
\draw (1,-2) rectangle (2,-3);
\draw (2,-2) rectangle (3,-3);
\draw (3,-2) rectangle (4,-3);
\draw (4,-2) rectangle (5,-3);
\draw (5,-2) rectangle (6,-3);
\draw (6,-2) rectangle (7,-3);
\draw (7,-2) rectangle (8,-3);
\draw (8,-2) rectangle (9,-3);
\draw (9,-2) rectangle (10,-3);\node at(9.5,-2.5){\tiny\( 3 \)};
\draw (10,-2) rectangle (11,-3);
\draw (11,-2) rectangle (12,-3);
\draw (12,-2) rectangle (13,-3);
\draw (0,-3) rectangle (1,-4);
\draw (1,-3) rectangle (2,-4);
\draw (2,-3) rectangle (3,-4);
\draw (3,-3) rectangle (4,-4);
\draw (4,-3) rectangle (5,-4);
\draw (5,-3) rectangle (6,-4);
\draw (6,-3) rectangle (7,-4);
\draw (7,-3) rectangle (8,-4);
\draw (8,-3) rectangle (9,-4);
\draw (9,-3) rectangle (10,-4);
\draw (10,-3) rectangle (11,-4);\node at(10.5,-3.5){\tiny\( 2 \)};
\draw (11,-3) rectangle (12,-4);
\draw (12,-3) rectangle (13,-4);
\draw (0,-4) rectangle (1,-5);
\draw (1,-4) rectangle (2,-5);
\draw (2,-4) rectangle (3,-5);
\draw (3,-4) rectangle (4,-5);
\draw (4,-4) rectangle (5,-5);
\draw (5,-4) rectangle (6,-5);
\draw (6,-4) rectangle (7,-5);
\draw (7,-4) rectangle (8,-5);
\draw (8,-4) rectangle (9,-5);
\draw (9,-4) rectangle (10,-5);
\draw (10,-4) rectangle (11,-5);
\draw (11,-4) rectangle (12,-5);\node at(11.5,-4.5){\tiny\( 2 \)};
\draw (12,-4) rectangle (13,-5);
\draw (0,-5) rectangle (1,-6);
\draw (1,-5) rectangle (2,-6);
\draw (2,-5) rectangle (3,-6);
\draw (3,-5) rectangle (4,-6);
\draw (4,-5) rectangle (5,-6);
\draw (5,-5) rectangle (6,-6);
\draw (6,-5) rectangle (7,-6);
\draw (7,-5) rectangle (8,-6);
\draw (8,-5) rectangle (9,-6);
\draw (9,-5) rectangle (10,-6);
\draw (10,-5) rectangle (11,-6);
\draw (11,-5) rectangle (12,-6);
\draw (12,-5) rectangle (13,-6);\node at(12.5,-5.5){\tiny\( 1 \)};
%
\end{tikzpicture}
\, .
\end{aligned}
\end{equation*}
For $\lambda = (7,6,5,5,4,4,4,\textcolor{red}{4},3,3,2,2,1)\in{\rm P}_{13}^7$,
$$\lambda-\epsilon_{7} = (7,6,5,5,4,4,\textcolor{green}{3},4,3,3,2,2,1)\notin{\rm P}_{13}^7.$$
Thus
\begin{equation*}
L_{\bDo{\lambda}{2}{z}}\sg{\lambda}{z} =
L_7 \,
\begin{tikzpicture}[scale=.3,line width=0.5pt,baseline=(a.base)]
\draw (0,7) rectangle (1,6);\node at(0.5,6.5){\tiny\( 7 \)};
\filldraw[red,draw=black] (1,7) rectangle (2,6);\node at(1.5,6.5){\tiny\( \bullet \)};
\filldraw[red,draw=black] (2,7) rectangle (3,6);\node at(2.5,6.5){\tiny\( \bullet \)};
\filldraw[red,draw=black] (3,7) rectangle (4,6);\node at(3.5,6.5){\tiny\( \bullet \)};
\filldraw[red,draw=black] (4,7) rectangle (5,6);\node at(4.5,6.5){\tiny\( \bullet \)};
\filldraw[red,draw=black] (5,7) rectangle (6,6);\node at(5.5,6.5){\tiny\( \bullet \)};
\filldraw[red,draw=black] (6,7) rectangle (7,6);\node at(6.5,6.5){\tiny\( \bullet \)};
\filldraw[red,draw=black] (7,7) rectangle (8,6);\node at(7.5,6.5){\tiny\( \bullet \)};
\filldraw[red,draw=black] (8,7) rectangle (9,6);\node at(8.5,6.5){\tiny\( \bullet \)};
\filldraw[red,draw=black] (9,7) rectangle (10,6);\node at(9.5,6.5){\tiny\( \bullet \)};
\filldraw[red,draw=black] (10,7) rectangle (11,6);\node at(10.5,6.5){\tiny\( \bullet \)};
\filldraw[red,draw=black] (11,7) rectangle (12,6);\node at(11.5,6.5){\tiny\( \bullet \)};
\filldraw[red,draw=black] (12,7) rectangle (13,6);\node at(12.5,6.5){\tiny\( \bullet \)};
\draw (0,6) rectangle (1,5);
\draw (1,6) rectangle (2,5);\node at(1.5,5.5){\tiny\( 6 \)};
\draw (2,6) rectangle (3,5);
\filldraw[red,draw=black] (3,6) rectangle (4,5);
\filldraw[red,draw=black] (4,6) rectangle (5,5);\node at(4.5,5.5){\tiny\( \bullet \)};
\filldraw[red,draw=black] (5,6) rectangle (6,5);\node at(5.5,5.5){\tiny\( \bullet \)};
\filldraw[red,draw=black] (6,6) rectangle (7,5);\node at(6.5,5.5){\tiny\( \bullet \)};
\filldraw[red,draw=black] (7,6) rectangle (8,5);\node at(7.5,5.5){\tiny\( \bullet \)};
\filldraw[red,draw=black] (8,6) rectangle (9,5);\node at(8.5,5.5){\tiny\( \bullet \)};
\filldraw[red,draw=black] (9,6) rectangle (10,5);\node at(9.5,5.5){\tiny\( \bullet \)};
\filldraw[red,draw=black] (10,6) rectangle (11,5);\node at(10.5,5.5){\tiny\( \bullet \)};
\filldraw[red,draw=black] (11,6) rectangle (12,5);\node at(11.5,5.5){\tiny\( \bullet \)};
\filldraw[red,draw=black] (12,6) rectangle (13,5);\node at(12.5,5.5){\tiny\( \bullet \)};
\draw (0,5) rectangle (1,4);
\draw (1,5) rectangle (2,4);
\draw (2,5) rectangle (3,4);\node at(2.5,4.5){\tiny\( 5 \)};
\draw (3,5) rectangle (4,4);
\draw (4,5) rectangle (5,4);
\filldraw[red,draw=black] (5,5) rectangle (6,4);
\filldraw[red,draw=black] (6,5) rectangle (7,4);\node at(6.5,4.5){\tiny\( \bullet \)};
\filldraw[red,draw=black] (7,5) rectangle (8,4);\node at(7.5,4.5){\tiny\( \bullet \)};
\filldraw[red,draw=black] (8,5) rectangle (9,4);\node at(8.5,4.5){\tiny\( \bullet \)};
\filldraw[red,draw=black] (9,5) rectangle (10,4);\node at(9.5,4.5){\tiny\( \bullet \)};
\filldraw[red,draw=black] (10,5) rectangle (11,4);\node at(10.5,4.5){\tiny\( \bullet \)};
\filldraw[red,draw=black] (11,5) rectangle (12,4);\node at(11.5,4.5){\tiny\( \bullet \)};
\filldraw[red,draw=black] (12,5) rectangle (13,4);\node at(12.5,4.5){\tiny\( \bullet \)};
\draw (0,4) rectangle (1,3);
\draw (1,4) rectangle (2,3);
\draw (2,4) rectangle (3,3);
\draw (3,4) rectangle (4,3);\node at(3.5,3.5){\tiny\( 5 \)};
\draw (4,4) rectangle (5,3);
\draw (5,4) rectangle (6,3);
\filldraw[red,draw=black] (6,4) rectangle (7,3);
\filldraw[red,draw=black] (7,4) rectangle (8,3);\node at(7.5,3.5){\tiny\( \bullet \)};
\filldraw[red,draw=black] (8,4) rectangle (9,3);\node at(8.5,3.5){\tiny\( \bullet \)};
\filldraw[red,draw=black] (9,4) rectangle (10,3);\node at(9.5,3.5){\tiny\( \bullet \)};
\filldraw[red,draw=black] (10,4) rectangle (11,3);\node at(10.5,3.5){\tiny\( \bullet \)};
\filldraw[red,draw=black] (11,4) rectangle (12,3);\node at(11.5,3.5){\tiny\( \bullet \)};
\filldraw[red,draw=black] (12,4) rectangle (13,3);\node at(12.5,3.5){\tiny\( \bullet \)};
\draw (0,3) rectangle (1,2);
\draw (1,3) rectangle (2,2);
\draw (2,3) rectangle (3,2);
\draw (3,3) rectangle (4,2);
\draw (4,3) rectangle (5,2);\node at(4.5,2.5){\tiny\( 4 \)};
\draw (5,3) rectangle (6,2);
\draw (6,3) rectangle (7,2);
\draw (7,3) rectangle (8,2);
\filldraw[red,draw=black] (8,3) rectangle (9,2);
\filldraw[red,draw=black] (9,3) rectangle (10,2);\node at(9.5,2.5){\tiny\( \bullet \)};
\filldraw[red,draw=black] (10,3) rectangle (11,2);\node at(10.5,2.5){\tiny\( \bullet \)};
\filldraw[red,draw=black] (11,3) rectangle (12,2);\node at(11.5,2.5){\tiny\( \bullet \)};
\filldraw[red,draw=black] (12,3) rectangle (13,2);\node at(12.5,2.5){\tiny\( \bullet \)};
\draw (0,2) rectangle (1,1);
\draw (1,2) rectangle (2,1);
\draw (2,2) rectangle (3,1);
\draw (3,2) rectangle (4,1);
\draw (4,2) rectangle (5,1);
\draw (5,2) rectangle (6,1);\node at(5.5,1.5){\tiny\( 4 \)};
\draw (6,2) rectangle (7,1);
\draw (7,2) rectangle (8,1);
\draw (8,2) rectangle (9,1);
\filldraw[red,draw=black] (9,2) rectangle (10,1);4
\filldraw[red,draw=black] (10,2) rectangle (11,1);\node at(10.5,1.5){\tiny\( \bullet \)};
\filldraw[red,draw=black] (11,2) rectangle (12,1);\node at(11.5,1.5){\tiny\( \bullet \)};
\filldraw[red,draw=black] (12,2) rectangle (13,1);\node at(12.5,1.5){\tiny\( \bullet \)};
\draw (0,1) rectangle (1,0);
\draw (1,1) rectangle (2,0);
\draw (2,1) rectangle (3,0);
\draw (3,1) rectangle (4,0);
\draw (4,1) rectangle (5,0);
\draw (5,1) rectangle (6,0);
\draw (6,1) rectangle (7,0);\node at(6.5,0.5){\tiny\( 4 \)};
\draw (7,1) rectangle (8,0);
\draw (8,1) rectangle (9,0);
\draw (9,1) rectangle (10,0);
\filldraw[red,draw=black] (10,1) rectangle (11,0);
\filldraw[red,draw=black] (11,1) rectangle (12,0);\node at(11.5,0.5){\tiny\( \bullet \)};
\filldraw[red,draw=black] (12,1) rectangle (13,0);\node at(12.5,0.5){\tiny\( \bullet \)};
\draw (0,0) rectangle (1,-1);
\draw (1,0) rectangle (2,-1);
\draw (2,0) rectangle (3,-1);
\draw (3,0) rectangle (4,-1);
\draw (4,0) rectangle (5,-1);
\draw (5,0) rectangle (6,-1);
\draw (6,0) rectangle (7,-1);
\draw (7,0) rectangle (8,-1);\node at(7.5,-0.5){\tiny\( 4 \)};
\draw (8,0) rectangle (9,-1);
\draw (9,0) rectangle (10,-1);
\draw (10,0) rectangle (11,-1);
\filldraw[red,draw=black] (11,0) rectangle (12,-1);
\filldraw[red,draw=black] (12,0) rectangle (13,-1);\node at(12.5,-0.5){\tiny\( \bullet \)};
\draw (0,-1) rectangle (1,-2);
\draw (1,-1) rectangle (2,-2);
\draw (2,-1) rectangle (3,-2);
\draw (3,-1) rectangle (4,-2);
\draw (4,-1) rectangle (5,-2);
\draw (5,-1) rectangle (6,-2);
\draw (6,-1) rectangle (7,-2);
\draw (7,-1) rectangle (8,-2);
\draw (8,-1) rectangle (9,-2);\node at(8.5,-1.5){\tiny\( 3 \)};
\draw (9,-1) rectangle (10,-2);
\draw (10,-1) rectangle (11,-2);
\draw (11,-1) rectangle (12,-2);
\draw (12,-1) rectangle (13,-2);
\draw (0,-2) rectangle (1,-3);
\draw (1,-2) rectangle (2,-3);
\draw (2,-2) rectangle (3,-3);
\draw (3,-2) rectangle (4,-3);
\draw (4,-2) rectangle (5,-3);
\draw (5,-2) rectangle (6,-3);
\draw (6,-2) rectangle (7,-3);
\draw (7,-2) rectangle (8,-3);
\draw (8,-2) rectangle (9,-3);
\draw (9,-2) rectangle (10,-3);\node at(9.5,-2.5){\tiny\( 3 \)};
\draw (10,-2) rectangle (11,-3);
\draw (11,-2) rectangle (12,-3);
\draw (12,-2) rectangle (13,-3);
\draw (0,-3) rectangle (1,-4);
\draw (1,-3) rectangle (2,-4);
\draw (2,-3) rectangle (3,-4);
\draw (3,-3) rectangle (4,-4);
\draw (4,-3) rectangle (5,-4);
\draw (5,-3) rectangle (6,-4);
\draw (6,-3) rectangle (7,-4);
\draw (7,-3) rectangle (8,-4);
\draw (8,-3) rectangle (9,-4);
\draw (9,-3) rectangle (10,-4);
\draw (10,-3) rectangle (11,-4);\node at(10.5,-3.5){\tiny\( 2 \)};
\draw (11,-3) rectangle (12,-4);
\draw (12,-3) rectangle (13,-4);
\draw (0,-4) rectangle (1,-5);
\draw (1,-4) rectangle (2,-5);
\draw (2,-4) rectangle (3,-5);
\draw (3,-4) rectangle (4,-5);
\draw (4,-4) rectangle (5,-5);
\draw (5,-4) rectangle (6,-5);
\draw (6,-4) rectangle (7,-5);
\draw (7,-4) rectangle (8,-5);
\draw (8,-4) rectangle (9,-5);
\draw (9,-4) rectangle (10,-5);
\draw (10,-4) rectangle (11,-5);
\draw (11,-4) rectangle (12,-5);\node at(11.5,-4.5){\tiny\( 2 \)};
\draw (12,-4) rectangle (13,-5);
\draw (0,-5) rectangle (1,-6);
\draw (1,-5) rectangle (2,-6);
\draw (2,-5) rectangle (3,-6);
\draw (3,-5) rectangle (4,-6);
\draw (4,-5) rectangle (5,-6);
\draw (5,-5) rectangle (6,-6);
\draw (6,-5) rectangle (7,-6);
\draw (7,-5) rectangle (8,-6);
\draw (8,-5) rectangle (9,-6);
\draw (9,-5) rectangle (10,-6);
\draw (10,-5) rectangle (11,-6);
\draw (11,-5) rectangle (12,-6);
\draw (12,-5) rectangle (13,-6);\node at(12.5,-5.5){\tiny\( 1 \)};
\draw[purple,line width=0.8pt] (4,3.5)--(6,3.5);
\draw[purple,line width=0.8pt] (3.5,4)--(3.5,5);
\draw[purple,line width=0.8pt] (6.5,3)--(6.5,1);
\draw[purple,line width=0.8pt] (4,3.5)--(6,3.5);
\draw[purple,line width=0.8pt] (7,0.5)--(10,0.5);
%
\draw[black,line width=0.8pt,->] (13,5.5)--(13.8,5.5);
\node at (14.1,5.5){ \tiny{\textbf{$2$}} };
\draw[black,line width=0.8pt,->] (13,0.5)--(13.8,0.5);
\node at (14.1,0.5){ \tiny{\textbf{$7$}} };
\draw[black,line width=0.8pt,->] (10.5,7)--(10.5,7.8);
\node at (10.5,8.2){ \tiny{\textbf{$11$}} };
\end{tikzpicture}
=
L_{11} \,
\begin{tikzpicture}[scale=.3,line width=0.5pt,baseline=(a.base)]
\draw (0,7) rectangle (1,6);\node at(0.5,6.5){\tiny\( 7 \)};
\filldraw[red,draw=black] (1,7) rectangle (2,6);\node at(1.5,6.5){\tiny\( \bullet \)};
\filldraw[red,draw=black] (2,7) rectangle (3,6);\node at(2.5,6.5){\tiny\( \bullet \)};
\filldraw[red,draw=black] (3,7) rectangle (4,6);\node at(3.5,6.5){\tiny\( \bullet \)};
\filldraw[red,draw=black] (4,7) rectangle (5,6);\node at(4.5,6.5){\tiny\( \bullet \)};
\filldraw[red,draw=black] (5,7) rectangle (6,6);\node at(5.5,6.5){\tiny\( \bullet \)};
\filldraw[red,draw=black] (6,7) rectangle (7,6);\node at(6.5,6.5){\tiny\( \bullet \)};
\filldraw[red,draw=black] (7,7) rectangle (8,6);\node at(7.5,6.5){\tiny\( \bullet \)};
\filldraw[red,draw=black] (8,7) rectangle (9,6);\node at(8.5,6.5){\tiny\( \bullet \)};
\filldraw[red,draw=black] (9,7) rectangle (10,6);\node at(9.5,6.5){\tiny\( \bullet \)};
\filldraw[red,draw=black] (10,7) rectangle (11,6);\node at(10.5,6.5){\tiny\( \bullet \)};
\filldraw[red,draw=black] (11,7) rectangle (12,6);\node at(11.5,6.5){\tiny\( \bullet \)};
\filldraw[red,draw=black] (12,7) rectangle (13,6);\node at(12.5,6.5){\tiny\( \bullet \)};
\draw (0,6) rectangle (1,5);
\draw (1,6) rectangle (2,5);\node at(1.5,5.5){\tiny\( 6 \)};
\draw (2,6) rectangle (3,5);
\filldraw[red,draw=black] (3,6) rectangle (4,5);
\filldraw[red,draw=black] (4,6) rectangle (5,5);\node at(4.5,5.5){\tiny\( \bullet \)};
\filldraw[red,draw=black] (5,6) rectangle (6,5);\node at(5.5,5.5){\tiny\( \bullet \)};
\filldraw[red,draw=black] (6,6) rectangle (7,5);\node at(6.5,5.5){\tiny\( \bullet \)};
\filldraw[red,draw=black] (7,6) rectangle (8,5);\node at(7.5,5.5){\tiny\( \bullet \)};
\filldraw[red,draw=black] (8,6) rectangle (9,5);\node at(8.5,5.5){\tiny\( \bullet \)};
\filldraw[red,draw=black] (9,6) rectangle (10,5);\node at(9.5,5.5){\tiny\( \bullet \)};
\filldraw[red,draw=black] (10,6) rectangle (11,5);\node at(10.5,5.5){\tiny\( \bullet \)};
\filldraw[red,draw=black] (11,6) rectangle (12,5);\node at(11.5,5.5){\tiny\( \bullet \)};
\filldraw[red,draw=black] (12,6) rectangle (13,5);\node at(12.5,5.5){\tiny\( \bullet \)};
\draw (0,5) rectangle (1,4);
\draw (1,5) rectangle (2,4);
\draw (2,5) rectangle (3,4);\node at(2.5,4.5){\tiny\( 5 \)};
\draw (3,5) rectangle (4,4);
\draw (4,5) rectangle (5,4);
\filldraw[red,draw=black] (5,5) rectangle (6,4);
\filldraw[red,draw=black] (6,5) rectangle (7,4);\node at(6.5,4.5){\tiny\( \bullet \)};
\filldraw[red,draw=black] (7,5) rectangle (8,4);\node at(7.5,4.5){\tiny\( \bullet \)};
\filldraw[red,draw=black] (8,5) rectangle (9,4);\node at(8.5,4.5){\tiny\( \bullet \)};
\filldraw[red,draw=black] (9,5) rectangle (10,4);\node at(9.5,4.5){\tiny\( \bullet \)};
\filldraw[red,draw=black] (10,5) rectangle (11,4);\node at(10.5,4.5){\tiny\( \bullet \)};
\filldraw[red,draw=black] (11,5) rectangle (12,4);\node at(11.5,4.5){\tiny\( \bullet \)};
\filldraw[red,draw=black] (12,5) rectangle (13,4);\node at(12.5,4.5){\tiny\( \bullet \)};
\draw (0,4) rectangle (1,3);
\draw (1,4) rectangle (2,3);
\draw (2,4) rectangle (3,3);
\draw (3,4) rectangle (4,3);\node at(3.5,3.5){\tiny\( 5 \)};
\draw (4,4) rectangle (5,3);
\draw (5,4) rectangle (6,3);
\filldraw[red,draw=black] (6,4) rectangle (7,3);
\filldraw[red,draw=black] (7,4) rectangle (8,3);\node at(7.5,3.5){\tiny\( \bullet \)};
\filldraw[red,draw=black] (8,4) rectangle (9,3);\node at(8.5,3.5){\tiny\( \bullet \)};
\filldraw[red,draw=black] (9,4) rectangle (10,3);\node at(9.5,3.5){\tiny\( \bullet \)};
\filldraw[red,draw=black] (10,4) rectangle (11,3);\node at(10.5,3.5){\tiny\( \bullet \)};
\filldraw[red,draw=black] (11,4) rectangle (12,3);\node at(11.5,3.5){\tiny\( \bullet \)};
\filldraw[red,draw=black] (12,4) rectangle (13,3);\node at(12.5,3.5){\tiny\( \bullet \)};
\draw (0,3) rectangle (1,2);
\draw (1,3) rectangle (2,2);
\draw (2,3) rectangle (3,2);
\draw (3,3) rectangle (4,2);
\draw (4,3) rectangle (5,2);\node at(4.5,2.5){\tiny\( 4 \)};
\draw (5,3) rectangle (6,2);
\draw (6,3) rectangle (7,2);
\draw (7,3) rectangle (8,2);
\filldraw[red,draw=black] (8,3) rectangle (9,2);
\filldraw[red,draw=black] (9,3) rectangle (10,2);\node at(9.5,2.5){\tiny\( \bullet \)};
\filldraw[red,draw=black] (10,3) rectangle (11,2);\node at(10.5,2.5){\tiny\( \bullet \)};
\filldraw[red,draw=black] (11,3) rectangle (12,2);\node at(11.5,2.5){\tiny\( \bullet \)};
\filldraw[red,draw=black] (12,3) rectangle (13,2);\node at(12.5,2.5){\tiny\( \bullet \)};
\draw (0,2) rectangle (1,1);
\draw (1,2) rectangle (2,1);
\draw (2,2) rectangle (3,1);
\draw (3,2) rectangle (4,1);
\draw (4,2) rectangle (5,1);
\draw (5,2) rectangle (6,1);\node at(5.5,1.5){\tiny\( 4 \)};
\draw (6,2) rectangle (7,1);
\draw (7,2) rectangle (8,1);
\draw (8,2) rectangle (9,1);
\filldraw[red,draw=black] (9,2) rectangle (10,1);4
\filldraw[red,draw=black] (10,2) rectangle (11,1);\node at(10.5,1.5){\tiny\( \bullet \)};
\filldraw[red,draw=black] (11,2) rectangle (12,1);\node at(11.5,1.5){\tiny\( \bullet \)};
\filldraw[red,draw=black] (12,2) rectangle (13,1);\node at(12.5,1.5){\tiny\( \bullet \)};
\draw (0,1) rectangle (1,0);
\draw (1,1) rectangle (2,0);
\draw (2,1) rectangle (3,0);
\draw (3,1) rectangle (4,0);
\draw (4,1) rectangle (5,0);
\draw (5,1) rectangle (6,0);
\draw (6,1) rectangle (7,0);\node at(6.5,0.5){\tiny\( 4 \)};
\draw (7,1) rectangle (8,0);
\draw (8,1) rectangle (9,0);
\draw (9,1) rectangle (10,0);
\filldraw[red,draw=black] (10,1) rectangle (11,0);
\filldraw[red,draw=black] (11,1) rectangle (12,0);\node at(11.5,0.5){\tiny\( \bullet \)};
\filldraw[red,draw=black] (12,1) rectangle (13,0);\node at(12.5,0.5){\tiny\( \bullet \)};
\draw (0,0) rectangle (1,-1);
\draw (1,0) rectangle (2,-1);
\draw (2,0) rectangle (3,-1);
\draw (3,0) rectangle (4,-1);
\draw (4,0) rectangle (5,-1);
\draw (5,0) rectangle (6,-1);
\draw (6,0) rectangle (7,-1);
\draw (7,0) rectangle (8,-1);\node at(7.5,-0.5){\tiny\( 4 \)};
\draw (8,0) rectangle (9,-1);
\draw (9,0) rectangle (10,-1);
\draw (10,0) rectangle (11,-1);
\filldraw[red,draw=black] (11,0) rectangle (12,-1);
\filldraw[red,draw=black] (12,0) rectangle (13,-1);\node at(12.5,-0.5){\tiny\( \bullet \)};
\draw (0,-1) rectangle (1,-2);
\draw (1,-1) rectangle (2,-2);
\draw (2,-1) rectangle (3,-2);
\draw (3,-1) rectangle (4,-2);
\draw (4,-1) rectangle (5,-2);
\draw (5,-1) rectangle (6,-2);
\draw (6,-1) rectangle (7,-2);
\draw (7,-1) rectangle (8,-2);
\draw (8,-1) rectangle (9,-2);\node at(8.5,-1.5){\tiny\( 3 \)};
\draw (9,-1) rectangle (10,-2);
\draw (10,-1) rectangle (11,-2);
\draw (11,-1) rectangle (12,-2);
\draw (12,-1) rectangle (13,-2);
\draw (0,-2) rectangle (1,-3);
\draw (1,-2) rectangle (2,-3);
\draw (2,-2) rectangle (3,-3);
\draw (3,-2) rectangle (4,-3);
\draw (4,-2) rectangle (5,-3);
\draw (5,-2) rectangle (6,-3);
\draw (6,-2) rectangle (7,-3);
\draw (7,-2) rectangle (8,-3);
\draw (8,-2) rectangle (9,-3);
\draw (9,-2) rectangle (10,-3);\node at(9.5,-2.5){\tiny\( 3 \)};
\draw (10,-2) rectangle (11,-3);
\draw (11,-2) rectangle (12,-3);
\draw (12,-2) rectangle (13,-3);
\draw (0,-3) rectangle (1,-4);
\draw (1,-3) rectangle (2,-4);
\draw (2,-3) rectangle (3,-4);
\draw (3,-3) rectangle (4,-4);
\draw (4,-3) rectangle (5,-4);
\draw (5,-3) rectangle (6,-4);
\draw (6,-3) rectangle (7,-4);
\draw (7,-3) rectangle (8,-4);
\draw (8,-3) rectangle (9,-4);
\draw (9,-3) rectangle (10,-4);
\draw (10,-3) rectangle (11,-4);\node at(10.5,-3.5){\tiny\( 2 \)};
\draw (11,-3) rectangle (12,-4);
\draw (12,-3) rectangle (13,-4);
\draw (0,-4) rectangle (1,-5);
\draw (1,-4) rectangle (2,-5);
\draw (2,-4) rectangle (3,-5);
\draw (3,-4) rectangle (4,-5);
\draw (4,-4) rectangle (5,-5);
\draw (5,-4) rectangle (6,-5);
\draw (6,-4) rectangle (7,-5);
\draw (7,-4) rectangle (8,-5);
\draw (8,-4) rectangle (9,-5);
\draw (9,-4) rectangle (10,-5);
\draw (10,-4) rectangle (11,-5);
\draw (11,-4) rectangle (12,-5);\node at(11.5,-4.5){\tiny\( 2 \)};
\draw (12,-4) rectangle (13,-5);
\draw (0,-5) rectangle (1,-6);
\draw (1,-5) rectangle (2,-6);
\draw (2,-5) rectangle (3,-6);
\draw (3,-5) rectangle (4,-6);
\draw (4,-5) rectangle (5,-6);
\draw (5,-5) rectangle (6,-6);
\draw (6,-5) rectangle (7,-6);
\draw (7,-5) rectangle (8,-6);
\draw (8,-5) rectangle (9,-6);
\draw (9,-5) rectangle (10,-6);
\draw (10,-5) rectangle (11,-6);
\draw (11,-5) rectangle (12,-6);
\draw (12,-5) rectangle (13,-6);\node at(12.5,-5.5){\tiny\( 1 \)};
%
\end{tikzpicture}
\, .
\end{equation*}
\hspace*{\fill}$\square$
\mlabel{ex:L7}
\end{exam}

\subsection{A recursive formula for $\sg{\lambda}{z+1}$}
In this subsection, we express the $K$-$k$-Schur functions $\sg{\lambda}{z+1}$ of weight $z+1$ in terms of those of lower weight $z$.

\begin{prop}
Let $\lambda\in\pkl$ and $z\in[\bott_\lambda]$. Suppose $\lambda_{z}>\lambda_{z+1}$ when $z\in[\bott_\lambda-1]$. Then
\begin{equation}
\sg{\lambda}{z+1} =\sum_{\mu\in\pkl, \, \mu_x = \lambda_x\,\textup{for}\,x\in[z]}b_{\lambda\mu}\sg{\mu}{z}, \quad \text{where }\, (-1)^{|\lambda|-|\mu|} b_{\lambda\mu}\in\ZZ_{\geq0}.
\mlabel{eq:ind}
\end{equation}
\mlabel{prop:ind}
\end{prop}

Let us pause for a moment to illustrate the main idea of the proof with an example.

\begin{exam}
Let
\[
k=7, \quad \ell=13, \quad \lambda = (7,6,5,5,4,4,4,3,3,3,2,2,1)\in{\rm P}_{13}^7.
\]
Then $b_\lambda = 8$. Consider $z=2$.
The weighted $K$-$k$-Schur functions $\sg{\lambda}{2}$ is represented by the Katalan function:
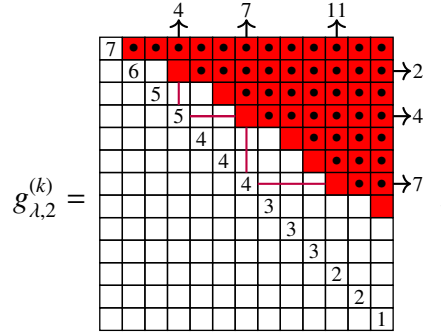
\begin{equation*}
\sg{\lambda}{2} =
\begin{tikzpicture}[scale=.3,line width=0.5pt,baseline=(a.base)]
\draw (0,7) rectangle (1,6);\node at(0.5,6.5){\tiny\( 7 \)};
\filldraw[red,draw=black] (1,7) rectangle (2,6);\node at(1.5,6.5){\tiny\( \bullet \)};
\filldraw[red,draw=black] (2,7) rectangle (3,6);\node at(2.5,6.5){\tiny\( \bullet \)};
\filldraw[red,draw=black] (3,7) rectangle (4,6);\node at(3.5,6.5){\tiny\( \bullet \)};
\filldraw[red,draw=black] (4,7) rectangle (5,6);\node at(4.5,6.5){\tiny\( \bullet \)};
\filldraw[red,draw=black] (5,7) rectangle (6,6);\node at(5.5,6.5){\tiny\( \bullet \)};
\filldraw[red,draw=black] (6,7) rectangle (7,6);\node at(6.5,6.5){\tiny\( \bullet \)};
\filldraw[red,draw=black] (7,7) rectangle (8,6);\node at(7.5,6.5){\tiny\( \bullet \)};
\filldraw[red,draw=black] (8,7) rectangle (9,6);\node at(8.5,6.5){\tiny\( \bullet \)};
\filldraw[red,draw=black] (9,7) rectangle (10,6);\node at(9.5,6.5){\tiny\( \bullet \)};
\filldraw[red,draw=black] (10,7) rectangle (11,6);\node at(10.5,6.5){\tiny\( \bullet \)};
\filldraw[red,draw=black] (11,7) rectangle (12,6);\node at(11.5,6.5){\tiny\( \bullet \)};
\filldraw[red,draw=black] (12,7) rectangle (13,6);\node at(12.5,6.5){\tiny\( \bullet \)};
\draw (0,6) rectangle (1,5);
\draw (1,6) rectangle (2,5);\node at(1.5,5.5){\tiny\( 6 \)};
\draw (2,6) rectangle (3,5);
\filldraw[red,draw=black] (3,6) rectangle (4,5);
\filldraw[red,draw=black] (4,6) rectangle (5,5);\node at(4.5,5.5){\tiny\( \bullet \)};
\filldraw[red,draw=black] (5,6) rectangle (6,5);\node at(5.5,5.5){\tiny\( \bullet \)};
\filldraw[red,draw=black] (6,6) rectangle (7,5);\node at(6.5,5.5){\tiny\( \bullet \)};
\filldraw[red,draw=black] (7,6) rectangle (8,5);\node at(7.5,5.5){\tiny\( \bullet \)};
\filldraw[red,draw=black] (8,6) rectangle (9,5);\node at(8.5,5.5){\tiny\( \bullet \)};
\filldraw[red,draw=black] (9,6) rectangle (10,5);\node at(9.5,5.5){\tiny\( \bullet \)};
\filldraw[red,draw=black] (10,6) rectangle (11,5);\node at(10.5,5.5){\tiny\( \bullet \)};
\filldraw[red,draw=black] (11,6) rectangle (12,5);\node at(11.5,5.5){\tiny\( \bullet \)};
\filldraw[red,draw=black] (12,6) rectangle (13,5);\node at(12.5,5.5){\tiny\( \bullet \)};
\draw (0,5) rectangle (1,4);
\draw (1,5) rectangle (2,4);
\draw (2,5) rectangle (3,4);\node at(2.5,4.5){\tiny\( 5 \)};
\draw (3,5) rectangle (4,4);
\draw (4,5) rectangle (5,4);
\filldraw[red,draw=black] (5,5) rectangle (6,4);
\filldraw[red,draw=black] (6,5) rectangle (7,4);\node at(6.5,4.5){\tiny\( \bullet \)};
\filldraw[red,draw=black] (7,5) rectangle (8,4);\node at(7.5,4.5){\tiny\( \bullet \)};
\filldraw[red,draw=black] (8,5) rectangle (9,4);\node at(8.5,4.5){\tiny\( \bullet \)};
\filldraw[red,draw=black] (9,5) rectangle (10,4);\node at(9.5,4.5){\tiny\( \bullet \)};
\filldraw[red,draw=black] (10,5) rectangle (11,4);\node at(10.5,4.5){\tiny\( \bullet \)};
\filldraw[red,draw=black] (11,5) rectangle (12,4);\node at(11.5,4.5){\tiny\( \bullet \)};
\filldraw[red,draw=black] (12,5) rectangle (13,4);\node at(12.5,4.5){\tiny\( \bullet \)};
\draw (0,4) rectangle (1,3);
\draw (1,4) rectangle (2,3);
\draw (2,4) rectangle (3,3);
\draw (3,4) rectangle (4,3);\node at(3.5,3.5){\tiny\( 5 \)};
\draw (4,4) rectangle (5,3);
\draw (5,4) rectangle (6,3);
\filldraw[red,draw=black] (6,4) rectangle (7,3);
\filldraw[red,draw=black] (7,4) rectangle (8,3);\node at(7.5,3.5){\tiny\( \bullet \)};
\filldraw[red,draw=black] (8,4) rectangle (9,3);\node at(8.5,3.5){\tiny\( \bullet \)};
\filldraw[red,draw=black] (9,4) rectangle (10,3);\node at(9.5,3.5){\tiny\( \bullet \)};
\filldraw[red,draw=black] (10,4) rectangle (11,3);\node at(10.5,3.5){\tiny\( \bullet \)};
\filldraw[red,draw=black] (11,4) rectangle (12,3);\node at(11.5,3.5){\tiny\( \bullet \)};
\filldraw[red,draw=black] (12,4) rectangle (13,3);\node at(12.5,3.5){\tiny\( \bullet \)};
\draw (0,3) rectangle (1,2);
\draw (1,3) rectangle (2,2);
\draw (2,3) rectangle (3,2);
\draw (3,3) rectangle (4,2);
\draw (4,3) rectangle (5,2);\node at(4.5,2.5){\tiny\( 4 \)};
\draw (5,3) rectangle (6,2);
\draw (6,3) rectangle (7,2);
\draw (7,3) rectangle (8,2);
\filldraw[red,draw=black] (8,3) rectangle (9,2);
\filldraw[red,draw=black] (9,3) rectangle (10,2);\node at(9.5,2.5){\tiny\( \bullet \)};
\filldraw[red,draw=black] (10,3) rectangle (11,2);\node at(10.5,2.5){\tiny\( \bullet \)};
\filldraw[red,draw=black] (11,3) rectangle (12,2);\node at(11.5,2.5){\tiny\( \bullet \)};
\filldraw[red,draw=black] (12,3) rectangle (13,2);\node at(12.5,2.5){\tiny\( \bullet \)};
\draw (0,2) rectangle (1,1);
\draw (1,2) rectangle (2,1);
\draw (2,2) rectangle (3,1);
\draw (3,2) rectangle (4,1);
\draw (4,2) rectangle (5,1);
\draw (5,2) rectangle (6,1);\node at(5.5,1.5){\tiny\( 4 \)};
\draw (6,2) rectangle (7,1);
\draw (7,2) rectangle (8,1);
\draw (8,2) rectangle (9,1);
\filldraw[red,draw=black] (9,2) rectangle (10,1);4
\filldraw[red,draw=black] (10,2) rectangle (11,1);\node at(10.5,1.5){\tiny\( \bullet \)};
\filldraw[red,draw=black] (11,2) rectangle (12,1);\node at(11.5,1.5){\tiny\( \bullet \)};
\filldraw[red,draw=black] (12,2) rectangle (13,1);\node at(12.5,1.5){\tiny\( \bullet \)};
\draw (0,1) rectangle (1,0);
\draw (1,1) rectangle (2,0);
\draw (2,1) rectangle (3,0);
\draw (3,1) rectangle (4,0);
\draw (4,1) rectangle (5,0);
\draw (5,1) rectangle (6,0);
\draw (6,1) rectangle (7,0);\node at(6.5,0.5){\tiny\( 4 \)};
\draw (7,1) rectangle (8,0);
\draw (8,1) rectangle (9,0);
\draw (9,1) rectangle (10,0);
\filldraw[red,draw=black] (10,1) rectangle (11,0);
\filldraw[red,draw=black] (11,1) rectangle (12,0);\node at(11.5,0.5){\tiny\( \bullet \)};
\filldraw[red,draw=black] (12,1) rectangle (13,0);\node at(12.5,0.5){\tiny\( \bullet \)};
\draw (0,0) rectangle (1,-1);
\draw (1,0) rectangle (2,-1);
\draw (2,0) rectangle (3,-1);
\draw (3,0) rectangle (4,-1);
\draw (4,0) rectangle (5,-1);
\draw (5,0) rectangle (6,-1);
\draw (6,0) rectangle (7,-1);
\draw (7,0) rectangle (8,-1);\node at(7.5,-0.5){\tiny\( 3 \)};
\draw (8,0) rectangle (9,-1);
\draw (9,0) rectangle (10,-1);
\draw (10,0) rectangle (11,-1);
\draw (11,0) rectangle (12,-1);
\filldraw[red,draw=black] (12,0) rectangle (13,-1);
\draw (0,-1) rectangle (1,-2);
\draw (1,-1) rectangle (2,-2);
\draw (2,-1) rectangle (3,-2);
\draw (3,-1) rectangle (4,-2);
\draw (4,-1) rectangle (5,-2);
\draw (5,-1) rectangle (6,-2);
\draw (6,-1) rectangle (7,-2);
\draw (7,-1) rectangle (8,-2);
\draw (8,-1) rectangle (9,-2);\node at(8.5,-1.5){\tiny\( 3 \)};
\draw (9,-1) rectangle (10,-2);
\draw (10,-1) rectangle (11,-2);
\draw (11,-1) rectangle (12,-2);
\draw (12,-1) rectangle (13,-2);
\draw (0,-2) rectangle (1,-3);
\draw (1,-2) rectangle (2,-3);
\draw (2,-2) rectangle (3,-3);
\draw (3,-2) rectangle (4,-3);
\draw (4,-2) rectangle (5,-3);
\draw (5,-2) rectangle (6,-3);
\draw (6,-2) rectangle (7,-3);
\draw (7,-2) rectangle (8,-3);
\draw (8,-2) rectangle (9,-3);
\draw (9,-2) rectangle (10,-3);\node at(9.5,-2.5){\tiny\( 3 \)};
\draw (10,-2) rectangle (11,-3);
\draw (11,-2) rectangle (12,-3);
\draw (12,-2) rectangle (13,-3);
\draw (0,-3) rectangle (1,-4);
\draw (1,-3) rectangle (2,-4);
\draw (2,-3) rectangle (3,-4);
\draw (3,-3) rectangle (4,-4);
\draw (4,-3) rectangle (5,-4);
\draw (5,-3) rectangle (6,-4);
\draw (6,-3) rectangle (7,-4);
\draw (7,-3) rectangle (8,-4);
\draw (8,-3) rectangle (9,-4);
\draw (9,-3) rectangle (10,-4);
\draw (10,-3) rectangle (11,-4);\node at(10.5,-3.5){\tiny\( 2 \)};
\draw (11,-3) rectangle (12,-4);
\draw (12,-3) rectangle (13,-4);
\draw (0,-4) rectangle (1,-5);
\draw (1,-4) rectangle (2,-5);
\draw (2,-4) rectangle (3,-5);
\draw (3,-4) rectangle (4,-5);
\draw (4,-4) rectangle (5,-5);
\draw (5,-4) rectangle (6,-5);
\draw (6,-4) rectangle (7,-5);
\draw (7,-4) rectangle (8,-5);
\draw (8,-4) rectangle (9,-5);
\draw (9,-4) rectangle (10,-5);
\draw (10,-4) rectangle (11,-5);
\draw (11,-4) rectangle (12,-5);\node at(11.5,-4.5){\tiny\( 2 \)};
\draw (12,-4) rectangle (13,-5);
\draw (0,-5) rectangle (1,-6);
\draw (1,-5) rectangle (2,-6);
\draw (2,-5) rectangle (3,-6);
\draw (3,-5) rectangle (4,-6);
\draw (4,-5) rectangle (5,-6);
\draw (5,-5) rectangle (6,-6);
\draw (6,-5) rectangle (7,-6);
\draw (7,-5) rectangle (8,-6);
\draw (8,-5) rectangle (9,-6);
\draw (9,-5) rectangle (10,-6);
\draw (10,-5) rectangle (11,-6);
\draw (11,-5) rectangle (12,-6);
\draw (12,-5) rectangle (13,-6);\node at(12.5,-5.5){\tiny\( 1 \)};
\draw[purple,line width=0.8pt] (4,3.5)--(6,3.5);
\draw[purple,line width=0.8pt] (3.5,4)--(3.5,5);
\draw[purple,line width=0.8pt] (6.5,3)--(6.5,1);
\draw[purple,line width=0.8pt] (4,3.5)--(6,3.5);
\draw[purple,line width=0.8pt] (7,0.5)--(10,0.5);
%
\draw[black,line width=0.8pt,->] (13,5.5)--(13.8,5.5);
\node at (14.1,5.5){ \tiny{\textbf{$2$}} };
\draw[black,line width=0.8pt,->] (13,3.5)--(13.8,3.5);
\node at (14.1,3.5){ \tiny{\textbf{$4$}} };
\draw[black,line width=0.8pt,->] (13,0.5)--(13.8,0.5);
\node at (14.1,0.5){ \tiny{\textbf{$7$}} };
\draw[black,line width=0.8pt,->] (3.5,7)--(3.5,7.8);
\node at (3.5,8.2){ \tiny{\textbf{$4$}} };
\draw[black,line width=0.8pt,->] (6.5,7)--(6.5,7.8);
\node at (6.5,8.2){ \tiny{\textbf{$7$}} };
\draw[black,line width=0.8pt,->] (10.5,7)--(10.5,7.8);
\node at (10.5,8.2){ \tiny{\textbf{$11$}} };
\end{tikzpicture}
\, .
\end{equation*}
We obtain $$2\in[\bott_\lambda-1], \quad \lambda_2 = 6>5=\lambda_3$$ and
\begin{align*}
\sg{\lambda}{3} =&\ \sg{\lambda}{2}-L_4\sg{\lambda}{2}\hspace{1cm}(\text{by Lemma~\ref{lem:ztoz1} and $\bdo{\lambda}{2} = 4$})\\
=&\ \sg{\lambda}{2}-\Big( \big( 1-L_8 \big)\sg{(7,6,5,\textcolor{green}{4},4,4,4,3,3,3,2,2,1)}{2} +L_7\sg{\lambda}{2} \Big)\hspace{0.5cm}(\text{by Proposition~\ref{prop:smalld} and $\gamma = \lambda - \epsilon_4 \in{\rm P}_{13}^7$})\\
=&\ \sg{\lambda}{2}-\sg{(7,6,5,4,4,4,4,3,3,3,2,2,1)}{2}
+L_8\sg{(7,6,5,4,4,4,4,3,3,3,2,2,1)}{2}-L_7\sg{\lambda}{2}\\
=&\ \sg{\lambda}{2}-\sg{(7,6,5,4,4,4,4,3,3,3,2,2,1)}{2}
+L_{\textcolor{green}{13}}\sg{(7,6,5,4,4,4,4,3,3,3,2,2,1)}{2}-L_7\sg{\lambda}{2}\\
& \ (\text{by Proposition~\ref{prop:smalld} and $(7,6,5,4,4,4,4,\textcolor{green}{2},3,3,2,2,1)\notin{\rm P}_{13}^7$ for the third summand})\\
=&\ \sg{\lambda}{2}-\sg{(7,6,5,4,4,4,4,3,3,3,2,2,1)}{2}
+L_{13}\sg{(7,6,5,4,4,4,4,3,3,3,2,2,1)}{2}
\\
&\ -\Big( (1-L_{12})\sg{(7,6,5,5,4,4,\textcolor{green}{3},3,3,3,2,2,1)}{2}+ L_{11}\sg{\lambda}{2} \Big) \hspace{1cm} (\text{by Example~\ref{ex:L7}})\\
=&\ \sg{\lambda}{2}-\sg{(7,6,5,4,4,4,4,3,3,3,2,2,1)}{2}
-\sg{(7,6,5,5,4,4,3,3,3,3,2,2,1)}{2}+L_{13}\sg{(7,6,5,4,4,4,4,3,3,3,2,2,1)}{2}\\
&\ +L_{12}\sg{(7,6,5,5,4,4,3,3,3,3,2,2,1)}{2}-L_{11}\sg{\lambda}{2}\\
=&\ \sg{\lambda}{2}-\sg{(7,6,5,4,4,4,4,3,3,3,2,2,1)}{2}
-\sg{(7,6,5,5,4,4,3,3,3,3,2,2,1)}{2}
+\sg{(7,6,5,4,4,4,4,3,3,3,2,2,\textcolor{green}{0})}{2}\\
&\ +\sg{(7,6,5,5,4,4,3,3,3,3,2,\textcolor{green}{1},1)}{2}-0 \hspace{1cm} (\text{by Proposition~\ref{prop:bigd} for the last three summands}).
\end{align*}
The coefficients in the right hand side of the above equation
are alternating positive:
\begin{align*}
(-1)^{|\lambda|-|\lambda|}\cdot 1 =&\ 1,\\
(-1)^{|\lambda|-|(7,6,5,\textcolor{green}{4},4,4,4,3,3,3,2,2,1)|}\cdot(-1) = &\ (-1)\cdot(-1) = 1, \\
(-1)^{|\lambda|-|(7,6,5,5,4,4,\textcolor{green}{3},3,3,3,2,2,1)|}\cdot(-1) = &\ (-1)\cdot(-1) = 1, \\
(-1)^{|\lambda|
-|(7,6,5,\textcolor{green}{4},4,4,4,3,3,3,2,2,\textcolor{green}{0})|}\cdot 1 = &\ (-1)^2\cdot 1 = 1, \\
(-1)^{|\lambda|
-|(7,6,5,5,4,4,\textcolor{green}{3},3,3,3,2,\textcolor{green}{1},1)|}\cdot 1 = &\ (-1)^2\cdot 1 = 1.
\end{align*}
\hspace*{\fill}$\square$
\mlabel{exam:ind}
\end{exam}

\begin{proof}[proof of Proposition~\mref{prop:ind}]
Suppose first $z = \bott_\lambda$. Lemma~\mref{lem:ztoz1} and Proposition~\mref{prop:bigd} (taking $a=1$) imply that
\begin{equation*}
\sg{\lambda}{z+1} =
\left\{
\begin{array}{ll}
\sg{\lambda}{z} -\sg{\lambda-\epsilon_{\bdo{\lambda}{z}}}{z}, & \quad \text{if $\lambda-\epsilon_{\bdo{\lambda}{z}}\in\pkl$},\\
\sg{\lambda}{z}, & \quad \text{otherwise}.
\end{array}
\right.
\end{equation*}
This yields the required result by
$$ (-1)^{|\lambda|-|\lambda|}\cdot 1 = 1\in\ZZ_{\geq0}, \quad  (-1)^{|\lambda|-|\lambda-\epsilon_{\bdo{\lambda}{z}}|}\cdot (-1)= 1\in\ZZ_{\geq0}. $$

Suppose next $z\in[\bott_\lambda-1]$. Then $\lambda_{z}>\lambda_{z+1}$ by assumption.
Denote $c'\in\ZZ_{\geq0}$ such that $\bDo{\lambda}{c'}{z} = {\rm bot}_{\dkl}(z)$. Then $\bDo{\lambda}{c'}{z}\in[\bott_\lambda+1,\ell]$ by Remark~\mref{re:rootfact}~(1) and there is a bounce path
$${\rm path}_{\dkl}\Big(\bDo{\lambda}{1}{z},\bDo{\lambda}{c'}{z}\Big) = \Big( \bDo{\lambda}{1}{z}, \ldots, \bDo{\lambda}{c'}{z}\Big)$$
in the root ideal $\dkl$. We use the induction on $c\in[0, c'-1]$ to prove that, if $\lambda_z>\lambda_{z+1}$ with $\lambda\in\pkl$, then
\begin{equation}
L_{\bDo{\lambda}{c'-c}{z}}\sg{\lambda}{z} =\sum_{\mu\in\pkl, \, \mu_x = \lambda_x\,\textup{for}\,x\in[z]}b_{\lambda\mu}\sg{\mu}{z}, \quad \text{where }\, (-1)^{|\lambda|-|\mu|-1} b_{\lambda\mu}\in\ZZ_{\geq0}.
\mlabel{eq:indaim}
\end{equation}

The initial step of $c = 0$ holds by Proposition~\mref{prop:bigd}. Consider the inductive step of $c\in[c'-1]$. Note that $\bDo{\lambda}{c'-c}{z} \in [\bott_\lambda]$. Proposition~\mref{prop:smalld} implies that
\begin{equation}
L_{\bDo{\lambda}{c'-c}{z}}\sg{\lambda}{z} =
\left\{
\begin{array}{ll}
\big(1-L_{\bDo{\gamma}{c'-(c-1)}{z}}\big)\sg{\gamma}{z} + L_{\bDo{\lambda}{c'-(c-1)}{z}}\sg{\lambda}{z}, & \quad \text{if $\gamma:= \lambda-\epsilon_{\bDo{\lambda}{c'-c}{z}}\in\pkl$},\\
L_{\bDo{\lambda}{c'-(c-1)}{z}}\sg{\lambda}{z}, & \quad \text{otherwise}.
\end{array}
\right.
\mlabel{eq:indind}
\end{equation}
By the inductive hypothesis,
\begin{equation}
L_{\bDo{\lambda}{c'-(c-1)}{z}}\sg{\lambda}{z} = \sum_{\mu^{(2)}\in\pkl, \, \mu^{(2)}_x = \lambda_x\,\textup{for}\,x\in[z]}b_{\lambda\mu^{(2)}}\sg{\mu^{(2)}}{z}, \quad \text{where }\,(-1)^{|\lambda|-|\mu^{(2)}|-1} b_{\lambda\mu^{(2)}}\in\ZZ_{\geq0}.
\mlabel{eq:indhy2}
\end{equation}
According to~(\mref{eq:indind}).
We divided the remaining proof into the following two cases.

{\bf Case~1.} $\gamma= \lambda-\epsilon_{\bDo{\lambda}{c'-c}{z}}\in\pkl$.
In terms of the definition of $c'$, we have $\bDo{\lambda}{c'-c}{z} \neq z,z+1$ and so
\[
\gamma_z = \big(\lambda-\epsilon_{\bDo{\lambda}{c'-c}{z}}\big)_z = \lambda_z
>\lambda_{z+1} = \big(\lambda-\epsilon_{\bDo{\lambda}{c'-c}{z}}\big)_{z+1} = \gamma_{z+1}.
\]
By the inductive hypothesis,
\begin{align}
L_{\bDo{\gamma}{c'-(c-1)}{z}}\sg{\gamma}{z} =&\ \sum_{\mu^{(1)}\in\pkl, \, \mu^{(1)}_x = \gamma_x\,\textup{for}\,x\in[z]}b_{\gamma\mu^{(1)}}\sg{\mu^{(1)}}{z}, \quad\text{where }\,(-1)^{|\gamma|-|\mu^{(1)}|-1} b_{\gamma\mu^{(1)}}.
\mlabel{eq:indhy1}
\end{align}
Denote $b_{\lambda\mu^{(1)}}:= -b_{\gamma\mu^{(1)}}$. Then
\begin{equation}
\begin{aligned}
(-1)^{|\lambda|-|\mu^{(1)}|-1}b_{\lambda\mu^{(1)}} =&\
(-1)^{|\lambda|-|\mu^{(1)}|-1}\big(-b_{\gamma\mu^{(1)}}\big)\\
=&\
(-1)^{|\gamma|+1-|\mu^{(1)}|-1}\big(-b_{\gamma\mu^{(1)}}\big)\\
=&\ (-1)^{|\gamma|-|\mu^{(1)}|-1} b_{\gamma\mu^{(1)}} \\
\overset{(\ref{eq:indhy1})}{\in}&\  \ZZ_{\geq0}.
\end{aligned}
\mlabel{eq:gamtolam}
\end{equation}
Now we have
\begin{align*}
L_{\bDo{\lambda}{c'-c}{z}}\sg{\lambda}{z} =&\
\big(1-L_{\bDo{\gamma}{c'-(c-1)}{z}}\big)\sg{\gamma}{z} + L_{\bDo{\lambda}{c'-(c-1)}{z}}\sg{\lambda}{z} \hspace{1cm} (\text{by~(\ref{eq:indind})})\\
=&\
\sg{\gamma}{z} - \sum_{\mu^{(1)}\in\pkl, \, \mu^{(1)}_x = \gamma_x\,\textup{for}\,x\in[z]}b_{\gamma\mu^{(1)}}\sg{\mu^{(1)}}{z} + \sum_{\mu^{(2)}\in\pkl, \, \mu^{(2)}_x = \lambda_x\,\textup{for}\,x\in[z]}b_{\lambda\mu^{(2)}}\sg{\mu^{(2)}}{z}\\
& \hspace{7.5cm} (\text{by~(\ref{eq:indhy2}) and~(\ref{eq:indhy1}))}\\
=&\
\sg{\gamma}{z}+ \sum_{\mu^{(1)}\in\pkl, \, \mu^{(1)}_x = \lambda_x\,\textup{for}\,x\in[z]}b_{\lambda\mu^{(1)}}\sg{\mu^{(1)}}{z} + \sum_{\mu^{(2)}\in\pkl, \, \mu^{(2)}_x = \lambda_x\,\textup{for}\,x\in[z]}b_{\lambda\mu^{(2)}}\sg{\mu^{(2)}}{z}\\
& \hspace{3cm} (\text{by $b_{\lambda\mu^{(1)}}:= -b_{\gamma\mu^{(1)}}$ for the second summand})\\
=&\ \sum_{\mu\in\pkl, \, \mu_x = \lambda_x\,\textup{for}\,x\in[z]} b_{\lambda\mu}\sg{\mu}{z},
\end{align*}
where
\[
b_{\lambda\gamma}:=1, \quad  \mu := \mu^{(1)} = \mu^{(2)}, \quad b_{\lambda\mu} := b_{\lambda\mu^{(1)}} + b_{\lambda\mu^{(2)}}.
\]
Further,
\begin{align*}
(-1)^{|\lambda|-|\gamma|-1}\cdot 1 =&\ 1\in\ZZ_{\geq0},\\
(-1)^{|\lambda|-|\mu|-1} b_{\lambda\mu}  =&\ (-1)^{|\lambda|-|\mu|-1}(b_{\lambda\mu^{(1)}} + b_{\lambda\mu^{(2)}})\\
=&\ (-1)^{|\lambda|-|\mu^{(1)}|-1}b_{\lambda\mu^{(1)}} + (-1)^{|\lambda|-|\mu^{(2)}|-1}b_{\lambda\mu^{(2)}}\hspace{1cm}(\text{by $\mu = \mu^{(1)} = \mu^{(2)}$})\\
\in &\ \ZZ_{\geq 0} \hspace{1cm}(\text{by~(\ref{eq:indhy2}) and~(\ref{eq:gamtolam})}).
\end{align*}
Therefore~(\mref{eq:indaim}) is valid.

{\bf Case~2.} $\gamma= \lambda-\epsilon_{\bDo{\lambda}{c'-c}{z}}\notin\pkl$. Then
\begin{align*}
L_{\bDo{\lambda}{c'-c}{z}}\sg{\lambda}{z}=&\
L_{\bDo{\lambda}{c'-(c-1)}{z}}\sg{\lambda}{z} \hspace{1cm} (\text{by~(\ref{eq:indind})})\\
=&\ \sum_{\mu\in\pkl, \, \mu_x = \lambda_x\,\textup{for}\,x\in[z]}b_{\lambda\mu}\sg{\mu}{z}
\hspace{1cm} (\text{by~(\ref{eq:indhy2}) and denote $\mu=\mu^{(2)}$}),
\end{align*}
where $$(-1)^{|\lambda|-|\mu|-1} b_{\lambda\mu}\in\ZZ_{\geq0}.$$
This confirms that~(\mref{eq:indaim}) holds and completes the inductive proof of~(\mref{eq:indaim}).

Finally, taking $c = c'-1$ in~(\mref{eq:indaim}) yields
\begin{equation}
L_{\bdo{\lambda}{z}}\sg{\lambda}{z} =\sum_{\mu\in\pkl, \, \mu_x = \lambda_x\,\textup{for}\,x\in[z]}b_{\lambda\mu}\sg{\mu}{z}, \quad \text{where }\, (-1)^{|\lambda|-|\mu|-1} b_{\lambda\mu}\in\ZZ_{\geq0}.
\mlabel{eq:rform}
\end{equation}
Substituting the above equation into Lemma~\mref{lem:ztoz1} gives
\begin{align*}
\sg{\lambda}{z+1} =&\ \sg{\lambda}{z} - L_{\bdo{\lambda}{z}}\sg{\lambda}{z}
= \sg{\lambda}{z} + \sum_{\mu\in\pkl, \, \mu_x = \lambda_x\,\textup{for}\,x\in[z]}(-1) b_{\lambda\mu}\sg{\mu}{z},
%
%
\end{align*}
which is the required form of~\eqref{eq:ind} by~(\mref{eq:rform}). Here we use the convention $\sg{\lambda}{z} = b_{\lambda \lambda}\sg{\lambda}{z}$ with $b_{\lambda \lambda}:=1$.
This completes the proof.
\end{proof}

\section{Alternating positivity of (weighted) $K$-$k$-Schur functions} \mlabel{sec:comproof}
Our goal in this section is to exploit on the proof of Theorem~\mref{thm:aim}
and further the proof of Theorem~\mref{thm:aim1}. Let us first exhibit an example to illustrate the main idea of the proof of
Theorem~\mref{thm:aim}.

\begin{exam}
Let
\[
k=5, \quad \ell=6, \quad \lambda = (5,4,3,3,2,2)\in{\rm P}_6^5.
\]
Then $\bott_\lambda=3$ and $\lambda \in\hat{\rm P}_6^5$.
We have
\begin{align}
\sg{\lambda}{4}
= &\ \sg{\lambda}{3} - \sg{(5,4,3,3,2,\textcolor{green}{1})}{3} \hspace{1cm}(\text{by Lemma~\ref{lem:ztoz1} and Proposition~\ref{prop:bigd}})\notag\\
=&\ \big( \sg{\lambda}{2} - \sg{(5,4,3,\textcolor{green}{2},2,2)}{2} \big)
- \big(  \sg{(5,4,3,3,2,1)}{2} - \sg{(5,4,3,\textcolor{green}{2},2,1)}{2} \big) \mlabel{eq:finalexam}\\
& \hspace{4cm}(\text{by Lemma~\ref{lem:ztoz1} and Proposition~\ref{prop:bigd}}).\notag
\end{align}
For the first summand $\sg{\lambda}{2}$ in~\eqref{eq:finalexam},
\begin{align}
\sg{\lambda}{2} =&\ \sg{\lambda}{1} - L_{\bdo{\lambda}{1}}\sg{\lambda}{1}
\hspace{1cm} (\text{by Lemma~\ref{lem:ztoz1}})\notag \\
=&\ \sg{\lambda}{1} - \big( (1- L_{\bDo{(5,3,3,3,2,2)}{1+1}{1}}) \sg{(5,\textcolor{green}{3},3,3,2,2)}{1} + L_{\bDo{\lambda}{1+1}{1}}\sg{\lambda}{1} \big)\notag\\
&\ (\text{by $\bdo{\lambda}{1}=2\in[\bott_\lambda]$ and Proposition~\ref{prop:smalld} for the second summand})\notag\\
=&\ \sg{\lambda}{1} - \sg{(5,3,3,3,2,2)}{1} + L_{\bDo{(5,3,3,3,2,2)}{2}{1}} \sg{(5,3,3,3,2,2)}{1} - L_{\bDo{\lambda}{2}{1}}\sg{\lambda}{1} \notag\\
=&\ \sg{\lambda}{1} - \sg{(5,3,3,3,2,2)}{1}+0 - \sg{(5,4,3,\textcolor{green}{2},2,2)}{1} \mlabel{eq:finallam2}\\
&\ (\text{by Proposition~\ref{prop:bigd} for the third and fourth summands}). \notag
\end{align}
Using a similar argument as for $\sg{\lambda}{2}$, the remaining three summands in~\eqref{eq:finalexam} can be computed as follows:
\begin{equation}
\begin{aligned}
\sg{(5,4,3,2,2,2)}{2} =&\ \sg{(5,4,3,2,2,2)}{1} - \sg{(5,\textcolor{green}{3},3,2,2,2)}{1} +0+0, \\
\sg{(5,4,3,3,2,1)}{2} =&\ \sg{(5,4,3,3,2,1)}{1}- \sg{(5,\textcolor{green}{3},3,3,2,1)}{1}
+ \sg{(5,\textcolor{green}{3},3,3,\textcolor{green}{1},1)}{1} -\sg{(5,4,3,\textcolor{green}{2},2,1)}{1}, \\
\sg{(5,4,3,2,2,1)}{2} =&\ \sg{(5,4,3,2,2,1)}{1} - \sg{(5,\textcolor{green}{3},3,2,2,1)}{1}
+\sg{(5,\textcolor{green}{3},3,2,\textcolor{green}{1},1)}{1}+0.
\end{aligned}
\mlabel{eq:finalthreesum}
\end{equation}
Inserting~\eqref{eq:finallam2} and~\eqref{eq:finalthreesum} into~\eqref{eq:finalexam} yields
\begin{align*}
\sg{\lambda}{4}=&\ \big( \sg{\lambda}{1} - \sg{(5,3,3,3,2,2)}{1}- \sg{(5,4,3,2,2,2)}{1} \big) -\big( \sg{(5,4,3,2,2,2)}{1} - \sg{(5,3,3,2,2,2)}{1} \big) \\
&\ -\big( \sg{(5,4,3,3,2,1)}{1}- \sg{(5,3,3,3,2,1)}{1}
+ \sg{(5,3,3,3,1,1)}{1} -\sg{(5,4,3,2,2,1)}{1} \big)\\
&\ +\big( \sg{(5,4,3,2,2,1)}{1} - \sg{(5,3,3,2,2,1)}{1}
+\sg{(5,3,3,2,1,1)}{1} \big)\\
=&\ g_\lambda^{(k)} -g_{(5,\textcolor{green}{3},3,3,2,2)}^{(k)} -2g_{(5,4,3,\textcolor{green}{2},2,2)}^{(k)}
+g_{(5,\textcolor{green}{3},3,\textcolor{green}{2},2,2)}^{(k)} -g_{(5,4,3,3,2,\textcolor{green}{1})}^{(k)}
+g_{(5,\textcolor{green}{3},3,3,2,\textcolor{green}{1})}^{(k)}\\
&\ -g_{(5,\textcolor{green}{3},3,3,\textcolor{green}{1},\textcolor{green}{1})}^{(k)}
+2g_{(5,4,3,\textcolor{green}{2},2,\textcolor{green}{1})}^{(k)}
-g_{(5,\textcolor{green}{3},3,\textcolor{green}{2},2,\textcolor{green}{1})}^{(k)}
+g_{(5,\textcolor{green}{3},3,\textcolor{green}{2},\textcolor{green}{1}
,\textcolor{green}{1})}^{(k)} \hspace{1cm} (\text{by~\eqref{eq:sgtofg}}).
\end{align*}
By direct calculation, the coefficients are alternating positive:
\begin{align*}
& (-1)^{|\lambda|-|\lambda|}\cdot 1 = 1, \quad (-1)^{|\lambda|-|(5,\textcolor{green}{3},3,3,2,2)|}\cdot(-1) = (-1)\cdot(-1) = 1,\\
& (-1)^{|\lambda|-|(5,4,3,\textcolor{green}{2},2,2)|}\cdot(-2) = (-1)\cdot(-2) =2, \quad (-1)^{|\lambda|-|(5,\textcolor{green}{3},3,\textcolor{green}{2},2,2)|}\cdot 1 = (-1)^2\cdot 1 = 1,\\
& (-1)^{|\lambda|-|(5,4,3,3,2,\textcolor{green}{1})|}\cdot(-1) = (-1)\cdot(-1) = 1, \quad (-1)^{|\lambda|-|(5,\textcolor{green}{3},3,3,2,\textcolor{green}{1})|}\cdot 1 = (-1)^2\cdot 1 = 1,\\
& (-1)^{|\lambda|-|(5,\textcolor{green}{3},3,3,
\textcolor{green}{1},\textcolor{green}{1})|}\cdot(-1) = (-1)^3\cdot(-1) = 1, \quad (-1)^{|\lambda|-|(5,4,3,\textcolor{green}{2},2,\textcolor{green}{1})|}\cdot 2 = (-1)^2\cdot 2 = 2,\\
& (-1)^{|\lambda|-|(5,\textcolor{green}{3},3,\textcolor{green}{2},
2,\textcolor{green}{1})|}\cdot(-1) = (-1)^3\cdot(-1) = 1, \quad (-1)^{|\lambda|-|(5,\textcolor{green}{3},3,\textcolor{green}{2},
\textcolor{green}{1},
\textcolor{green}{1})|}\cdot 1 = (-1)^4\cdot 1 = 1.
\end{align*}
\hspace*{\fill}$\square$
\end{exam}

We are in a position to give a detailed proof of Theorem~\mref{thm:aim}.

\begin{proof}[proof of Theorem~\mref{thm:aim}]
The proof proceeds by induction on $z\in[0,\bott_\lambda]$.
The initial step of $z = 0$ is from~(\mref{eq:sgtofg}), which shows that $\sg{\lambda}{1} = g_\lambda^{(k)}$.
Consider the inductive step of $1\leq z\leq \bott_\lambda$, that is, $z\in[\bott_\lambda]$.
By~Proposition~\ref{prop:ind},
\begin{equation}
\sg{\lambda}{z+1} = \sum_{\mu'\in\pkl, \, \mu'_x = \lambda_x\,\textup{for}\,x\in[z]}b_{\lambda\mu'}\sg{\mu'}{z},
\quad \text{where }\, (-1)^{|\lambda|-|\mu'|} b_{\lambda\mu'}\in\ZZ_{\geq0}.
\mlabel{eq:finalini}
\end{equation}
Consider each $\mu'\in\pkl$ in the right hand side of~(\mref{eq:finalini}).
Notice that the condition $\mu'_x = \lambda_x$ for $x\in[z]$ implies that the root ideals $\Delta^k(\mu')$ and $\Delta^k(\lambda)$
are the same in rows $1, \ldots, z$.
By $z\in[\bott_\lambda]$,  we have
\begin{equation*}
z\in[\bott_{\mu'}], \quad\text{that is},\ \,z-1\in[0, \bott_{\mu'}-1]\subsetneq [0, \bott_{\mu'}]
\mlabel{eq:zz1}
\end{equation*}
and
$$
\mu'_x = \lambda_x > \lambda_{x+1} = \mu'_{x+1}\,\text{ for each }\, x\in[z-1].
$$
Hence
\[
z-1\neq \bott_{\mu'}, \quad \mu'_1>\cdots>\mu'_{z-1}>\mu'_{z}.
\]
Now by the inductive hypothesis, each $\sg{\mu'}{z} $ in the right hand side of~(\mref{eq:finalini}) satisfies
\begin{equation}
\sg{\mu'}{z} =  \sum_{\mu\in{\rm P}}b_{\mu'\mu}g_\mu^{(k)}, \quad \text{where }\, (-1)^{|\mu'|-|\mu|}b_{\mu'\mu}\in\ZZ_{\geq0}.
\mlabel{eq:finalini1}
\end{equation}
Inserting the above equation into~(\mref{eq:finalini}),
\begin{align*}
\sg{\lambda}{z+1} &=\ \sum_{\mu'\in\pkl, \, \mu'_x = \lambda_x\,\textup{for}\,x\in[z]}b_{\lambda\mu'}\Bigg( \sum_{\mu\in{\rm P}}b_{\mu'\mu}g_\mu^{(k)} \Bigg) \\
&=\ \sum_{\mu\in{\rm P}}b_{\lambda\mu'}b_{\mu'\mu}g_\mu^{(k)}\\
&=:\ \sum_{\mu\in{\rm P}} b_{\lambda\mu}g_\mu^{(k)}.
\end{align*}
Then~(\mref{eq:finalonlyneed}) holds by
\begin{align*}
(-1)^{|\lambda|-|\mu|}b_{\lambda\mu} = &\ (-1)^{|\lambda|-|\mu'|+|\mu'|-|\mu|}b_{\lambda\mu'}b_{\mu'\mu}\\
=&\ (-1)^{|\lambda|-|\mu'|}b_{\lambda\mu'} (-1)^{|\mu'|-|\mu|}b_{\mu'\mu} \\
\in&\ \ZZ_{\geq0} \hspace{1cm} (\text{by~\eqref{eq:finalini} and~\eqref{eq:finalini1}}).
\end{align*}
This completes the proof.
\end{proof}

The following simple result will be used.

\begin{lemma}
For $\lambda\in\pkl$ and $x\in[\ell]$, $x\in[\bott_\lambda]$ if and only if
$
k-\lambda_x+x<\ell
$.
\mlabel{lem:bl}
\end{lemma}

\begin{proof}
By Definition~\mref{def:com}, $x\in[\bott_\lambda]$ if and only if row $x$ has at least one root in the root ideal $\dkl$. Besides, (\mref{eq:dkl}) implies that row $x$ has at least one root in $\dkl$ if and only if $
k-\lambda_x+x<\ell
$. Hence the lemma holds.
%
\end{proof}

We are ready to give the proof of Theorem~\ref{thm:aim1}.

\begin{proof}[Proof of Theorem~\ref{thm:aim1}]
First,
\begin{align*}
\lambda\in\spkl \quad \Leftrightarrow & \quad \lambda_{x-1}>\lambda_x, \,\text{ if }\,k-\lambda_x+x<\ell \hspace{1cm} (\text{by~\eqref{eq:spkl}}) \\
\Leftrightarrow & \quad \lambda_{x-1}>\lambda_{x},\,\text{ if }\, x\in[\bott_\lambda]\hspace{1cm} (\text{by Lemma~\ref{lem:bl}}) \\
\Leftrightarrow & \quad \lambda_0> \lambda_1>\cdots>\lambda_{\bott_\lambda} \\
\Leftrightarrow & \quad \lambda_1>\cdots>\lambda_{\bott_\lambda} \hspace{1cm} (\text{by $\lambda_0:=\infty$}).
\end{align*}
Further, applying Theorem~\mref{thm:aim} with $z = \bott_\lambda$ yields
\begin{equation*}
\fg{\lambda}{k} \overset{\eqref{eq:sgtofg}}{=} \sg{\lambda}{\bott_\lambda+1}
= \sum_{\mu\in{\rm P}}b_{\lambda\mu}g_\mu^{(k)}, \quad
\text{where }\, (-1)^{|\lambda|-|\mu|}b_{\lambda\mu}\in\ZZ_{\geq0}.
\end{equation*}
This completes the proof.
\end{proof}

\smallskip

\noindent
{\bf Acknowledgements}: This work is supported by the Natural Science Foundation of Gansu Province (25JRRA644), Innovative Fundamental Research Group Project of Gansu Province (23JRRA684) and Longyuan Young Talents of Gansu Province.

\noindent
{\bf Declaration of interests.} The authors have no conflicts of interest to disclose.

\noindent
{\bf Data availability.} Data sharing is not applicable as no new data were created or analyzed.

\end{document}